\documentclass[11pt]{article}

\usepackage{paralist}
\usepackage[margin=3.0cm]{geometry} 
\usepackage{graphicx} 

\usepackage{amssymb, amsfonts, amsmath, amscd, amsthm} 
\usepackage{mathtools} 
\usepackage{titlesec} 
\usepackage{csquotes} 
\usepackage[shortlabels]{enumitem} 
\usepackage{tikz-cd} 
\usepackage{wrapfig} 
\usepackage{epigraph} 
\usepackage{hyperref} 
\usepackage{todonotes} 
\usepackage{comment}
\usepackage{mathrsfs} 

\hypersetup{
     colorlinks   = true,
     citecolor    = green
}

\numberwithin{equation}{section}

\titleformat*{\section}{\Large \scshape\center}
\titleformat*{\subsection}{\fontsize{14}{14} \sffamily}

\theoremstyle{plain}
\newtheorem{theorem}{Theorem}[section]
\newtheorem*{theorem*}{Theorem}
\newtheorem{introthm}{Theorem}

\newtheorem{lemma}[theorem]{Lemma}
\newtheorem{proposition}[theorem]{Proposition}
\newtheorem{corollary}[theorem]{Corollary}

\theoremstyle{definition}
\newtheorem{definition}[theorem]{Definition}
\newtheorem*{definition*}{Definition}
\newtheorem{example}[theorem]{Example}

\theoremstyle{remark}
\newtheorem*{remark}{Remark}

\newcommand{\aff}{\mathrm{Aff}}

\newcommand{\intA}[3]{\int_{\aff} #1 \, \frac{d#2 \, d#3}{#3}}

\DeclareMathOperator{\tr}{tr}

\newcommand{\Addresses}{{
  \bigskip
  \footnotesize

  \textsc{Department of Mathematical Sciences, Norwegian University of Science and Technology,\\ 7491 Trondheim, Norway.}\par\nopagebreak
  \textit{E-mail addresses}: 
  \texttt{eirik.berge@ntnu.no}, \texttt{stine.m.berge@ntnu.no}, \texttt{franz.luef@ntnu.no},
  
  \hspace{2.7cm}
  and
  \texttt{eirik.skrettingland@ntnu.no}

}}

\begin{document}
\pagenumbering{gobble}
\title{\huge{Affine Quantum Harmonic Analysis}}
\author{Eirik Berge, Stine M. Berge, Franz Luef, Eirik Skrettingland}
\date{}
\maketitle
\pagenumbering{arabic}
\begin{abstract}
We develop a quantum harmonic analysis framework for the affine group. This encapsulates several examples in the literature such as 
affine localization operators, covariant integral quantizations, and affine quadratic time-frequency representations. In the process, we develop a notion of admissibility for operators and extend well known results to the operator setting. A major theme of the paper is the interaction between operator convolutions, affine Weyl quantization, and admissibility. 
\end{abstract}

\section{Introduction}
\label{sec: Introduction}
The affine group and the Heisenberg group play prominent roles in wavelet theory and Gabor analysis, respectively. As is well-known, the representation theory of the Heisenberg group is intrinsically linked to quantization on phase space $\mathbb{R}^{2n}$. Similarly, the relation between quantization schemes on the affine group and its representation theory has received some attention and several schemes have been proposed, e.g.\ \cite{gayral2007fourier,berge2019affine,gazeau2016covariant}. However, there are still many open questions awaiting a definite answer in the case of the affine group. 

As has been shown by two of the authors in \cite{Luef2018}, the theory of \textit{quantum harmonic analysis on phase space} introduced by Werner \cite{werner1984} provides a coherent framework for many aspects of quantization and Gabor analysis associated with the Heisenberg group. Based on this connection, advances in the understanding of time-frequency analysis have been made \cite{lusk19-1,lusk20,luef2021jfa}. In this paper we aim to develop a variant of Werner's quantum harmonic analysis in \cite{werner1984} for time-scale analysis. This is based on unitary representations of the affine group in a similar way to the Schrödinger representation of the Heisenberg group being used in Werner's framework. We will refer to this theory on the affine group as \textit{affine quantum harmonic analysis}.  
 
\subsubsection*{Affine Operator Convolutions}
In Werner's quantum harmonic analysis on phase space, a crucial component is extending convolutions to operators. Recall that the affine group $\aff$ has the underlying set $\mathbb{R} \times \mathbb{R}_{+}$ and group operation modeling composition of affine transformations. A key feature of this group is that the left Haar measure $a^{-2} dx \, da$ and the right Haar measure $a^{-1} dx \, da$ are not equal, making the group \textit{non-unimodular}. Both measures play a role in affine quantum harmonic analysis, making the theory more involved than the case of the Heisenberg group. In addition to the standard function (right-)convolution on the affine group
\begin{equation*}
    f\ast_\aff g(x,a) \coloneqq \int_{\aff} f(y,b)g(\left(x,a\right)\cdot(y,b)^{-1})\, \frac{dy \, db}{b},
\end{equation*}
we introduce the following \textit{operator convolutions} for operators on $L^2(\mathbb{R}_+)\coloneqq L^{2}(\mathbb{R}_{+}, r^{-1}\, dr)$ in Section \ref{sec: Affine Operator Convolutions}: 
\begin{itemize}
    \item Let $f \in L_{r}^{1}(\aff) \coloneqq L^{1}(\aff, a^{-1}dx \, da)$ and let $S$ be a trace-class operator on $L^2(\mathbb{R}_+)$. We define the \textit{convolution} $f\star_\aff S$ between $f$ and $S$ to be the operator on $L^{2}(\mathbb{R}_{+})$ given by 
\begin{equation*}
f\star_\aff S \coloneqq \int_{\aff} f(x,a) U(-x,a)^*SU(-x,a)\, \frac{dx \, da}{a},
\end{equation*}
where $U$ is the unitary representation of $\aff$ on $L^{2}(\mathbb{R}_{+})$ given by \[U(x,a)\psi(r) \coloneqq e^{2\pi i x r}\psi(ar).\]
    \item Let $S$ be a trace-class operator and let $T$ be a bounded operator on $L^{2}(\mathbb{R}_{+})$. Then we define the \textit{convolution} $S \star_{\aff} T$ between $S$ and $T$ to be the function on $\textrm{Aff}$ given by
\begin{equation*}
    S\star_\aff T(x,a) \coloneqq \tr(SU(-x,a)^*TU(-x,a)).
\end{equation*}
\end{itemize}
The three convolutions are compatible in the following sense:
Let $f,g\in L^1_r(\aff)$ and denote by $S$ a trace-class operator and by $T$ a bounded operator, both on $L^{2}(\mathbb{R}_{+})$. Then
\begin{align*}
    (f\star_\aff S)\star_\aff T &= f*_\aff(S\star_\aff T),\\
    f\star_\aff (g\star_\aff S) &= (f*_\aff g)\star_\aff S.
\end{align*}

\subsubsection*{Interplay Between Affine Weyl Quantization and Convolutions}

Integral to the theory in this paper is the affine Wigner distribution and the associated affine Weyl quantization.
The \textit{affine (cross-)Wigner distribution} $W_{\textrm{Aff}}^{\psi,\phi}$ of $\phi, \psi \in L^{2}(\mathbb{R}_{+})$ is the function on $\mathrm{Aff}$ given by 
\begin{equation}
\label{eq: affine_wigner_introduction}
W_{\mathrm{Aff}}^{\psi,\phi}(x,a)  = \int_{-\infty}^{\infty}\psi\left(\frac{aue^u}{e^u - 1}\right)\overline{\phi\left(\frac{au}{e^{u} - 1}\right)}e^{-2\pi i x u} \, du.
\end{equation} 
Although at first glance the definition \eqref{eq: affine_wigner_introduction} might look unnatural, it can be motivated through the representation theory of the affine group as illustrated in \cite{ali2000}. We will elaborate on this viewpoint in Section \ref{sec: From the Viewpoint of Representation Theory}.
One defines the \textit{affine Weyl quantization} of $f \in L_{r}^{2}(\aff)\coloneqq L^2(\aff, a^{-1}dx\,da)$ as the operator $A_{f}$ given by 
\begin{equation*}
    \left\langle A_{f}\phi,\psi \right\rangle_{L^{2}(\mathbb{R}_{+})} = \left\langle f,W_{\mathrm{Aff}}^{\psi,\phi}\right\rangle_{L_{r}^{2}(\mathrm{Aff})}, \qquad
\text{for all }\phi, \psi \in L^{2}(\mathbb{R}_{+}). 
\end{equation*}
We will explore the intimate relation between the convolutions and the affine Weyl quantization. The following theorem, being a combination of Proposition \ref{function-operator-convolution-quantization} and Proposition \ref{prop:convolutionasweyl}, highlights this relation.
\begin{introthm}
    Let $f,g\in L^2_r(\aff)$, where $g$ is additionally in $L^1_r(\aff)$ and square integrable with respect to the left Haar measure. Then
    \begin{align*}
        g\star_\aff A_f &= A_{g*_\aff f},\\
        A_g\star_\aff A_f &= f*_\aff \check{g},
    \end{align*}
    where $\check{g}(x,a) \coloneqq g((x,a)^{-1})$.
\end{introthm}
We will exploit the previous theorem to define the affine Weyl quantization of tempered distributions in Section \ref{sec: Affine Quantization of Coordinate Functions}. To do this rigorously, we will utilize a Schwartz space $\mathscr{S}(\aff)$ on the affine group introduced in \cite{berge2019affine}. An important example we prove in Theorem \ref{quantization_of_coordinates} is the affine Weyl quantization of the coordinate functions:
\begin{introthm}
    Let $f_{x}(x,a) \coloneqq x$ and $f_{a}(x,a) \coloneqq a$ be the coordinate functions on $\aff$. 
    The affine Weyl quantizations $A_{f_{x}}$ and $A_{f_{a}}$ satisfy the commutation relation \[[A_{f_{x}},A_{f_{a}}] =  \frac{1}{2\pi i}A_{f_{a}}.\]
    This is, up to re-normalization, precisely the infinitesimal structure of the affine group.
\end{introthm}
We define \textit{affine parity operator} $P_{\aff}$ as
\[P_{\aff} = A_{\delta_{(0,1)}},\] 
where $\delta_{(0,1)}$ denotes the Dirac distribution at the identity element $(0, 1) \in \aff$. The following result, which will be rigorously stated in Section \ref{sec: Operator Convolution for Tempered Distributions}, builds on these definitions.
\begin{introthm}
    The affine Weyl quantization $A_{g}$ of $g \in \mathscr{S}(\mathrm{Aff})$ can be written as \[A_{g} = g \star_{\mathrm{Aff}} P_{\mathrm{Aff}}.\]
    Moreover, for $\phi, \psi$ such that $\phi(e^x), \psi(e^{x}) \in \mathscr{S}(\mathbb{R})$, the affine Weyl symbol $W_{\mathrm{Aff}}^{\psi,\phi}$ of the rank-one operator $\psi \otimes \phi$ can be written as
\[W_{\mathrm{Aff}}^{\psi,\phi} = (\psi \otimes \phi) \star_{\mathrm{Aff}} P_{\mathrm{Aff}}.\]
\end{introthm}

\subsubsection*{Operator Admissibility}

One of the key features of representations of non-unimodular groups is the concept of admissibility. Recall that the \textit{Duflo-Moore operator} $\mathcal{D}^{-1}$ corresponding to the representation $U$ is the densely defined positive operator on $L^{2}(\mathbb{R}_{+})$ given by $\mathcal{D}^{-1}\psi(r) = r^{-1/2}\psi(r)$. We will often use that $\mathcal{D}^{-1}$ has a densely defined inverse given by $\mathcal{D}\psi(r) = r^{1/2} \psi(r)$. A function $\psi$ is said to be an \textit{admissible wavelet} if $\psi\in \mathrm{dom}(\mathcal{D}^{-1})$. It is well known \cite{duflo1976} that admissible wavelets satisfy the orthogonality relation
\begin{equation}
\label{eq: addmissible function orthogonality}
    \int_\aff |\langle \phi,U(-x,a)^*\psi \rangle_{L^2(\mathbb{R}_+)}|^2 \, \frac{dx \, da}{a}=\|\phi\|^2_{L^2(\mathbb{R}_+)}  \|\mathcal{D}^{-1}\psi\|_{L^2(\mathbb{R}_+)}^{2}.
\end{equation} 
We extend the definition of admissibility to operators as follows:
 
\begin{definition*}
Let $S$ be a non-zero bounded operator on $L^{2}(\mathbb{R}_{+})$ that maps $\mathrm{dom}(\mathcal{D})$ into $\mathrm{dom}(\mathcal{D}^{-1})$. We say that $S$ is \textit{admissible} if the composition $\mathcal{D}^{-1}S\mathcal{D}^{-1}$ is bounded on  $\mathrm{dom}(\mathcal{D}^{-1})$ and extends to a trace-class operator $\mathcal{D}^{-1}S\mathcal{D}^{-1}$ on $L^{2}(\mathbb{R}_{+})$. 
\end{definition*}

Note that the rank-one operator $S=\psi\otimes \psi$ for $\psi \in L^2(\mathbb{R}_+)$ is admissible precisely when $\psi$ is an admissible wavelet. In Section \ref{sec: Laguerre Connection} we show that a large class of admissible operators can be constructed from Laguerre bases. The following result, which we prove in Corollary \ref{cor:admissible_condition}, is motivated by \cite[Lemma 3.1]{werner1984} and extends \eqref{eq: addmissible function orthogonality} to the operator setting.

\begin{introthm}
Let $S$ be an admissible operator on $L^{2}(\mathbb{R}_{+})$. For any trace-class operator $T$ on $L^{2}(\mathbb{R}_{+})$, we have that $T \star_{\aff} S \in L_{r}^{1}(\aff)$ with 
	\begin{equation*} 
		\int_\aff T\star_{\aff} S(x,a) \, \frac{dx\, da}{a} = \tr(T)\tr(\mathcal{D}^{-1}S\mathcal{D}^{-1}).
	\end{equation*}
\end{introthm}

Determining whether an operator is admissible or not can be a daunting task. We managed in Corollary \ref{corr:admissibility_check} to find an elegant characterization in terms of operator convolutions of admissible operators that are additionally positive trace-class operators.

\begin{introthm}
Let $S$ be a non-zero, positive trace-class operator. Then $S$ is admissible if and only if $S\star_\aff S\in L_{r}^{1}(\aff)$.    
\end{introthm}

The following result is derived in Section \ref{sec: Admissibility as a Measure of Non-Unimodularity} and uses the affine Weyl quantization to show that admissibility is an operator manifestation of the non-unimodularity of the affine group.

\begin{introthm}
\hfill
\begin{itemize}
    \item  Let $f \in L^1_r(\aff)$ be such that $A_f$ is a trace-class operator on $L^2(\mathbb{R}_+)$. Then 
	 \begin{equation*}
     \tr(A_f)=\int_\aff f(x,a)\,\frac{dx\, da}{a}.
 \end{equation*}
 \item Let $g\in L_{l}^{1}(\aff) \coloneqq L^1(\aff, a^{-2}dx\,da)$ be such that $A_g$ is an admissible Hilbert-Schmidt operator. Then 
	\begin{equation*}
		\tr\left(\mathcal{D}^{-1}A_g\mathcal{D}^{-1}\right) = \int_\aff g(x,a) \, \frac{dx \, da}{a^2}.
	\end{equation*}
\end{itemize}
\end{introthm}

\subsubsection*{Relationship with Fourier Transforms}
For completeness, we will also investigate how notions of Fourier transforms on the affine group fit into the theory, and use known results from abstract harmonic analysis to explore the relationship between affine Weyl quantization and affine Fourier transforms. Recall that the \textit{integrated representation} $U(f)$ of $f \in L_{l}^{1}(\aff)$ is the operator on $L^{2}(\mathbb{R}_{+})$ given by \[U(f)\psi \coloneqq \int_{\aff}f(x,a) U(x,a)\psi \, \frac{dx \, da}{a^2}, \qquad \psi \in L^{2}(\mathbb{R}_{+}).\]
We define the following operator Fourier transform in the affine setting.
\begin{definition*}
The \textit{affine Fourier-Wigner transform} is the isometry $\mathcal{F}_{W}$ sending a Hilbert-Schmidt operator on $L^{2}(\mathbb{R}_{+})$ to a function in $L_{r}^{2}(\aff)$ such that \[\mathcal{F}_W^{-1}(f) = U(\check{f}) \circ \mathcal{D}, \qquad f \in \mathrm{Im}(\mathcal{F}_{W}) \cap L_{r}^{1}(\aff).\]
\end{definition*}

The following result is proved in Proposition \ref{prop:equvalence_right} and provides a connection between the affine Fourier-Wigner transform and admissibility.

\begin{introthm}
    Let $A$ be a trace-class operator on $L^{2}(\mathbb{R}_{+})$. The following are equivalent:
    \begin{enumerate}[1)]
        \item $\mathcal{F}_W(A\mathcal{D}^{-1})\in L^2_r(\aff)$.
        \item $A\mathcal{D}^{-1}$ extends from $\mathrm{dom}(\mathcal{D}^{-1})$ to a Hilbert-Schmidt operator on $L^2(\mathbb{R}_+)$.
        \item $A^*A$ is admissible.
    \end{enumerate}
\end{introthm}
Another Fourier transform of interest is the (modified) \textit{Fourier-Kirillov transform} on the affine group $\mathcal{F}_{\mathrm{KO}}$ given by
\begin{align*}
(\mathcal{F}_{\mathrm{KO}}f)(x,a)&=\sqrt{a}\int_{\mathbb{R}^{2}}f\left(\frac{v}{\lambda(-u)},e^u\right)e^{-2\pi i (xu+av)} \, \frac{du \, dv}{\sqrt{\lambda(-u)}}, \qquad f \in \mathrm{Im}(\mathcal{F}_{W}).
\end{align*}
As in quantum harmonic analysis on phase space, we have that the affine Weyl quantization is the composition of these Fourier transforms, see Proposition \ref{commutative_diagram_result}. In the affine setting we have in general that
\begin{equation*}
\mathcal{F}_W(f\star_\aff S) \neq \mathcal{F}_{KO}(f)\mathcal{F}_W(S), \qquad \mathcal{F}_{KO}(S\star_\aff T) \neq  \mathcal{F}_W(S)\mathcal{F}_W(T).
\end{equation*}
This contrasts the analogous result in Werner's original quantum harmonic analysis, see \eqref{eq: conv_fourier}. In spite of this, not all properties typically associated with the Fourier transform are lost: In Section \ref{sec: Plancherel Measure and Affine KLM conditions} we prove a quantum Bochner theorem in the affine setting. 

\subsubsection*{Main Applications}
In Section \ref{sec: Examples} we show  that affine quantum harmonic analysis provides a conceptual framework for the study of \textit{covariant integral quantizations} and a version of the \textit{Cohen class} for the affine group. In addition, we show in Section \ref{sec: Representing Localization Operators Through Affine Convolution} that if $S$ is a rank-one operator, then the study of operators $f\star_{\aff} S$ for functions $f$ on $\aff$ reduces to the study of time-scale localization operators \cite{daubechies1988localization}.

We have seen that affine Weyl quantization is given by $f\mapsto f\star_\aff P_\aff$ for $f \in \mathscr{S}(\aff)$. Inspired by this, we consider a whole class of quantization procedures: For any suitably nice operator $S$ on $L^2(\mathbb{R}_+)$ we define a quantization procedure $\Gamma_S$ for functions $f$ on $\aff$ by \[\Gamma_S(f) \coloneqq f\star_\aff S.\] This class of quantization procedures coincides with the  \textit{covariant integral quantizations} studied by Gazeau and his collaborators motivated by applications in physics, see e.g.\ \cite{gazeau2016covariant,gazeau20192d,gazeau2020signal}. Our results on affine quantum harmonic analysis are therefore also results on covariant integral quantizations. In particular, the abstract notion of admissibility of an operator $S$ implies that $\Gamma_S$ satisfies the simple property \[\Gamma_S(1)=c\cdot I_{L^2(\mathbb{R}_{+})},\] where $c$ is some constant, $I_{L^2(\mathbb{R}_+)}$ is the identity operator on $L^2(\mathbb{R}_+)$, and $1(x,a)=1$ for all $(x,a)\in \aff$.  

As the name suggests, covariant integral quantizations $\Gamma_S$ satisfy a \textit{covariance} property, namely 
\begin{equation*}
    U(-x,a)^*\Gamma_{S}(f)U(-x,a)=\Gamma_S(R_{(x,a)^{-1}}f),
\end{equation*}
where $R$ denotes right translations of functions on $\aff$. In Theorem \ref{thm:kiukas} we point out that, by a known result on covariant positive operator valued measures \cite{kiukas2006,cassinelli2003}, this covariance assumption together with other mild assumptions completely characterize the covariant integral quantizations. We have also seen that the affine cross-Wigner distribution is given for sufficiently nice $\psi, \phi$ by $W_{\mathrm{Aff}}^{\psi,\phi} = (\psi \otimes \phi) \star_{\mathrm{Aff}} P_{\mathrm{Aff}}$. Inspired by this and the description in \cite{lusk19-1} of the Cohen class of time-frequency distributions on $\mathbb{R}^{2n}$, we make the following definition. 

\begin{definition*}
A bilinear map $Q:L^2(\mathbb{R}_+)\times L^2(\mathbb{R}_+)\to L^\infty(\aff)$ belongs to the \textit{affine Cohen class} if $Q = Q_S$ for some operator $S$ on $L^2(\mathbb{R}_+)$, where
\begin{equation*}
	Q_S(\psi,\phi)(x,a) \coloneqq (\psi \otimes \phi) \star_\aff S(x,a)
	= \langle SU(-x,a)\psi,U(-x,a)\phi \rangle_{L^2(\mathbb{R}_+)}.
\end{equation*}
\end{definition*}
We will show how properties of $S$ (such as admissibility) influence properties of $Q_S$, and obtain an abstract characterization of the affine Cohen class. Readers familiar with the Cohen class on $\mathbb{R}^{2n}$ \cite{Cohen:1966} will know that it is defined in terms of convolutions with the Wigner function. In the affine setting, we have the analogous result \[Q_{A_f}(\psi,\phi)= W_{\textrm{Aff}}^{\psi, \phi} \ast_\aff \check{f}.\] 
As we explain in Proposition \ref{prop:affineclass},  the \textit{affine class of quadratic time-frequency representations} from \cite{papandreou1998quadratic} may be identified with a subclass of the affine Cohen class. 
\subsubsection*{Structure of the Paper}
In Section \ref{sec: Background Material} we recall necessary background material for completeness. In particular, Section \ref{sec: Quantum Harmonic Analysis on the Heisenberg Group} should serve as a brief reference for quantum harmonic analysis on phase space. We define affine operator convolution in Section \ref{sec: Definitions_and_Basic_Properties} and show the relationship with the affine Weyl quantization in Section \ref{sec: Relationship With Quantization}. The affine parity operator will be introduced in Section \ref{sec: Affine Grossmann-Royer Operator}, and its relationship to affine Weyl quantization will be explored in Section \ref{sec: Operator Convolution for Tempered Distributions}. We have dedicated the entirety of Section \ref{sec: Operator Admissibility} to operator admissibility. Section \ref{sec: From the Viewpoint of Representation Theory} discusses affine Weyl quantization from the viewpoint of representation theory. In particular, in Section \ref{sec: Plancherel Measure and Affine KLM conditions} we derive a Bochner type theorem for our setting. In Section \ref{sec: Representing Localization Operators Through Affine Convolution} and Section \ref{sec: Other Covariant Integral Quantizations} we relate our work to time-scale localization operators and covariant integral quantizations, respectively. Finally, in Section \ref{sec: Affine Cohen Class Distributions} we define the affine Cohen class and derive some basic properties.

\section{Preliminaries}
\label{sec: Background Material}

\textbf{Notation:} Given a Hilbert space $\mathcal{H}$ we let $\mathcal{L}(\mathcal{H})$ denote the bounded operators on $\mathcal{H}$. The notation $\mathcal{S}_{p}(\mathcal{H})$ for $1 \leq p < \infty$ will be used for the \textit{Schatten-p class operators} on $\mathcal{H}$. We remark that $\mathcal{S}_{1}(\mathcal{H})$and $\mathcal{S}_{2}(\mathcal{H})$ are respectively the trace-class operators and the Hilbert-Schmidt operators on $\mathcal{H}$. The space $\mathcal{S}_{\infty}(\mathcal{H})$ is by definition $\mathcal{L}(\mathcal{H})$ for duality reasons. When the Hilbert space in question is $\mathcal{H} = L^{2}(\mathbb{R}_{+}) \coloneqq L^{2}(\mathbb{R}_{+}, r^{-1} \, dr)$, we will simplify the notation to $\mathcal{S}_{p} \coloneqq \mathcal{S}_{p}(L^{2}(\mathbb{R}_{+}))$ for readability. We will denote by $\mathscr{S}(\mathbb{R}^{n})$ the space of Schwartz functions on $\mathbb{R}^{n}$. For a function $f$ on a group $G$, the function $\check{f}$ is defined by $\check{f}(g)=f(g^{-1})$ for all $g \in G$.

\subsection{Basic Constructions on the Affine Group}
\label{sec: Basic Constructions on the affine Group}

We begin by giving a brief introduction to the affine group and relevant constructions on it. The \textit{(reduced) affine group} $(\mathrm{Aff}, \cdot_{\mathrm{Aff}})$ is the Lie group whose underlying set is the upper half plane $\mathrm{Aff} \coloneqq \mathbb{R} \times \mathbb{R}_{+} \coloneqq \mathbb{R} \times (0, \infty)$, while the group operation is given by \[(x,a) \cdot_{\mathrm{Aff}} (y,b) \coloneqq (ay + x,ab), \qquad (x,a), (y,b) \in \mathrm{Aff}.\]
We will often neglect the subscript in the group operation to improve readability. Moreover, we use the notation $L_{(x,a)}$ and $R_{(x,a)}$ to denote respectively the \textit{left-translation} and  \textit{right-translation} by $(x,a) \in \mathrm{Aff}$, acting on a function $f:\mathrm{Aff} \to \mathbb{C}$ by
\[\left(L_{(x,a)}f\right)(y,b) \coloneqq f((x,a)^{-1} \cdot_{\mathrm{Aff}} (y,b)), \qquad \left(R_{(x,a)}f\right)(y,b) \coloneqq f((y,b) \cdot_{\mathrm{Aff}} (x,a)).\]
Recall that the \textit{translation operator} $T_{x}$ and the \textit{dilation operator} $D_{a}$ are respectively given by 
\begin{equation}
    \label{Eq:translation_operator_and_dilation_operator}
    T_{x}f(y) \coloneqq f(y - x), \qquad D_{a}f(y) \coloneqq \frac{1}{\sqrt{a}}f\left(\frac{y}{a}\right), \qquad x,y \in \mathbb{R}, \, a \in \mathbb{R}_{+}.
\end{equation}
The following computation motivates the group operation on the affine group: 
\[(T_x D_a)(T_y D_b) = T_x T_{ay} D_a D_b = T_{x + ay}D_{ab}.\] \par 
We can represent the affine group $\mathrm{Aff}$ and its Lie algebra $\mathfrak{aff}$ in matrix form 
\[\mathrm{Aff} = \left\{\begin{pmatrix} a & x \\ 0 & 1\end{pmatrix}\Big| \, a> 0, \, x \in \mathbb{R}\right\}, \quad \mathfrak{aff} = \left\{\begin{pmatrix} u & v \\ 0 & 0\end{pmatrix}\Big| \, u,v \in \mathbb{R}\right\}.\]
The Lie algebra structure of $\mathfrak{aff}$ is completely determined by 
\begin{equation}
\label{eq:affine_lie_algebra_structure}
    \left[\begin{pmatrix} 1 & 0 \\ 0 & 0\end{pmatrix}, \begin{pmatrix} 0 & 1 \\ 0 & 0\end{pmatrix}\right] = \begin{pmatrix} 0 & 1 \\ 0 & 0\end{pmatrix}.
\end{equation}
An important feature of the affine group is that it is non-unimodular; the left and right Haar measures are respectively given by
\[\mu_{L}(x,a) = \frac{dx \, da}{a^2}, \qquad \mu_{R}(x,a) = \frac{dx \, da}{a}.\]
As such, the modular function on the affine group is given by $\Delta(x,a) = a^{-1}$. The affine group is exponential, meaning that the exponential map $\exp : \mathfrak{aff} \to \mathrm{Aff}$ given by
\[\exp \begin{pmatrix} u & v \\ 0 & 0 \end{pmatrix} = \begin{pmatrix} e^{u} & \frac{v(e^{u} - 1)}{u} \\ 0 & 1\end{pmatrix}\]
is a global diffeomorphism. Hence we can write the left and right Haar measures in exponential coordinates by the formulas 
\begin{equation}
\label{eq:lambda_function}
    \mu_{L}(x,a) = \frac{du \, dv}{\lambda(u)}, \qquad \mu_{R}(x,a) = \frac{du \, dv}{\lambda(-u)}, \qquad \lambda(u) \coloneqq \frac{ue^{u}}{e^{u} - 1}.
\end{equation}
Throughout the paper, we will heavily use the spaces $L_{l}^{p}(\mathrm{Aff}) \coloneqq L^{p}(\mathrm{Aff}, \mu_{L})$ and $L_{r}^{p}(\mathrm{Aff}) \coloneqq L^{p}(\mathrm{Aff}, \mu_{R})$ for $1 \leq p \leq \infty$. 

\subsection{Quantum Harmonic Analysis on the Heisenberg Group}
\label{sec: Quantum Harmonic Analysis on the Heisenberg Group}
Before delving into quantum harmonic analysis on the affine group, it is advantageous to review the Heisenberg setting, originally introduced by Werner \cite{werner1984}. There are three primary constructions that appear: 
\begin{inparaenum}[(a)] 
    \item A quantization scheme,
    \item an integrated representation, and
    \item a way to define convolution that incorporates operators.
\end{inparaenum}
We give a brief overview of these three constructions and refer the reader to \cite{werner1984,grochenig2013foundations,Luef2018} for more details. 

\subsubsection{Weyl Quantization}
\label{sec: Weyl Quantization}
The \textit{cross-Wigner distribution} of $\phi,\psi \in L^2(\mathbb{R}^{n})$ is given by 
\begin{equation*}
    W(\phi,\psi)(x,\omega) \coloneqq \int_{\mathbb{R}^n}\phi\left(x + \frac{t}{2}\right)\overline{\psi\left(x - \frac{t}{2}\right)}e^{-2\pi i \omega t} \, dt, \quad (x,\omega) \in \mathbb{R}^{2n}.
\end{equation*}
When $\phi = \psi$ we refer to $W\phi \coloneqq W(\phi,\phi)$ as the \textit{Wigner distribution} of $\phi \in L^{2}(\mathbb{R}^{n})$. The cross-Wigner distribution satisfies the orthogonality relation
\begin{equation*}
    \langle W(\phi_1,\psi_1), W(\phi_2,\psi_2) \rangle_{L^{2}(\mathbb{R}^{2n})} = \langle \phi_1, \phi_2 \rangle_{L^{2}(\mathbb{R}^n)}\overline{\langle \psi_1, \psi_2 \rangle}_{L^{2}(\mathbb{R}^n)}, \quad \phi_1, \phi_2, \psi_1, \psi_2 \in L^{2}(\mathbb{R}^{n}).
\end{equation*}
Moreover, the Wigner distribution satisfies the marginal properties 
\begin{equation*}
    \int_{\mathbb{R}^{n}}W\phi(x,\omega) \, d\omega = |\phi(x)|^2, \qquad \int_{\mathbb{R}^{n}}W\phi(x,\omega) \, dx = |\hat{\phi}(x)|^2,
\end{equation*}
for $\phi \in \mathscr{S}(\mathbb{R}^{n})$. \par Our primary interest in the cross-Wigner distribution stems from the following connection: For each $f \in L^{2}(\mathbb{R}^{2n})$ we define the operator 
$L_{f}: L^{2}(\mathbb{R}^{n}) \to L^{2}(\mathbb{R}^{n})$ by the formula 
\begin{equation*}
    \langle L_{f}\phi,\psi \rangle_{L^{2}(\mathbb{R}^{n})} = \langle f, W(\psi,\phi) \rangle_{L^{2}(\mathbb{R}^{2n})}, \qquad \phi,\psi \in L^{2}(\mathbb{R}^{n}).
\end{equation*}
Then $L_{f}$ is the \textit{Weyl quantization} of $f$, see \cite[Ch. 14]{grochenig2013foundations} for details. It is a non-trivial fact, see \cite{pool1966mathematical}, that the Weyl quantization gives a well-defined isomorphism between $L^{2}(\mathbb{R}^{2n})$ and $\mathcal{S}_{2}(L^{2}(\mathbb{R}^{n}))$, the space of Hilbert-Schmidt operators on $L^{2}(\mathbb{R}^{n})$. 

\subsubsection{Integrated Schr\"{o}dinger Representation}
\label{sec: Integrated Schrodinger Representation}
Recall that the Heisenberg group $\mathbb{H}^{n}$ is the Lie group with underlying manifold $\mathbb{R}^{n} \times \mathbb{R}^{n} \times \mathbb{R}$ and with the group multiplication 
\begin{equation*}
    (x,\omega,t) \cdot (x',\omega',t') \coloneqq \left(x + x', \omega + \omega', t + t' + \frac{1}{2}\left(x' \omega - x \omega' \right)\right).
\end{equation*}
The Heisenberg group is omnipresent in modern mathematics and theoretical physics, see \cite{howe1980role}. For a Hilbert space $\mathcal{H}$ we let $\mathcal{U}(\mathcal{H})$ denote the unitary operators on $\mathcal{H}$. The most important representation of the Heisenberg group is the \textit{Schr\"{o}dinger representation} $\rho: \mathbb{H}^{n} \to \mathcal{U}(L^{2}(\mathbb{R}^{n}))$ given by 
\begin{equation*}
    \rho(x,\omega,t)\phi(y) \coloneqq e^{2 \pi i t}e^{-\pi i x \omega}M_{\omega}T_{x}\phi(y),
\end{equation*}
where $T_{x}$ is the $n$-dimensional analogue of the translation operator defined in \eqref{Eq:translation_operator_and_dilation_operator} and $M_{\omega}$ is the \textit{modulation operator} given by 
\begin{equation*}
    M_{\omega}\phi(y) \coloneqq e^{2 \pi i \omega y}\phi(y), \qquad \phi \in L^{2}(\mathbb{R}^{n}).
\end{equation*}
The Schr\"{o}dinger representation is both irreducible and unitary. Let us use the abbreviated notation $z \coloneqq (x, \omega) \in \mathbb{R}^{2n}$ and $\pi(z) = M_{\omega}T_{x}$. Ignoring the central variable $t$, we can consider the \textit{integrated Schr\"{o}dinger representation} $\rho: L^{1}(\mathbb{R}^{2n}) \to \mathcal{L}(L^{2}(\mathbb{R}^{n}))$ given by 
\begin{equation}
\label{schrodinger_rep_def}
    \rho(f) = \int_{\mathbb{R}^{2n}} f(z) e^{- \pi i x \omega} \pi(z) \, dz,
\end{equation}
where $\mathcal{L}(L^{2}(\mathbb{R}^{n}))$ denotes the bounded linear operators on $L^{2}(\mathbb{R}^{n})$. We remark that the integral in \eqref{schrodinger_rep_def} is defined weakly. It turns out, see \cite[Thm.~1.30]{folland1989harmonic},
that the integrated representation $\rho$ extends from $L^1(\mathbb{R}^{2n})\cap L^2(\mathbb{R}^{2n})$ to a unitary map
$\rho: L^{2}(\mathbb{R}^{2n}) \to \mathcal{S}_{2}(L^{2}(\mathbb{R}^{n}))$. 
\subsubsection{Operator Convolution}
\label{sec: Operator Convolution}

Given a function $f\in L^1(\mathbb{R}^{2n})$ and a trace-class operator $S \in \mathcal{S}_{1}(L^{2}(\mathbb{R}^{n}))$, their convolution is the trace-class operator on $L^{2}(\mathbb{R}^{n})$ defined by
\begin{equation*}
f \star S \coloneqq \int_{\mathbb{R}^{2n}} f(z) \pi(z)S\pi(z)^{*} \ dz.
\end{equation*}
The convolution $f \star S$ satisfies the estimate $\|f \star S\|_{\mathcal{S}_{1}} \leq \|f\|_{L^{1}}\|S\|_{\mathcal{S}_{1}}$. \par 
One can also define the convolution between two operators: For two trace-class operators $S, T \in \mathcal{S}_{1}(L^{2}(\mathbb{R}^{n}))$ we define their convolution to be the function on $\mathbb{R}^{2n}$ given by
\begin{equation*}
\label{operator_convolution_heisenberg_case}
S \star T(z) \coloneqq \tr(S\pi(z)PTP\pi(z)^*),
\end{equation*} 
where $P\psi(t) \coloneqq \psi(-t)$ is the \textit{parity operator}.  The convolution $S \star T$ satisfies the estimate $\|S \star T\|_{L^{1}} \leq \|S\|_{\mathcal{S}_{1}}\|T\|_{\mathcal{S}_{1}}$, and the important integral relation \cite[Lem.~3.1]{werner1984}
\begin{equation} \label{eq:wernerintegral}
    \int_{\mathbb{R}^{2n}} S\star T(z) \, dz = \tr(S)\tr(T).
\end{equation}
To see the connection with the Wigner distribution, we note that the cross-Wigner distribution of $\psi,\phi \in L^{2}(\mathbb{R}^{n})$ can be written as
\begin{equation}
\label{eq:wigner_through_parity}
    W(\psi,\phi) = \psi \otimes \phi \star P,
\end{equation}
where $\psi \otimes \phi$ denotes the rank-one operator on $L^{2}(\mathbb{R}^{n})$ given by \[(\psi \otimes \phi)(\xi) \coloneqq \langle \xi, \phi \rangle_{L^2(\mathbb{R}^n)} \psi\qquad\text{ for }\xi \in L^{2}(\mathbb{R}^{n}).\] Similarly, the Weyl quantization of $f \in L^{1}(\mathbb{R}^{2n})$ may be expressed in terms of operator convolutions: 
\begin{equation}
\label{eq:quantization_through_parity}
    L_{f} = f \star P.
\end{equation}
Hence convolution with the parity operator $P$ gives a convenient way to represent the Wigner distribution and the Weyl quantization. \par
Finally, there is a Fourier transform for operators: Given a trace-class operator $S \in \mathcal{S}_{1}(L^{2}(\mathbb{R}^{n}))$ we define the \textit{Fourier-Wigner transform} $\mathcal{F}_{W}(S)$ of $S$ to be the function on $\mathbb{R}^{2n}$ given by
\begin{equation} \label{eq:fwheisenberg}
    \mathcal{F}_W(S)(z) \coloneqq e^{i \pi x \omega}\tr(S\pi(z)^*), \qquad z \in \mathbb{R}^{2n}.
\end{equation}
The Fourier-Wigner transform extends to a unitary map $\mathcal{F}_{W}: \mathcal{S}_{2}(L^{2}(\mathbb{R}^{n})) \to L^{2}(\mathbb{R}^{2n})$, where it turns out the to be inverse of the integrated Schr\"{o}dinger representation given in \eqref{schrodinger_rep_def}. By \cite[Prop.~2.5]{folland1989harmonic} it is related to the Weyl transform by  the elegant formula
\begin{equation} \label{eq: heisenbergweylFW}
    f = \mathcal{F}_{\sigma}(\mathcal{F}_W(L_f)),
\end{equation}
where $\mathcal{F}_{\sigma}$ denotes the \textit{symplectic Fourier transform}.

\subsection{Affine Weyl Quantization}
\label{sec: Affine_Quantization}

We briefly describe affine Weyl quantization and how this gives rise to the affine Wigner distribution. There is a unitary representation $\pi$ of the affine group $\textrm{Aff}$ on $L^{2}(\mathbb{R}_{+}, r^{-1} \, dr)$ given by
\begin{equation}
\label{Fourier_side_action}
    U(x,a)\psi(r) \coloneqq e^{2\pi i x r}\psi(ar) = \frac{1}{\sqrt{a}}M_{x}D_{\frac{1}{a}}\psi(r), \qquad \psi \in L^{2}(\mathbb{R}_{+}, r^{-1}\, dr).
\end{equation}
Since $r^{-1}\, dr$ is the Haar measure on $\mathbb{R}_{+}$ we will write $L^{2}(\mathbb{R_{+}}) \coloneqq L^{2}(\mathbb{R}_{+}, r^{-1}\, dr)$. Later we also consider another measure on $\mathbb{R}_{+}$ and will be more explicit when the situation requires it. 

To define the quantization scheme we will utilize the \textit{Stratonovich-Weyl operator} on $L^{2}(\mathbb{R}_{+})$ given by
\begin{equation}
\label{eq:Stratonovich-Weyl operator}
    \Omega(x,a)\psi(r) \coloneqq a \int_{\mathbb{R}^2}e^{-2\pi i (xu + av)}U\left(\frac{ve^u}{\lambda(u)},e^{u}\right)\psi(r) \, du \, dv.
\end{equation}
The following result was shown in \cite{gayral2007fourier} and provides us with an affine analogue of Weyl quantization.

\begin{proposition}[\cite{gayral2007fourier}]
\label{quantization_isometry_proposition}
There is a norm-preserving isomorphism between $L_{r}^{2}(\mathrm{Aff})$ and the space of Hilbert-Schmidt operators on $L^{2}(\mathbb{R}_{+})$. The isomorphism sends $f \in L_{r}^{2}(\mathrm{Aff})$ to the operator $A_{f}$ on $L^{2}(\mathbb{R}_{+})$ defined weakly by 
\begin{equation*}
A_{f}\psi(r) \coloneqq \int_{-\infty}^{\infty} \int_{0}^{\infty}f(x,a)\Omega(x,a)\psi(r) \, \frac{da \, dx}{a}, \qquad \psi \in L^{2}(\mathbb{R}_{+}).
\end{equation*}
\end{proposition}
We will refer to the association $f \mapsto A_{f}$ as \textit{affine Weyl quantization}, while $f$ is called the \textit{affine (Weyl) symbol} of $A_f$. To emphasize the correspondence between a Hilbert-Schmidt operator $A$ and its affine symbol $f$ we use the notation $f_{A} \coloneqq f$. The affine Weyl symbol of an operator $A$ is explicitly given by 
\begin{equation}
\label{eq: dequantization_formula}
    f_A(x,a) = \int_{-\infty}^{\infty}A_{K}\left(a\lambda(u),a\lambda(-u)\right)e^{-2\pi i x u} \, du,
\end{equation}
where $A_{K}:\mathbb{R}_{+} \times \mathbb{R}_{+} \to \mathbb{C}$ is the integral kernel of $A$ defined by \[A\psi(r) = \int_{0}^{\infty}A_{K}(r,s)\psi(s) \, \frac{ds}{s}, \qquad \psi \in L^{2}(\mathbb{R}_{+}).\]
By taking the affine Weyl symbol of the rank-one operator $\psi \otimes \phi$ on $L^{2}(\mathbb{R}_{+})$ given by \[\psi \otimes \phi(\xi) = \langle \xi, \phi \rangle_{L^{2}(\mathbb{R}_{+})} \psi\] for $\psi, \phi, \xi \in L^{2}(\mathbb{R}_{+})$, we obtain the following definition.

\begin{definition}
For $\phi, \psi \in L^{2}(\mathbb{R}_{+})$ we define the \textit{affine (cross-)Wigner distribution} $W_{\textrm{Aff}}^{\psi,\phi}$ to be the function on $\mathrm{Aff}$ given for $(x,a) \in \mathrm{Aff}$ by 
\begin{align*}
W_{\mathrm{Aff}}^{\psi,\phi}(x,a) & \coloneqq \int_{-\infty}^{\infty}\psi(a\lambda(u))\overline{\phi(a\lambda(-u))}e^{-2\pi i x u} \, du \\ & = \int_{-\infty}^{\infty}\psi\left(\frac{aue^u}{e^u - 1}\right)\overline{\phi\left(\frac{au}{e^{u} - 1}\right)}e^{-2\pi i x u} \, du. 
\end{align*} 
\end{definition}

When $\phi = \psi$ we refer to $W_{\textrm{Aff}}^{\psi} \coloneqq W_{\textrm{Aff}}^{\psi, \psi}$ as the \textit{affine Wigner distribution} of $\psi$. The weak interpretation of the integral defining $A_f$ means that we have the relation

\begin{equation}
\label{affine_weyl_correspondence}
    \left\langle A_{f}\phi,\psi \right\rangle_{L^{2}(\mathbb{R}_{+})} = \left\langle f,W_{\mathrm{Aff}}^{\psi,\phi}\right\rangle_{L_{r}^{2}(\mathrm{Aff})},
\end{equation}
for $f \in L_{r}^{2}(\textrm{Aff})$ and $\phi, \psi \in L^{2}(\mathbb{R}_{+})$. The affine Wigner distribution satisfies the orthogonality relation
\begin{equation}
\label{affine_orthogonality_relation_new}
    \int_{-\infty}^{\infty}\int_{0}^{\infty}W_{\mathrm{Aff}}^{\psi_{1}, \psi_{2}}(x,a)\overline{W_{\mathrm{Aff}}^{\phi_{1}, \phi_{2}}(x,a)} \, \frac{da \, dx}{a} = \langle \psi_{1},\phi_{1} \rangle_{L^2(\mathbb{R}_+)}\overline{\langle \psi_{2}, \phi_{2} \rangle}_{L^2(\mathbb{R}_+)},
\end{equation}
for $\psi_{1}, \psi_{2},\phi_{1}, \phi_{2} \in L^{2}(\mathbb{R}_+)$. Moreover, the affine Wigner distribution also satisfies the marginal property 
\begin{equation}
\label{eq:first_marginal_property}
    \int_{-\infty}^{\infty} W_{\mathrm{Aff}}^{\psi}(x,a) \, dx = |\psi(a)|^2, \quad (x,a) \in \mathrm{Aff},
\end{equation} 
for all rapidly decaying smooth functions $\psi$ on $\mathbb{R}_+$. We remark that a \textit{rapidly decaying smooth function} (also called a \textit{Schwartz function}) $\psi:\mathbb{R}_{+} \to \mathbb{C}$ is by definition a smooth function such that $x \mapsto \psi(e^x)$ is a rapidly decaying function on $\mathbb{R}$. The space of all rapidly decaying smooth functions on $\mathbb{R}_+$ will be denoted by $\mathscr{S}(\mathbb{R}_+)$. We will later also need the space $\mathscr{S}'(\mathbb{R}_{+})$ of bounded, anti-linear functionals on $\mathscr{S}(\mathbb{R}_+)$ called the \textit{tempered distributions} on $\mathbb{R}_{+}$. For more information regarding the affine Wigner distribution the reader is referred to \cite{berge2019affine}.

\section{Affine Operator Convolutions}
\label{sec: Affine Operator Convolutions}

In this part we introduce operator convolutions in the affine setting. We show that this notion is intimately related to affine Weyl quantization in Section \ref{sec: Relationship With Quantization}. In Section \ref{sec: Affine Grossmann-Royer Operator} we will introduce the affine Grossmann-Royer operator, which will be essential in Section \ref{sec: Operator Convolution for Tempered Distributions} where we prove the main connection between the affine Weyl quantization and the operator convolutions in Theorem \ref{thm:quantization_through_convolution}.

\subsection{Definitions and Basic Properties}
\label{sec: Definitions_and_Basic_Properties}
We begin by defining operator convolutions in the affine setting and derive basic properties. Recall that the usual convolution on the affine group with respect to the right Haar measure is given by
\begin{equation*}
    f\ast_\aff g(x,a) \coloneqq \int_{\aff} f(y,b)g(\left(x,a\right)\cdot(y,b)^{-1})\, \frac{dy \, db}{b}.
\end{equation*}
\begin{remark}
Other sources, e.g.\ \cite{folland2016course}, use the left Haar measure and define the convolution to be \[f\ast_{\aff_L} g((x,a))\coloneqq \check{f}\ast_\aff \check{g}((x,a)^{-1}),\] where $\check{f}(x,a)\coloneqq f((x,a)^{-1})$.
We will mainly work with the right Haar measure, and our definition ensures that 
\[\|f\ast_{\aff} g\|_{L^1_r(\aff)}\le \|f\|_{L^1_r(\aff)} \|g\|_{L^1_r(\aff)}.\]
Additionally, we have that 
\begin{equation*}
R_{(x,a)}(f*_\aff g)=(R_{(x,a)}f)*_\aff g.
\end{equation*}
\end{remark}
\begin{definition}\label{def:fun_ope_convolution}
Let $f\in L^1_r(\aff)$ and let $S$ be a trace-class operator on $L^2(\mathbb{R}_+)$. We define the \textit{convolution} $f\star_\aff S$ between $f$ and $S$ to be the operator on $L^{2}(\mathbb{R}_{+})$ given by 
\begin{equation*}
f\star_\aff S \coloneqq \int_{\aff} f(x,a) U(-x,a)^*SU(-x,a)\, \frac{dx \, da}{a},
\end{equation*}
where $U$ is the unitary representation given in \eqref{Fourier_side_action}. The integral is a convergent Bochner integral in the space of trace-class operators. 
\end{definition}
\begin{remark}
\hfill
\begin{enumerate}
    \item As we will see later, using $U(-x,a)$ instead of $U(x,a)$ in Definition \ref{def:fun_ope_convolution} ensures that the convolution is compatible with the following covariance property of the affine Wigner distribution:
\begin{equation}\label{E: right-inv}
    W^{U(-x,a)\phi,U(-x,a)\psi}_{\aff}(y,b)=W_{\aff}^{\phi,\psi}((y,b)\cdot(x,a)).
\end{equation}
    \item The notation $\star$ has a different meaning in \cite{gayral2007fourier}, where it is used to denote the so-called Moyal product of two functions defined on $\aff$.
\end{enumerate}
\end{remark}

\begin{definition}
Let $S$ be a trace-class operator and let $T$ be a bounded operator on $L^{2}(\mathbb{R}_{+})$. Then we define the \textit{convolution} $S \star_{\aff} T$ between $S$ and $T$ to be the function on $\textrm{Aff}$ given by
\begin{equation*}
    S\star_\aff T(x,a) \coloneqq \tr(SU(-x,a)^*TU(-x,a)).
\end{equation*}
\end{definition}

\begin{remark}
Recently, \cite{choi2020} defined another notion of convolution of trace-class operators. Unlike our definition, this convolution produces a new trace-class operator, with the aim of interpreting the trace-class operators as an analogue of the Fourier algebra. 
\end{remark}

It is straightforward to check that if $f$ is a positive function and $S,T$ are positive operators, then $f \star_{\textrm{Aff}} S$ is a positive operator and $S \star_{\textrm{Aff}} T$ is a positive function. Moreover, we have the elementary estimate 
\begin{equation} \label{eq:conv_op_fun_bochner}
    \|f \star_{\textrm{Aff}}S\|_{\mathcal{S}_{1}} \leq \|f\|_{L_{r}^{1}(\aff)}\|S\|_{\mathcal{S}_{1}}
\end{equation}
and 
\begin{equation} \label{eq:conv_op_op_duality}
    \|S \star_{\textrm{Aff}}T\|_{L^{\infty}(\aff)} \leq \|S\|_{\mathcal{S}_{1}}\|T\|_{\mathcal{L}(L^2(\mathbb{R}_+))}.
\end{equation}
The following result is proved by a simple computation.

\begin{lemma}\label{lem: convolution_of_rank_one}
For $\psi,\phi \in L^2(\mathbb{R}_+)$ and $S\in \mathcal{L}(L^2(\mathbb{R}_+))$, we have
\begin{equation*}
    (\psi \otimes \phi) \star_\aff S(x,a)=\langle SU(-x,a)\psi,U(-x,a)\phi \rangle_{L^2(\mathbb{R}_+)}.
\end{equation*}
In particular, for $\,\eta,\,\xi\in L^2(\mathbb{R}_+)$ we have
	\begin{equation*}
		(\psi \otimes \phi)\star_\aff (\eta \otimes \xi)(x,a)=\langle \psi,U(-x,a)^*\xi \rangle_{L^{2}(\mathbb{R}_{+})} \overline{\langle \phi,U(-x,a)^*\eta \rangle_{L^{2}(\mathbb{R}_{+})}},
	\end{equation*}
	and 
		\begin{equation*}
		(\psi \otimes \psi)\star_\aff (\xi \otimes \xi)(x,a)=|\langle \psi,U(-x,a)^*\xi \rangle_{L^{2}(\mathbb{R}_{+})}|^2.
	\end{equation*}
\end{lemma}

A natural question to ask is whether the three different notions of convolution we have introduced are compatible. The following proposition gives an affirmative answer to this question.

\begin{proposition}
\label{first_associativity_proposition}
Let $f,g\in L^1_r(\aff)$, $S \in \mathcal{S}_{1}$, and let $T$ be a bounded operator on $L^{2}(\mathbb{R}_{+})$. Then we have the compatibility equations
\begin{align*}
    (f\star_\aff S)\star_\aff T&=f\ast_\aff (S\star_\aff T), \\
    f\star_\aff (g\star_\aff S)&=(f\ast_\aff g)\star_\aff S.
\end{align*}
\end{proposition}
\begin{proof}
The first equality follows from the computation
\begin{align*}
    \left(f\ast_\aff(S\star_\aff T)\right)(x,a)&=  \int_{\aff} f(y,b)\tr(SU(-y,b)U(-x,a)^*TU(-x,a)U(-y,b)^*)  \,\frac{dy\, db}{b} \\
    &=\int_{\aff} f(y,b)\tr(U(-y,b)^*SU(-y,b)U(-x,a)^*TU(-x,a))  \,\frac{dy\, db}{b} \\
    &= \tr \left(\int_\aff f(y,b) U(-y,b)^*SU(-y,b)\, \frac{dy\, db}{b} U(-x,a)^* T U(-x,a) \right) \\
    &= \left(\left(f\star_\aff S\right)\star_\aff T\right)(x, a).
\end{align*}
We are allowed to take the trace outside the integral since the second to last line is essentially the duality action of the bounded operator $U(-x,a)^*TU(-x,a)$ on a convergent Bochner integral in the space of trace-class operators.

For the second equality, we use change of variables and obtain
\begin{align*}
    (f\ast_\aff g)\star_\aff S&= \int_\aff \int_\aff f(x,a)g((z,c)\cdot (x,a)^{-1}) U(-z,c)^*SU(-z,c)\, \frac{dx\,da}{a} \, \frac{dz\, dc}{c} \\
    &=  \int_\aff \int_\aff f(x,a) g(y,b) U(-x,a)^*U(-y,b)^* SU(-y,b)U(-x,a) \,\frac{dy\,db}{b} \,\frac{dx\,da}{a} \\
    &=\int_\aff f(x,a) U(-x,a)^* \int_\aff  g(y,b) U(-y,b)^* SU(-y,b)\, \frac{dy\,db}{b} U(-x,a) \,\frac{dx\,da}{a} \\
    &= f\star_\aff (g\star_\aff S).
\end{align*}
Changing the order of integration above is allowed by Fubini's theorem for Bochner integrals \cite[Prop.~1.2.7]{hytonen2016book}. Fubini's theorem is applicable since 
\begin{equation*}
    \int_\aff \int_\aff |f(x,a)|\cdot |g((z,c)\cdot (x,a)^{-1})|\cdot \|U(-z,c)^*SU(-z,c)\|_{\mathcal{S}_1}\, \frac{dx\,da}{a} \, \frac{dz\, dc}{c}
\end{equation*}
is bounded from above by
\begin{equation*}
    \|S\|_{\mathcal{S}_1} \int_\aff  |f(x,a)|\, \frac{dx\,da}{a} \int_\aff |g(z,c)|\, \frac{dz\, dc}{c}<\infty. \qedhere
\end{equation*}
\end{proof}

\subsection{Relationship With Affine Weyl Quantization}
\label{sec: Relationship With Quantization}
The goal of this section is to connect the affine Weyl quantization described in Section \ref{sec: Affine_Quantization} with the convolutions defined in Section \ref{sec: Definitions_and_Basic_Properties}. We first establish a preliminary result describing how right multiplication on the affine group affects the affine Weyl quantization.

\begin{lemma}\label{L: Right mul of symbol}
Let $A_f \in \mathcal{S}_{2}$ with affine Weyl symbol $f \in L_{r}^{2}(\aff)$. For $(x,a)\in \aff$, the affine Weyl symbol of $U(-x,a)^*A_fU(-x,a)$ is $R_{(x,a)^{-1}}f$.
\end{lemma}

\begin{proof}
	The result follows from \eqref{affine_weyl_correspondence} and the computation
	\begin{align*}
		\langle U(-x,a)^*A_fU(-x,a) \psi,\phi \rangle_{L^2(\mathbb{R}_+)} &= \langle A_fU(-x,a)\psi,U(-x,a)\phi \rangle_{L^2(\mathbb{R}_+)} \\
		&=  \langle f,W_\aff^{U(-x,a)\phi,U(-x,a)\psi} \rangle_{L^2_r(\aff)} \\
		&= \langle f , R_{(x,a)}W_\aff^{\phi,\psi} \rangle_{L^2_r(\aff)} \\
		&= \langle R_{(x,a)^{-1}}f, W_\aff^{\phi,\psi} \rangle_{L^2_r(\aff)}.\qedhere
	\end{align*}
\end{proof}
We are now ready to prove the first result showing the connection between convolution and affine Weyl quantization.
\begin{proposition}
\label{function-operator-convolution-quantization}
Assume that $A_{f}\in\mathcal{S}_2$ with affine Weyl symbol $f \in L_{r}^{2}(\mathrm{Aff})$, and let $g \in L_{r}^{1}(\mathrm{Aff})$. Then the affine Weyl symbol of $g \star_\aff A_f$ is $g *_{\mathrm{Aff}} f$, that is, \[g \star_\aff A_f = A_{g *_{\mathrm{Aff}} f}.\]
\end{proposition}

\begin{proof}
The operator $g\star_\aff A_f$ is defined as the $\mathcal{S}_2$-convergent Bochner integral \[g\star_\aff A_f = \int_\aff g(x,a) U(-x,a)^*A_fU(-x,a) \, \frac{dx\, da}{a}. \]
 By Proposition \ref{quantization_isometry_proposition}, the map $\mathfrak{W}:\mathcal{S}_2\to L^2_r(\aff)$ given by $\mathfrak{W}(A_f)=f$ is unitary. Since bounded operators commute with convergent Bochner integrals, we have using Lemma \ref{L: Right mul of symbol} that
\begin{align*}
    \mathfrak{W}\left( g\star_\aff A_f \right)&=\int_\aff g(x,a) \mathfrak{W}\left(U(-x,a)^*A_fU(-x,a)\right) \, \frac{dx\, da}{a} \\
    &= \int_\aff g(x,a) R_{(x,a)^{-1}}\mathfrak{W}\left(A_f\right) \, \frac{dx\, da}{a} \\
    &= g\ast_\aff f.\qedhere
\end{align*}
\end{proof}

We can also express the convolution of two operators in terms of their affine Weyl symbols. 

\begin{proposition} \label{prop:convolutionasweyl}
	Let $A_{f},A_{g} \in \mathcal{S}_{2}$ with affine Weyl symbols $f,g \in L_{r}^{2}(\aff)$. If additionally $g\in L_{l}^2(\aff)$, then we have
	\[
A_{f}\star_\aff A_{g} = f\ast_\aff \check{g},
	\]
	where $\check{g}(x,a) = g((x,a)^{-1})$ for $(x,a) \in \mathrm{Aff}$.
\end{proposition}

\begin{proof}
	Using Proposition \ref{quantization_isometry_proposition} and Lemma \ref{L: Right mul of symbol} we compute that 
	\begin{align*}
		(A_f\star_\aff A_g)(x,a)&=\tr(A_fU(-x,a)^*A_gU(-x,a)) \\
		&= \langle A_f,U(-x,a)^*A_g^*U(-x,a) \rangle_{\mathcal{S}_2} \\
		&= \langle f,R_{(x,a)^{-1}}\overline{g}  \rangle_{L^2_r(\aff)} \\
		&= \int_\aff f(y,b)g((y,b)\cdot (x,a)^{-1}) \, \frac{dy \, db}{b} \\
		&=  \int_\aff f(y,b)\check{g}((x,a)\cdot(y,b)^{-1}) \, \frac{dy \, db}{b} \\
		&= f\ast_\aff \check{g}(x,a).
	\end{align*}
	The result follows as $\check{g}\in L_{r}^2(\aff)$ if and only if $g \in L_{l}^2(\aff)$.
\end{proof}

\subsection{Affine Weyl Quantization of Coordinate Functions}
\label{sec: Affine Quantization of Coordinate Functions}

Of particular interest is the affine Weyl quantization of the coordinate functions $f_{x}(x,a) \coloneqq x$ and $f_{a}(x,a) \coloneqq a$ for $(x,a) \in \mathrm{Aff}$. Due to the fact that the coordinate functions are not in $L^2_r(\aff)$, we first need to interpret the quantizations $A_{f_{x}}$ and $A_{f_{a}}$ in a rigorous manner. We begin this task by defining rapidly decaying smooth function and tempered distributions on the affine group.
\begin{definition}
Let $\mathscr{S}(\mathrm{Aff})$ denote the smooth functions $f:\mathrm{Aff} \to \mathbb{C}$ such that
\[(x,\omega) \longmapsto f(x,e^{\omega}) \in \mathscr{S}(\mathbb{R}^2).\] We refer to $\mathscr{S}(\mathrm{Aff})$ as the space of \textit{rapidly decaying smooth functions} (or \textit{Schwartz functions}) on the affine group.
\end{definition}
There is a natural topology on $\mathscr{S}(\aff)$ induced by the semi-norms
\begin{equation}
\label{eq:semi-norms-on-affine}
    \|f\|_{\alpha, \beta} \coloneqq \sup_{x,\, \omega\in \mathbb{R}}|x|^{\alpha_1}|\omega|^{\alpha_2}\left|\partial_x^{\beta_1}\partial_{\omega}^{\beta_2}f(x,e^{\omega})\right|,
\end{equation}
for $\alpha = (\alpha_1, \alpha_2)$ and $\beta = (\beta_1, \beta_2)$ in $\mathbb{N}_{0} \times \mathbb{N}_{0}$. With these semi-norms, the space $\mathscr{S}(\aff)$ becomes a Fréchet space. The space of bounded, anti-linear functionals on $\mathscr{S}(\mathrm{Aff})$ is denoted by $\mathscr{S}'(\mathrm{Aff})$ and called the space of \textit{tempered distributions} on $\aff$. 

\begin{lemma}
\label{lemma_swartz}
For any $f \in \mathscr{S}'(\mathrm{Aff})$ we can define $A_{f}$ as the map $A_{f}: \mathscr{S}(\mathbb{R}_{+}) \to \mathscr{S}'(\mathbb{R}_{+})$ defined by the relation 
\[\langle A_{f}\psi, \phi \rangle_{\mathscr{S}',\mathscr{S}} = \left\langle f, W_{\mathrm{Aff}}^{\phi, \psi} \right\rangle_{\mathscr{S}',\mathscr{S}}, \qquad \psi, \phi \in \mathscr{S}(\mathbb{R}_{+}).\]
Additionally, the map $f\mapsto A_f$ is injective.
\end{lemma}
\begin{proof}
It was shown in \cite[Cor.~6.6]{berge2019affine} that for any $\phi, \psi \in \mathscr{S}(\mathbb{R}_{+})$ then $W_{\mathrm{Aff}}^{\phi,\psi} \in \mathscr{S}(\mathrm{Aff})$. 
Hence the pairing $\left\langle f, W_{\mathrm{Aff}}^{\phi, \psi} \right\rangle_{\mathscr{S}',\mathscr{S}}$ is well defined. 

For the injectivity it suffices to show that $A_{f} = 0$ implies that $f = 0$. Let us first reformulate this slightly: If $A_{f} = 0$, then we have that 
\[\langle A_{f}\psi, \phi \rangle_{\mathscr{S}',\mathscr{S}} = \left\langle f, W_{\mathrm{Aff}}^{\phi, \psi} \right\rangle_{\mathscr{S}',\mathscr{S}} = 0\]
for all $\psi, \phi \in \mathscr{S}(\mathbb{R}_{+})$. We could conclude that $f = 0$ if we knew that any $g \in \mathscr{S}(\aff)$ could be approximated (in the Fréchet topology) by linear combinations of elements on the form $ W_{\mathrm{Aff}}^{\phi, \psi}$ for $\psi, \phi \in \mathscr{S}(\mathbb{R}_{+})$.
To see that this is the case, we translate the problem to the  Heisenberg setting.\par
The \textit{Mellin transform} $\mathcal{M}$ is given by
\[\mathcal{M}(\phi)(x) = \mathcal{M}_{r}(\phi)(x) \coloneqq \int_{0}^{\infty}\phi(r) r^{- 2 \pi i x} \, \frac{dr}{r}.\] 
Define the functions $\Psi$ and $\Phi$ to be $\Psi(x) \coloneqq \psi(e^{x})$ and $\Phi(x) \coloneqq \phi(e^{x})$ for $\psi,\phi\in L^2(\mathbb{R}_+)$. 
A reformulation of \cite[Lem.~6.4]{berge2019affine} shows that we have the relation
\begin{equation*}
       W_{\mathrm{Aff}}^{\psi,\phi}(x,a) = \mathcal{M}_{y}^{-1} \otimes \mathcal{M}_{b}\left[\left(\frac{\sqrt{b}\log(b)}{b - 1}\right)^{2\pi i y}  \mathcal{F}_{\sigma}W(\Psi,\Phi)\big(\log(b),y\big)\right](x,a),
   \end{equation*}
where $W$ is the cross-Wigner distribution. The correspondence preserves Schwartz functions, due to the term \[\left(\frac{\sqrt{b}\log(b)}{b - 1}\right)^{2\pi i y}\] being smooth with polynomially bounded derivatives. This gives a bijective correspondence between $W_{\mathrm{Aff}}^{\psi,\phi} \in \mathscr{S}(\aff)$ and $W(\Psi,\Phi) \in \mathscr{S}(\mathbb{R}^{2})$.  As such, the injectivity question is reduced to asking whether the linear span of elements on the form $W(f,g)$ for $f,g \in \mathscr{S}(\mathbb{R})$ is dense in $\mathscr{S}(\mathbb{R}^{2})$. One way to verify this well-known fact is to note that the map $f\otimes g\mapsto W(f,g)$, where $f\otimes g(x,y)=f(x)g(y)$, extends to a topological isomorphism on $\mathscr{S}(\mathbb{R}^2)$, see for instance \cite[(14.21)]{grochenig2013foundations} for the formula of this isomorphism. The density of elements on the form $W(f,g)$ for $f,g \in \mathscr{S}(\mathbb{R})$ therefore follows as the functions $h_m\otimes h_n$, where $\{h_n\}_{n = 0}^{\infty}$ are the Hermite functions, span a dense subspace of $\mathscr{S}(\mathbb{R}^{2})$ by \cite[Thm.~V.13]{Reed:1980}.
\end{proof}

\begin{example}
Consider the constant function on the affine group given by $1(x,a) = 1$ for all $(x,a) \in \mathrm{Aff}$. Then the quantization $A_{1}$ is the identity operator since for $\psi,\phi \in \mathscr{S}(\mathbb{R}_+)$
\begin{align*}
    \langle A_{1} \psi, \phi \rangle_{\mathscr{S}',\mathscr{S}} & = \langle 1, W_{\mathrm{Aff}}^{\phi,\psi} \rangle_{\mathscr{S}',\mathscr{S}}
    \\ & = 
    \int_{\mathrm{Aff}} \overline{W_{\mathrm{Aff}}^{\phi,\psi}(x,a)} \, \frac{da \, dx}{a} 
    \\ & =
    \int_{0}^{\infty} \psi(a) \overline{\phi(a)} \, \frac{da}{a} 
    \\ & = 
   \langle \psi, \phi \rangle_{L^2(\mathbb{R}_+)}. 
\end{align*}
Notice that we used a straightforward generalization of the marginal property of the affine Wigner distribution given in \eqref{eq:first_marginal_property}, see the proof of \cite[Prop.~3.4]{berge2019affine} for details.
\end{example}

To motivate the next result, consider the coordinate functions $\sigma_{x}(x,\omega) \coloneqq x$ and $\sigma_{\omega}(x,\omega) \coloneqq \omega$ for $(x,\omega) \in \mathbb{R}^{2n}$. 
The Weyl quantizations $L_{\sigma_{x}}$ and $L_{\sigma_{\omega}}$ are the well-known \textit{position operator} and \textit{momentum operator} in quantum mechanics. In particular, the commutator \[\left[L_{\sigma_{x}}, L_{\sigma_{\omega}}\right] \coloneqq L_{\sigma_{x}} \circ L_{\sigma_{\omega}} - L_{\sigma_{\omega}} \circ L_{\sigma_{x}}\] is a constant times the identity by \cite[Prop.~3.8]{hall2013quantum}. This is precisely the relation for the Lie algebra of the Heisenberg group.
In light of this, the following proposition shows that the affine Weyl quantization has the expected expression for the coordinate functions.
\begin{theorem}
\label{quantization_of_coordinates}
    Let $f_{x}$ and $f_{a}$ be the coordinate functions on the affine group. 
    The affine Weyl quantizations $A_{f_{x}}$ and $A_{f_{a}}$ are well-defined as maps from $\mathscr{S}(\mathbb{R}_{+})$ to $\mathscr{S}'(\mathbb{R}_{+})$ and are explicitly given by \[A_{f_{x}}\psi(r) = \frac{1}{2\pi i}r \psi'(r), \qquad A_{f_{a}}\psi(r) = r\psi(r), \qquad \psi \in \mathscr{S}(\mathbb{R}_{+}).\] In particular, we have the commutation relation \[[A_{f_{x}},A_{f_{a}}] =  \frac{1}{2\pi i}A_{f_{a}}.\]
    This is, up to re-normalization, precisely the Lie algebra structure of $\mathfrak{aff}$ given in \eqref{eq:affine_lie_algebra_structure}.
\end{theorem}

\begin{proof}
Let us begin by computing $A_{f_{x}}$. We can change the order of integrating by Fubini's theorem and obtain for $\psi, \phi \in \mathscr{S}(\mathbb{R}_{+})$ that
\begin{align*}
    \langle A_{f_{x}}\psi, \phi \rangle_{\mathscr{S}',\mathscr{S}} & = \left\langle f_{x}, W_{\mathrm{Aff}}^{\phi, \psi} \right\rangle_{\mathscr{S}',\mathscr{S}}
    \\ & = 
    \int_{-\infty}^{\infty} \int_{0}^{\infty} x \, \overline{\int_{-\infty}^{\infty}\phi(a \lambda(u)) \overline{\psi(a\lambda(-u))} e^{-2\pi i x u} \, du} \, \frac{da \, dx}{a}
    \\ & = 
    \int_{-\infty}^{\infty} \int_{0}^{\infty} \left(\int_{-\infty}^{\infty} x e^{2\pi i x u} \, dx \right) \, \psi(a\lambda(u)) \overline{\phi(a \lambda(-u))} \, \frac{da \, du}{a}.
\end{align*}
Notice that the inner integral is equal to 
\[\int_{-\infty}^{\infty} x e^{2\pi i x u} \, dx = \frac{1}{2\pi i}\delta_{0}^{'}(u),\]
where \[\int_{-\infty}^{\infty}\delta_{0}^{'}(u)\psi(u) \, du = \psi'(0).\] Hence we have the relation 
\[\langle A_{f_{x}}\psi, \phi \rangle_{\mathscr{S}',\mathscr{S}} = \frac{1}{2 \pi i} \int_{0}^{\infty} \frac{\partial}{\partial u}\left(\psi(a\lambda(u)) \overline{\phi(a \lambda(-u))}\right)\Big|_{u = 0} \, \, \frac{da}{a}.\]
By using the formulas $\lambda(0) = 1$ and $\lambda'(0) = 1/2$ we can simplify and obtain
\[\langle A_{f_{x}}\psi, \phi \rangle_{\mathscr{S}',\mathscr{S}} = \frac{1}{4 \pi i} \int_{0}^{\infty} a \cdot \left(\psi'(a)\overline{\phi(a)} - \psi(a)\overline{\phi'(a)} \right)  \, \frac{da}{a}.\]
Using integration by parts we obtain the claim since
\[\langle A_{f_{x}}\psi, \phi \rangle_{\mathscr{S}',\mathscr{S}} = \int_{0}^{\infty} \left[\frac{1}{2\pi i}a \psi'(a)\right]\overline{\phi(a)}  \, \frac{da}{a}.\]

For $A_{f_{a}}$ we have by similar calculations as above that 
\begin{align*}
    \langle A_{f_{a}}\psi, \phi \rangle_{\mathscr{S}',\mathscr{S}} 
    & =
    \int_{-\infty}^{\infty}\int_{0}^{\infty}\left(\int_{-\infty}^{\infty} 1 \cdot e^{2\pi i x u} \, dx \right) \, a \cdot \psi(a\lambda(u)) \overline{\phi(a \lambda(-u))} \, \frac{da \, du}{a}
    \\ & =
    \int_{-\infty}^{\infty}\int_{0}^{\infty}\delta_{0}(u) \left( a \cdot \psi(a\lambda(u)) \overline{\phi(a \lambda(-u))}\right) \, \frac{da \, du}{a}
    \\ & = 
    \int_{0}^{\infty} a \psi(a) \overline{\phi(a)} \, \frac{da}{a}.
\end{align*}
The commutation relation follows from straightforward computation.
\end{proof}

\subsection{The Affine Grossmann-Royer Operator}
\label{sec: Affine Grossmann-Royer Operator}

In this section we introduce the affine Grossmann-Royer operator with the aim of obtaining an affine parity operator analogous to the (Heisenberg) parity operator $P$ in Section \ref{sec: Operator Convolution}. The main reason for this is to obtain affine version of the formulas \eqref{eq:wigner_through_parity} and \eqref{eq:quantization_through_parity} so that we can describe the affine Weyl quantization through convolution. Recall that the \textit{(Heisenberg) Grossmann-Royer operator} $R(x,\omega)$ for $(x,\omega) \in \mathbb{R}^{2n}$ is defined by the relation \[W(f,g)(x,\omega) = \left\langle R(x,\omega)f, g \right\rangle_{L^{2}(\mathbb{R}^n)}, \qquad f,g \in L^{2}(\mathbb{R}^n).\] Analogously, we have the following definition.
\begin{definition}
We define the \textit{affine Grossmann-Royer operator} $R_{\mathrm{Aff}}(x,a)$ for $(x,a) \in \mathrm{Aff}$ by the relation \[W_{\mathrm{Aff}}^{\psi,\phi}(x,a) = \left\langle R_{\mathrm{Aff}}(x,a)\psi, \phi \right\rangle_{\mathscr{S}',\mathscr{S}}, \qquad \psi,\phi \in \mathscr{S}(\mathbb{R}_+).\]
\end{definition}
We restrict our attention to Schwartz functions for convenience since then $W_{\mathrm{Aff}}^{\psi,\phi} \in \mathscr{S}(\mathrm{Aff})$ by \cite[Cor.~6.6]{berge2019affine}, and hence have well-defined point values. The Grossmann-Royer operator $R_{\mathrm{Aff}}(x,a)$ is precisely the affine Weyl quantization of the point mass $\delta_{\mathrm{Aff}}(x,a) \in \mathscr{S}'(\textrm{Aff})$ for $(x,a) \in \aff$ defined by \[\langle\delta_{\mathrm{Aff}}(x,a),f\rangle_{\mathscr{S}',\mathscr{S}}\coloneqq \overline{f(x,a)}, \qquad f \in \mathscr{S}(\mathrm{Aff}).\] Since this is also true for the Stratonovich-Weyl operator $\Omega(x,a)$ given in \eqref{eq:Stratonovich-Weyl operator}, it follows that $R_{\mathrm{Aff}}(x,a) = \Omega(x,a)$ for all $(x,a) \in \mathrm{Aff}$. From \cite[p. 12]{gayral2007fourier} it follows that we have the \textit{affine covariance relation}
\begin{equation}
\label{eq:affine covariance relation}
    U(-x,a)^*R_{\mathrm{Aff}}(0,1)U(-x,a) = R_{\mathrm{Aff}}(x,a).
\end{equation}
The following result, which is a straightforward computation, shows that $R_{\mathrm{Aff}}(x,a)$ is an unbounded and densely defined operator on $L^{2}(\mathbb{R}_{+})$.
\begin{lemma}
\label{Grossman_royer_formula}
Fix $\psi \in \mathscr{S}(\mathbb{R}_+)$ and $(x,a) \in \mathrm{Aff}$. The affine Grossmann-Royer operator $R_{\mathrm{Aff}}(x,a)$ has the explicit form \[R_{\mathrm{Aff}}(x,a)\psi(r) = \frac{e^{2 \pi i x \lambda^{-1}\left(\frac{r}{a}\right)}\lambda^{-1}\left(\frac{r}{a}\right)\left(1 - e^{\lambda^{-1}\left(\frac{r}{a}\right)}\right)}{1 + \lambda^{-1}\left(\frac{r}{a}\right) - e^{\lambda^{-1}\left(\frac{r}{a}\right)}} \cdot \psi\left(r e^{-\lambda^{-1}\left(\frac{r}{a}\right)}\right),\]
where $\lambda$ is the function given in \eqref{eq:lambda_function}.
\end{lemma}

We will be particularly interested in the \textit{affine parity operator} $P_{\mathrm{Aff}}$ given by the affine Grossmann-Royer operator at the identity element, that is,  \[P_{\mathrm{Aff}}(\psi)(r) : = R_{\textrm{Aff}}(0,1)\psi(r)=\frac{\lambda^{-1}(r)(1-e^{\lambda^{-1}(r)})}{1+\lambda^{-1}(r)-e^{\lambda^{-1}(r)}}\psi\left(re^{-\lambda^{-1}(r)}\right),\]
for $\psi \in \mathscr{S}(\mathbb{R}_{+})$. The affine parity operator $P_{\mathrm{Aff}}$ is symmetric as an unbounded operator on $L^{2}(\mathbb{R}_{+})$. Moreover, we see from the relation
\[e^{\lambda^{-1}(r)} - 1 = \frac{\lambda^{-1}(r)e^{\lambda^{-1}(r)}}{r}\]
that we have the alternative formula
\begin{equation}
\label{parity_first_formula}
    P_{\mathrm{Aff}}(\psi)(r) = \frac{\lambda^{-1}(r)}{1-re^{-\lambda^{-1}(r)}}\psi\left(re^{-\lambda^{-1}(r)}\right).
\end{equation}

An important commutation relation for the (Heisenberg) Grossman-Royer operator $R(x,\omega)$ for $(x,\omega) \in \mathbb{R}^{2n}$ is given by 
\begin{equation}
\label{classical_Grossman_identity}
    P \circ R(x, \omega) = R(-x, -\omega) \circ P.
\end{equation}
The following proposition shows that the analogue of \eqref{classical_Grossman_identity} breaks down in the affine setting due to $\aff$ being non-unimodular. As the proof is a straightforward computation, we leave the details to the reader.

\begin{proposition}
The commutation relation \[P_{\mathrm{Aff}} \circ R_{\mathrm{Aff}}(x,a) = R_{\mathrm{Aff}}\left((x,a)^{-1}\right) \circ P_{\mathrm{Aff}}\]
holds precisely for those $(x,a) \in \mathrm{Aff}$ such that $\Delta(x, a) = \frac{1}{a} = 1$.
\end{proposition}

We will now show that both the function $\lambda$ in \eqref{eq:lambda_function} and the affine parity operator $P_{\mathrm{Aff}}$ are related to the Lambert $W$ function. Recall that the \textit{(real) Lambert $W$ function} is the multivalued function defined to be the inverse relation of the function $f(x) = xe^x$ for $x \in \mathbb{R}$. The function $f(x)$ for $x < 0$ is not injective. There exist for each $y \in (-1/e, 0)$ precisely two values $x_1,x_2 \in (-\infty, 0)$ such that \[x_{1}e^{x_1} = x_{2}e^{x_{2}} = y.\] As the solutions appear in pairs, we can define $\sigma$ to be the function that permutes these solutions, that is, $\sigma(x_{1}) = x_{2}$ and $\sigma(x_{2}) = x_1$. For $y = -1/e$ there is only one solution to the equation $xe^{x} = y$, namely $x = -1$. Hence we define $\sigma(-1) = -1$. We can represent the function $\sigma$ as 
\[
\sigma(x) = 
\begin{cases}
W_0(xe^x),& x < -1 \\
-1,& x = -1 \\
W_{-1}(xe^x),& -1 < x < 0 \\
\end{cases},
\]
where $W_{0}, W_{-1}$ are the two branches of the Lambert $W$ function satisfying
\[W_0(xe^x)=x, \quad \text{ for } x\ge -1\]
and
\[W_{-1}(xe^x)=x, \quad \text{ for } x\le -1.\]

\begin{lemma}
\label{lambda_inverse}
    The inverse of $\lambda$ is given by 
    \[\lambda^{-1}(r) = \log\left(\frac{-r}{\sigma(-r)}\right) = \sigma(-r) + r, \qquad r > 0.\]
\end{lemma}

\begin{proof}
   
    To find the inverse of $\lambda$ we solve the equation \[r = \lambda(u) = \frac{ue^u}{e^u - 1} = \frac{-u}{e^{-u} - 1}.\] A simple computation shows that $-r = -u - re^{-u}$. Making the substitution $v=e^{-u}$ together with straightforward manipulations shows that
    \begin{equation}
    \label{intermediate_equation}
        -re^{-r}=-rve^{-rv}.
    \end{equation}
    The trivial solution to \eqref{intermediate_equation} is given by solving the equation $-r=-rv$. Checking with the original equation, this can not give the inverse of $\lambda$. We get the first equality from the definition of $\sigma$ together with recalling that $u = -\log(v)$.
     The final equality follows from  \[\log\left(\frac{-r}{\sigma(-r)}\right) = \log\left(\frac{-r}{\sigma(-r)}\frac{\sigma(-r)e^{\sigma(-r)}}{-re^{-r}}\right)=\sigma(-r)+r.\qedhere\]
\end{proof}

\begin{remark}
A minor variation of the function $\sigma$ appeared in \cite[Section 3]{gayral2007fourier} where it was defined by the relation in Lemma \ref{lambda_inverse}. The advantage of understanding the connection to the Lambert $W$ function is that properties such as $\sigma(\sigma(x)) = x$ for every $x < 0$ become trivial in this description. 
\end{remark}

\begin{corollary}
The affine parity operator $P_{\mathrm{Aff}}$ can be written as
\[P_{\mathrm{Aff}}(\psi)(r) = \frac{\sigma(-r)+r}{\sigma(-r)+1}\psi(-\sigma(-r)), \qquad \psi \in \mathscr{S}(\mathbb{R}_{+}).\]
In particular, we have $P_{\mathrm{Aff}}(\psi)(1) = 2\psi(1)$.
\end{corollary}

\begin{proof}
The formula for $P_{\mathrm{Aff}}(\psi)$ is obtained from Lemma \ref{lambda_inverse} together with \eqref{parity_first_formula}. To find the value $P_{\mathrm{Aff}}(\psi)(1)$, we use \eqref{parity_first_formula} and the fact that \[\psi\left(re^{-\lambda^{-1}(r)}\right)\Big|_{r = 1} = \psi(1).\]
Hence the claim follows from L'Hopital's rule since \[ \lim_{r \to 1}\frac{\lambda^{-1}(r)}{\lambda^{-1}(r) + 1 - r} = \frac{(\lambda^{-1})'(1)}{(\lambda^{-1})'(1) - 1} = 2. \qedhere\]
\end{proof}

\subsection{Operator Convolution for Tempered Distributions}
\label{sec: Operator Convolution for Tempered Distributions}

This section is all about expressing the affine Weyl quantization of a function $f \in \mathscr{S}(\aff)$ by using affine convolution. To be able to do this, we will first define what it means for $A_{f}$ to be a Schwartz operator. 

\begin{definition}
We say that a Hilbert-Schmidt operator $A:L^{2}(\mathbb{R}_{+}) \to L^{2}(\mathbb{R}_{+})$ is a \textit{ Schwartz operator} if the integral kernel $A_{K}$ of $A$ satisfies $A_{K} \in \mathscr{S}(\mathbb{R}_{+} \times \mathbb{R}_{+})$, that is, if 
    \[(x,\omega) \longmapsto A_{K}(e^{x}, e^{\omega}) \in \mathscr{S}(\mathbb{R}^{2}).\]
\end{definition}

\begin{proposition}
\label{prop: equallity of spaces}
A Hilbert-Schmidt operator $A\in \mathcal{S}_2$ is a Schwartz operator if and only if $A=A_f$ for some $f\in \mathscr{S}(\aff)$.
\end{proposition}

\begin{proof}
Assume that $A$ is a Schwartz operator. In \cite[Equation (4.8)]{gayral2007fourier} it is shown that the integral kernel $A_{K}$ of $A$ is related to the affine Weyl symbol $f_{A}$ of $A$ by the formula 
\[A_{K}(r,s) = \int_{-\infty}^{\infty}f_{A}\left(x, \frac{r - s}{\log(r/s)}\right) e^{2 \pi i x \log(r/s)} \, dx.\]
Since the inverse-Fourier transform preserves Schwartz functions, together with the definition of $\mathscr{S}(\mathbb{R}_{+} \times \mathbb{R}_{+})$, we have that \[(r,s) \longmapsto f_{A}\left(\log(r/s), \frac{r - s}{\log(r/s)}\right) \in \mathscr{S}(\mathbb{R}_{+} \times \mathbb{R}_{+}).\]
By performing the change of variable $x = \log(r/s)$ and $s = e^{\omega}$ for $\omega \in \mathbb{R}$ we obtain \[(x,\omega) \longmapsto f_{A}\left(x, e^{\omega}\frac{e^{x} - 1}{x}\right) \in \mathscr{S}(\mathbb{R}^{2}).\]
Finally, by letting $u = \log((e^{x} - 1)/x) + \omega$ we see that 
\[(x,u) \longmapsto f_{A}\left(x, e^{u}\right) \in \mathscr{S}(\mathbb{R}^{2}),\]
due to the fact that $x \mapsto \log((e^{x} - 1)/x)$ has polynomial growth. \par 
Conversely, assume that $A=A_{f}$ for $f \in \mathscr{S}(\aff)$. The integral kernel $A_{K}$ is then given by \[A_{K}(r,s) = \mathcal{F}_{1}^{-1}(f)\left(\log(r/s), \frac{r - s}{\log(r/s)}\right).\]
By using that the inverse-Fourier transform $\mathcal{F}_{1}^{-1}$ in the first component preserves $\mathscr{S}(\aff)$ together with similar substitutions as previously, we have that $A_{K} \in \mathscr{S}(\mathbb{R}_{+} \times \mathbb{R}_{+})$.
\end{proof}

We will use the notation $\mathscr{S}(L^{2}(\mathbb{R}_{+}))$ for all Schwartz operators on $L^{2}(\mathbb{R}_{+})$.
There is a natural topology on $\mathscr{S}(L^{2}(\mathbb{R}_{+}))$ induced by the semi-norms $\|A_{f}\|_{\alpha, \beta} \coloneqq \|f\|_{\alpha, \beta}$ where $\|\cdot\|_{\alpha, \beta}$ are the semi-norms on $\mathscr{S}(\aff)$ given in \eqref{eq:semi-norms-on-affine}.

\begin{proposition}
\label{Schwartz-result}
    The affine convolution gives a well-defined map \[\mathscr{S}(\aff) \star_{\aff} \mathscr{S}(L^{2}(\mathbb{R}_{+})) \to \mathscr{S}(L^{2}(\mathbb{R}_{+})).\] Moreover, for fixed $A \in \mathscr{S}(L^{2}(\mathbb{R}_{+}))$ the map \[\mathscr{S}(\aff) \ni f \longmapsto f \star_{\aff} A \in \mathscr{S}(L^{2}(\mathbb{R}_{+}))\] is continuous.
\end{proposition}

\begin{proof}
Let $f \in \mathscr{S}(\aff)$ and $A \in \mathscr{S}(L^{2}(\mathbb{R}_{+}))$. Then $A = A_{g}$ for some $g \in \mathscr{S}(\aff)$ and we have by Proposition \ref{function-operator-convolution-quantization} that
\begin{equation}
\label{eq: convolution_equation_in_rapid_decay}
    f \star_{\aff} A = f \star_{\aff} A_{g} = A_{f *_{\aff} g}.
\end{equation}
Hence the first statement reduces to showing that the usual affine group convolution is a well-defined map \[\mathscr{S}(\aff) *_{\aff} \mathscr{S}(\aff) \to \mathscr{S}(\aff).\]
After a change of variables, the question becomes whether the map
\begin{equation}
\label{convolution_after_variable_change}
    (x, u) \longmapsto (f *_{\aff} g)(x,e^{u}) = \int_{\mathbb{R}^{2}}f(y,e^{z})g(x - ye^{u - z}, e^{u - z}) \, dy \, dz
\end{equation}
is an element in $\mathscr{S}(\mathbb{R}^{2})$. It is straightforward to check that \eqref{convolution_after_variable_change} is a smooth function. Moreover, since $f$ and $g$ are both in $\mathscr{S}(\aff)$, it suffices to show that \eqref{convolution_after_variable_change} decays faster than any polynomial towards infinity; we can then iterate the argument to obtain the required decay statements for the derivatives.\par
We claim that
\begin{equation}
\label{estimate_on_g}
    \sup_{x,u}|x|^{k}|u|^{l}|g(x-ye^{u - z}, e^{u-z})| \leq A_{k,l}^{g}\left(1 + |y|\right)^{k}\left(1 + |z|\right)^{l},
\end{equation}
where $A_{k,l}^{g}$ is a constant that depends only on the indices $k,l \in \mathbb{N}_{0}$ and $g \in \mathscr{S}(\aff)$. To show this, we need to individually consider three cases:
\begin{itemize}
    \item Assume that we only take the supremum over $x$ and $u$ satisfying $2|z| \geq |u|$ and $2|y| \geq |x|$. Then clearly \eqref{estimate_on_g} is satisfied with $A_{k,l}^{g} = 2^{k + l}\max|g|$.
    \item Assume that we only take the supremum over $u$ satisfying $2|z| \leq |u|$ and let $x \in \mathbb{R}$ be arbitrary. Then $e^{u - z}$ is outside the interval $[e^{-|u|/2}, e^{|u|/2}]$. Since $g \in \mathscr{S}(\aff)$ the left-hand side of \eqref{estimate_on_g} will eventually decrease when increasing $u$. When $y \leq 0$ the left hand-side of \eqref{estimate_on_g} will also obviously eventually decrease by increasing $x$. When $y > 0$ then any increase of $x$ would necessitate an increase of $u$ on the scale of $u \sim \ln(x)$ to compensate so that the first coordinate in $g$ does not blow up. However, this again forces the second coordinate to grow on the scale of $x$ and we would again, due to $g \in \mathscr{S}(\aff)$, have that the left hand-side of \eqref{estimate_on_g} would eventually decrease.
    \item Finally, we can consider taking the supremum over $x$ and $u$ satisfying $2|z| \geq |u|$ and $2|y| \leq |x|$. As this case uses similar arguments as above, we leave the straightforward verification to the reader.
\end{itemize}
Using \eqref{estimate_on_g} we have that 
\begin{equation}
\label{eq: bound_rapid_decay}
    \sup_{x,u}|x^{k}u^{l}(f *_{\aff}g)(x, e^{u})| \leq A_{k,l}^{g}\int_{\mathbb{R}^{2}}|f(y, e^{z})|\left(1 + |y|\right)^{k}\left(1 + |z|\right)^{l} \, dy \, dz < \infty,
\end{equation}
where the last inequality follows from that $f \in \mathscr{S}(\aff)$. Finally, the continuity of the map $f \mapsto f \star_\aff A$ follows from \eqref{eq: convolution_equation_in_rapid_decay} and \eqref{eq: bound_rapid_decay}.
\end{proof}

\begin{remark}
Notice that the proof of Proposition \ref{Schwartz-result} shows that affine convolution between $f, g \in \mathscr{S}(\aff)$ satisfies $f *_{\aff} g \in \mathscr{S}(\aff)$. This fact, together with Proposition \ref{prop: equallity of spaces}, strengthens the claim that $\mathscr{S}(\aff)$ is the correct definition for Schwartz functions on the group $\aff$.
\end{remark}
The main result in this section is Theorem \ref{thm:quantization_through_convolution} presented below. To state the result rigorously, we first need to make sense of the convolution between Schwartz functions $g \in \mathscr{S}(\aff)$ and the affine parity operator $P_{\aff}$. As motivation for our definition we will use the following computation: Let $S,T \in \mathcal{S}_{2}$ with affine Weyl symbols $f_S, f_{T} \in L^2_r(\aff)$. Fix $g\in \mathscr{S}(\aff)$ and consider the affine Weyl symbol $f_{g\star_\aff S}$ corresponding to the convolution $g\star_\aff S$. Then 
\begin{align*}
	\langle f_{g\star_\aff S},f_T \rangle_{L_{r}^{2}(\aff)} & = \langle g\star_\aff S,T \rangle_{\mathcal{S}_{2}} 
	\\ & = 
	\left\langle S,\int_\aff \overline{g(x,a)} U(-x,a)TU(-x,a)^* \, \frac{dx \, da}{a} \right\rangle_{\mathcal{S}_{2}} 
	\\ & = 
	\left\langle f_S,\int_\aff \overline{g(x,a)} R_{(x,a)}f_T \, \frac{dx \, da}{a} \right\rangle_{L_{r}^{2}(\aff)}.
\end{align*}
With this motivation in mind we get the following definition.
\begin{definition}
\label{def: convolution_tempered}
Let $S:\mathscr{S}(\mathbb{R}_+)\to \mathscr{S}'(\mathbb{R}_+)$ be the operator with affine Weyl symbol $f_S \in \mathscr{S}'(\aff)$ and let $g\in \mathscr{S}(\aff)$. Then $g\star_\aff S$ is defined by its Weyl symbol $f_{g\star_\aff S}\in \mathscr{S}'(\aff)$ satisfying
\begin{equation*}
	\langle f_{g\star_\aff S},h \rangle_{\mathscr{S}',\mathscr{S}} \coloneqq \left\langle f_S, \int_\aff \overline{g(x,a)}R_{(x,a)}h \, \frac{dx \, da}{a} \right\rangle_{\mathscr{S}',\mathscr{S}},
\end{equation*}
for all $h\in \mathscr{S}(\aff)$.
\end{definition}
Recall that the injectivity in Lemma \ref{lemma_swartz} ensures that the operator $S$ in Definition \ref{def: convolution_tempered} is well-defined. The argument to show $f_{g\star_\aff S}\in \mathscr{S}'(\aff)$ is similar to the one presented in Proposition \ref{Schwartz-result}. Hence $g \star_{\aff} S$ is well-defined.

\begin{remark}
We could similarly have defined $S\star_\aff A_f$ for $S\in \mathscr{S}(L^{2}(\mathbb{R}_{+}))$ and $f\in \mathscr{S}'(\aff)$ by using Proposition \ref{prop:convolutionasweyl}. For brevity, we restrict ourselves in the next theorem to the case where $S=\phi\otimes \psi$ for $\psi,\phi\in \mathscr{S}(\aff)$. In this case, we can extend Lemma \ref{lem: convolution_of_rank_one} and define \[(\phi \otimes \psi)\star_\aff A_f \coloneqq  \langle A_fU(-x,a)\psi,U(-x,a)\phi \rangle_{\mathscr{S}',\mathscr{S}}.\]
\end{remark}

We can now finally state the main theorem in this section.

\begin{theorem}
\label{thm:quantization_through_convolution}
The affine Weyl quantization $A_{g}$ of $g \in \mathscr{S}(\mathrm{Aff})$ can be written as 
\[A_{g} = g \star_{\mathrm{Aff}} P_{\mathrm{Aff}},\]
where $P_{\mathrm{Aff}}$ is the affine parity operator.
Moreover, for $\psi, \phi \in \mathscr{S}(\mathbb{R}_{+})$ we have that the affine Weyl symbol $W_{\mathrm{Aff}}^{\psi,\phi}$ of the rank-one operator $\psi \otimes \phi$ can be written as
\[W_{\mathrm{Aff}}^{\psi,\phi} = (\psi \otimes \phi) \star_{\mathrm{Aff}} P_{\mathrm{Aff}}.\]
\end{theorem}

\begin{proof}
Recall that the affine parity operator $P_{\mathrm{Aff}}$ is the affine Weyl quantization of the point measure $\delta_{(0,1)} \in \mathscr{S}'(\aff)$. As such, the convolution $g \star_{\aff} P_{\mathrm{Aff}}$ is well-defined with the interpretation given in Definition \ref{def: convolution_tempered}. The affine Weyl symbol $f_{g \star_{\aff} P_{\mathrm{Aff}}}$ of $g \star_{\aff} P_{\mathrm{Aff}}$ is acting on $h \in \mathscr{S}(\aff)$ by 
		\begin{align*}
			\langle f_{g \star_{\aff} P_{\mathrm{Aff}}},h \rangle_{\mathscr{S}',\mathscr{S}} &\coloneqq \left\langle \delta_{(0,1)}, \int_\aff \overline{g(x,a)}R_{(x,a)}h \, \frac{dx \, da}{a} \right\rangle_{\mathscr{S}',\mathscr{S}} \\
			&= \overline{\int_\aff \overline{g(x,a)}h((0,1)\cdot (x,a)) \, \frac{dx \, da}{a}} \\
			&=  \int_\aff g(x,a)\overline{h(x,a)} \, \frac{dx \, da}{a} \\
			&= \langle g,h \rangle_{L^2_r(\aff)}.
	\end{align*}
	Since $\mathscr{S}(\aff) \subset L_{r}^{2}(\aff)$ is dense, we can conclude that $f_{g \star_{\aff} P_{\mathrm{Aff}}} = g$ and thus $A_{g} = g \star_{\mathrm{Aff}} P_{\mathrm{Aff}}$. \par
	For the second statement, we get that 
\begin{align*}
    \left((\psi \otimes \phi) \star_\aff P_{\mathrm{Aff}}\right)(x,a)&=\langle P_{\mathrm{Aff}}U(-x,a)\psi,U(-x,a)\phi \rangle_{\mathscr{S}',\mathscr{S}} \\
    &= \langle R_{\mathrm{Aff}}(x,a)\psi,\phi \rangle_{\mathscr{S}',\mathscr{S}} \\
    &= W^{\psi,\phi}_\aff(x,a). \qedhere
\end{align*}
\end{proof}

\section{Operator Admissibility}
\label{sec: Operator Admissibility}

For operator convolutions on the Heisenberg group, we have from \eqref{eq:wernerintegral} the important integral relation \[\int_{\mathbb{R}^{2n}}S\star T(z) \, dz=\tr(S)\tr(T).\] A similar formula for the integral of operator convolutions will not hold generally in the affine setting. We therefore search for a class of operators where such a relation does hold: the \textit{admissible operators}. As a first step, we recall the notion of \textit{admissible functions}. 

\begin{definition} \label{def:admissiblefunction}
We say that $\psi \in L^{2}(\mathbb{R}_{+})$ is \textit{admissible} if \[\int_{0}^{\infty}\frac{|\psi(r)|^{2}}{r} \, \frac{dr}{r} < \infty.\]
\end{definition}

This definition of admissibility is motivated by the theorem of Duflo and Moore \cite{duflo1976}, see also \cite{grossmann1985general}. The \textit{Duflo-Moore operator} $\mathcal{D}^{-1}$ in our setting is formally given by
\[\mathcal{D}^{-1}\psi(r) \coloneqq \frac{\psi(r)}{\sqrt{r}}.\]
It is clear that the Duflo-Moore operator $\mathcal{D}^{-1}$ is a densely defined, self-adjoint positive operator on $L^{2}(\mathbb{R}_{+})$ with a densely defined inverse, namely \[\mathcal{D}\psi(r) \coloneqq \sqrt{r} \psi(r).\]
Hence a function $\psi\in L^2(\mathbb{R}_+)$ is admissible if and only if $\mathcal{D}^{-1}\psi\in L^2(\mathbb{R}_+).$
We will on several occasions use the commutation relations
\begin{equation} \label{eq:comrelD}
    \mathcal{D}U(x,a)=\sqrt{\frac{1}{a}} U(x,a)\mathcal{D},\quad     U(x,a)^*\mathcal{D}^{-1}=\sqrt{a}\mathcal{D}^{-1}U(x,a)^*,\quad (x,a)\in \aff.
\end{equation}

The following orthogonality relation is a trivial reformulation of the classic orthogonality relations for wavelets, see for instance \cite{grossmann1986examples}.
\begin{proposition} \label{prop:admissible}
Let $\phi,\psi,\xi, \eta \in L^2(\mathbb{R}_+)$ and assume that $\psi$ and $\eta$ are admissible. Then
\begin{equation*}
    \int_\aff \langle \phi,U(-x,a)^*\psi \rangle_{L^2(\mathbb{R}_+)} \overline{\langle \xi,U(-x,a)^* \eta \rangle_{L^2(\mathbb{R}_+)}} \, \frac{dx\, da}{a}=\langle \phi,\xi \rangle_{L^2(\mathbb{R}_+)} \langle \mathcal{D}^{-1}\eta  , \mathcal{D}^{-1} \psi \rangle_{L^2(\mathbb{R}_+)}.
\end{equation*}
In particular, we have
\begin{equation*}
    \int_\aff \langle \phi,U(-x,a)^*\psi \rangle_{L^2(\mathbb{R}_+)} \overline{\langle \xi,U(-x,a)^* \psi \rangle_{L^2(\mathbb{R}_+)}} \, \frac{dx \, da}{a}=\langle \phi,\xi \rangle_{L^2(\mathbb{R}_+)}  \|\mathcal{D}^{-1}\psi\|_{L^2(\mathbb{R}_+)}^{2}.
\end{equation*}
\end{proposition}

\begin{remark}
By Proposition \ref{prop:admissible},
admissibility of $\psi \in L^{2}(\mathbb{R}_{+})$ is equivalent to the condition
\[
	\int_\aff |\langle \psi,U(-x,a)^* \psi \rangle_{L^2(\mathbb{R}_+)}|^2 \, \frac{dx \, da}{a}<\infty.
\]
\end{remark}

\subsection{Admissibility for Operators}
\label{sec: Admissibility for operators}

Our goal is now to extend the notion of admissibility to bounded operators on $L^2(\mathbb{R}_+)$, with the aim of obtaining a class of operators where a formula for the integral of operator convolutions similar to \eqref{eq:wernerintegral} holds. We will often use that any compact operator $S$ on $L^{2}(\mathbb{R}_{+})$ has a \textit{singular value decomposition}
\begin{equation}\label{eq:singular_value}
	S=\sum_{n=1}^N s_n\xi_n \otimes \eta_n, \qquad N\in \mathbb{N}\cup \{\infty\},
\end{equation}
where $\{\xi_n\}_{n=1}^N$ and $\{\eta_n\}_{n=1}^N$ are orthonormal sets in $L^2(\mathbb{R}_+)$. The \textit{singular values} $\{s_n\}_{n=1}^N$ with $s_n>0$ will converge to zero when $N = \infty$. If $S$ is a trace-class operator we have $\{s_n\}_{n=1}^N \in \ell^1(\mathbb{N})$ with $\|S\|_{\mathcal{S}_1}=\|s_n\|_{\ell^1}$.
Since the admissible functions in $L^2(\mathbb{R}_+)$ form a dense subspace, we can always find an orthonormal basis consisting of admissible functions.

The next result concerns bounded operators $\mathcal{D}S\mathcal{D}$ for a trace-class operator $S$. To be precise, this means that we assume that $S$ maps $\mathrm{dom}(\mathcal{D}^{-1})$ into $\mathrm{dom}(\mathcal{D})$, and that the operator $\mathcal{D}S\mathcal{D}$ defined on $\mathrm{dom}(\mathcal{D})$  extends to a bounded operator.

\begin{theorem} \label{thm:operator_orthogonality_relation}
	Let $S\in \mathcal{S}_1$ satisfy that $\mathcal{D}S\mathcal{D}\in \mathcal{L}(L^2(\mathbb{R}_+))$. For any $T\in \mathcal{S}_1$ we have that $T\star_\aff \mathcal{D}S\mathcal{D}\in L^1_r(\aff)$ with
	\begin{equation*}
		\|T\star_\aff \mathcal{D}S\mathcal{D}\|_{L^1_r(\aff)} \leq \|S\|_{\mathcal{S}_1} \|T\|_{\mathcal{S}_1},
	\end{equation*}
	and
	\begin{equation} \label{eq:orthogonalityfinal}
		\int_\aff T\star_{\aff} \mathcal{D}S\mathcal{D}(x,a) \, \frac{dx\, da}{a} = \tr(T)\tr(S).
	\end{equation}
\end{theorem}

\begin{proof}
We divide the proof into three steps.\newline
\textbf{Step 1:} We first assume that $T=\psi\otimes \phi$ for $\psi,\phi\in \mathrm{dom}(\mathcal{D})$. Recall that $S$ can be written in the form \eqref{eq:singular_value}.
 From Lemma \ref{lem: convolution_of_rank_one} and \eqref{eq:comrelD} we find that 
	\begin{align*}
		T\star_\aff \mathcal{D}S\mathcal{D}(x,a)  &= \langle S\mathcal{D} U(-x,a)\psi,\mathcal{D}U(-x,a)\phi \rangle_{L^2(\mathbb{R}_+)} \\
		&= \frac{1}{a} \langle SU(-x,a)\mathcal{D} \psi, U(-x,a)\mathcal{D} \phi \rangle_{L^2(\mathbb{R}_+)} \\
		&=  \sum_{n=1}^N s_n \frac{1}{a} \langle U(-x,a)\mathcal{D}\psi,\eta_n \rangle_{L^2(\mathbb{R}_+)} \langle \xi_n,U(-x,a)\mathcal{D}\phi \rangle_{L^2(\mathbb{R}_+)}.
	\end{align*}
Integrating with respect to the right Haar measure and using that $(x,a)\mapsto (x,a)^{-1}$ interchanges left and right Haar measure, we get
\begin{align*}
	&\int_\aff |\langle U(-x,a)\mathcal{D}\psi,\eta_n \rangle_{L^2(\mathbb{R}_+)} \langle \xi_n,U(-x,a)\mathcal{D}\phi \rangle_{L^2(\mathbb{R}_+)}| \frac{1}{a} \, \frac{dx\, da}{a} \\
	&= \int_\aff |\langle U(-x,a)^*\mathcal{D}\psi,\eta_n \rangle_{L^2(\mathbb{R}_+)} \langle \xi_n,U(-x,a)^*\mathcal{D}\phi \rangle_{L^2(\mathbb{R}_+)}| \, \frac{dx\, da}{a} \\
	&\leq \left(\int_\aff |\langle U(-x,a)^*\mathcal{D}\psi,\eta_n \rangle_{L^2(\mathbb{R}_+)} |^2 \, \frac{dx\, da}{a}  \right)^{1/2} \left(\int_\aff |\langle \xi_n,U(-x,a)^*\mathcal{D}\phi \rangle_{L^2(\mathbb{R}_+)}|^2 \, \frac{dx\, da}{a} \right)^{1/2} \\
	&= \|\psi\|_{L^2(\mathbb{R}_+)} \|\phi\|_{L^2(\mathbb{R}_+)},
\end{align*}
where the last line uses Proposition \ref{prop:admissible}.
It follows that the sum in the expression for $T\star_\aff \mathcal{D}S\mathcal{D}(x,a)$ converges absolutely in $L^1_r(\aff)$ with
\begin{equation*}
	\|T\star_\aff \mathcal{D}S\mathcal{D}\|_{L^1_r(\aff)}\leq \left(\sum_{n=1}^N s_n \right)\|\psi\|_{L^2(\mathbb{R}_+)} \|\phi\|_{L^2(\mathbb{R}_+)} = \|S\|_{\mathcal{S}_1}\|T\|_{\mathcal{S}_1}.
\end{equation*}
Equation \eqref{eq:orthogonalityfinal} follows in a similar way by integrating the sum expressing $T\star_\aff \mathcal{D}S\mathcal{D}$ and using Proposition \ref{prop:admissible}.
\newline
\textbf{Step 2:} We now assume that $T=\psi\otimes \phi$ for arbitrary $\psi,\phi\in L^2(\mathbb{R}_+)$.
Pick sequences $\{\psi_n\}_{n=1}^\infty,\{\phi_n\}_{n=1}^\infty$ in $\mathrm{dom}(\mathcal{D})$ converging to $\psi$ and $\phi$, respectively, and let $T_n=\psi_n \otimes \phi_n$. It is straightforward to check that $T_n$ converges to $T$ in $\mathcal{S}_1$. By \eqref{eq:conv_op_op_duality} this implies that $T_n \star_\aff \mathcal{D}S\mathcal{D}$ converges uniformly to $T \star_\aff \mathcal{D}S\mathcal{D}$. On the other hand, $T_n \star_\aff \mathcal{D}S\mathcal{D}$ is a Cauchy sequence in $L^1_r(\aff)$: for $m,n \in \mathbb{N}$ we find by Step 1 that
\begin{align*}
	\| T_n \star_\aff \mathcal{D}S\mathcal{D}-T_m \star_\aff \mathcal{D}S\mathcal{D} \|_{L^1_r(\aff)} &\leq \| \psi_n\otimes \phi_n \star_\aff \mathcal{D}S\mathcal{D}-\psi_m\otimes \phi_n \star_\aff \mathcal{D}S\mathcal{D} \|_{L^1_r(\aff)} \\
	&\quad+ \| \psi_m\otimes \phi_n \star_\aff \mathcal{D}S\mathcal{D}-\psi_m\otimes \phi_m \star_\aff \mathcal{D}S\mathcal{D} \|_{L^1_r(\aff)} \\
	&=\| (\psi_n-\psi_m)\otimes \phi_n \star_\aff \mathcal{D}S\mathcal{D} \|_{L^1_r(\aff)} \\
	&\quad+ \| \psi_m\otimes (\phi_n-\phi_m) \star_\aff \mathcal{D}S\mathcal{D}\|_{L^1_r(\aff)}\\
	&\leq \|S\|_{\mathcal{S}_1}  \|\psi_n-\psi_m\|_{L^2(\mathbb{R}_+)} \|\phi_n\|_{L^2(\mathbb{R}_+)} \\
	&\quad+ \|S\|_{\mathcal{S}_1} \|\psi_m\|_{L^2(\mathbb{R}_+)} \|\phi_m-\phi_n\|_{L^2(\mathbb{R}_+)} 
\end{align*}
which clearly goes to zero as $m,n \to \infty$. This means that $T_n \star_\aff \mathcal{D}S\mathcal{D}$ converges in $L^1_r(\aff)$, and the limit must be $T\star_\aff \mathcal{D}S\mathcal{D}$ as we already know that $T_n \star_\aff \mathcal{D}S\mathcal{D}$ converges uniformly to this function. In particular, this implies 
\begin{align*}
	\|T\star_\aff \mathcal{D}S\mathcal{D}\|_{L^1_r(\aff)} &= \lim_{n\to \infty} \|T_n\star_\aff \mathcal{D}S\mathcal{D}\|_{L^1_r(\aff)} \\
	&\leq  \lim_{n\to \infty} \|\psi_n\|_{L^2(\mathbb{R}_+)} \|\phi_n\|_{L^2(\mathbb{R}_+)}  \|S\|_{\mathcal{S}_1} \\
	&= \|\psi\|_{L^2(\mathbb{R}_+)} \|\phi\|_{L^2(\mathbb{R}_+)} \|S\|_{\mathcal{S}_1}. 
\end{align*}
Equation \eqref{eq:orthogonalityfinal} also follows by taking the limit of $\int_\aff T_n \star_\aff \mathcal{D}S\mathcal{D}(x,a)\, \frac{dx\, da}{a}$.
\newline
\textbf{Step 3:} We now assume that $T\in \mathcal{S}_1$.
Consider the singular value decomposition of $T$ given by
\begin{equation*}
	T=\sum_{m=1}^M t_m \psi_m \otimes \phi_m
\end{equation*}
for $M\in \mathbb{N}\cup \{\infty\}$. By \eqref{eq:conv_op_op_duality} we have, with uniform convergence of the sum, that
\begin{equation} \label{eq:orthoproof}
	T\star_\aff \mathcal{D}S\mathcal{D}=\sum_{m=1}^M t_m \psi_m \otimes \phi_m \star_\aff \mathcal{D}S\mathcal{D}.
\end{equation}
Notice that Step 2 implies that the convergence is also in $L^1_r(\aff)$, since 
\begin{align*}
	\sum_{m=1}^M t_m \|\psi_m \otimes \phi_m \star_\aff \mathcal{D}S\mathcal{D}\|_{L^1_r(\aff)}&\leq \sum_{m=1}^M t_m \|\psi_m\|_{L^2(\mathbb{R}_+)} \|\phi_m\|_{L^2(\mathbb{R}_+)} \|S\|_{\mathcal{S}_1} =\|T\|_{\mathcal{S}_1} \|S\|_{\mathcal{S}_1}.
\end{align*}
In particular, $T\star_\aff \mathcal{D}S\mathcal{D}\in L^1_r(\aff)$. Finally, \eqref{eq:orthogonalityfinal} follows by integrating \eqref{eq:orthoproof} and using that the sum converges in $L^1_r(\aff)$ and Step 2.
\end{proof}

The integral relation \eqref{eq:orthogonalityfinal} is somewhat artificial in the sense that it introduces $\mathcal{D}$ in the integrand. We will typically be interested in the integral of $T\star_\aff S$, not of $T\star_\aff \mathcal{D}S\mathcal{D}$. This motivates the following definition.

\begin{definition}\label{def:admissibility}
Let $S$ be a non-zero bounded operator on $L^{2}(\mathbb{R}_{+})$ that maps $\mathrm{dom}(\mathcal{D})$ into $\mathrm{dom}(\mathcal{D}^{-1})$. We say that $S$ is \textit{admissible} if the composition $\mathcal{D}^{-1}S\mathcal{D}^{-1}$ is bounded on  $\mathrm{dom}(\mathcal{D}^{-1})$ and extends to a trace-class operator $\mathcal{D}^{-1}S\mathcal{D}^{-1} \in \mathcal{S}_{1}$. 
\end{definition}

Assume now that $S$ is admissible, and define $R\coloneqq \mathcal{D}^{-1}S\mathcal{D}^{-1}$. Clearly $R$ maps $\mathrm{dom}(\mathcal{D}^{-1})$ into $\mathrm{dom}(\mathcal{D})$ as we assume that $S$ maps $\mathrm{dom}(\mathcal{D})$ into $\mathrm{dom}(\mathcal{D}^{-1})$. The following corollary is therefore immediate from Theorem \ref{thm:operator_orthogonality_relation}. We also note that it extends \cite[Cor.~1]{kiukas2006} to non-positive, non-compact operators.

\begin{corollary} \label{cor:admissible_condition}
	Let $S\in \mathcal{L}(L^2(\mathbb{R}_+))$ be an admissible operator. For any $T\in \mathcal{S}_1$ we have that $T\star_\aff S\in L^1_r(\aff)$ with
	\begin{equation*}
		\|T\star_\aff S\|_{L^1_r(\aff)} \leq \|\mathcal{D}^{-1}S\mathcal{D}^{-1}\|_{\mathcal{S}_1} \|T\|_{\mathcal{S}_1},
	\end{equation*}
	and
	\begin{equation*} 
		\int_\aff T\star_{\aff} S(x,a) \, \frac{dx\, da}{a} = \tr(T)\tr(\mathcal{D}^{-1}S\mathcal{D}^{-1}).
	\end{equation*}
\end{corollary}

\begin{example}
	A rank-one operator $S=\eta \otimes \xi$ for non-zero $\eta,\xi$ is an admissible operator if and only if $\eta,\xi \in L^{2}(\mathbb{R}_{+})$ are admissible functions. Requiring that $S$ maps $\mathrm{dom}(\mathcal{D})$ into $\mathrm{dom}(\mathcal{D}^{-1})$ clearly implies that $\eta \in \mathrm{dom}(\mathcal{D}^{-1})$, i.e.\ $\eta$ is admissible. For $\mathcal{D}^{-1}S\mathcal{D}^{-1}$ to be trace-class, the map 
	\[\psi \mapsto \|\mathcal{D}^{-1}S\mathcal{D}^{-1}\psi\|_{L^2{(\mathbb{R}_+)}}=|\langle \mathcal{D}^{-1}\psi,\xi  \rangle_{L^2(\mathbb{R}_+)}|\cdot  \|\mathcal{D}^{-1}\eta\|_{L^2(\mathbb{R}_+)}, \qquad \psi \in \mathrm{dom}(\mathcal{D}^{-1}), \]
	must at least be bounded for $\|\psi\|_{L^2(\mathbb{R}_+)}\leq 1$. This is bounded if and only if 
	\[\psi \mapsto \langle \mathcal{D}^{-1}\psi,\xi  \rangle_{L^2(\mathbb{R}_+)}\]
	is bounded, which is precisely the condition that $\xi \in \mathrm{dom}\left(\left(\mathcal{D}^{-1}\right)^*\right)=\mathrm{dom}(\mathcal{D}^{-1})$.
	Hence our notion of admissibility for operators naturally extends the classical function admissibility.
	In the case of rank-one operators, it follows from Lemma \ref{lem: convolution_of_rank_one} and the computation 
	\begin{equation*}	\tr(\mathcal{D}^{-1}(\eta \otimes \xi)\mathcal{D}^{-1})=\langle \mathcal{D}^{-1}\eta,\mathcal{D}^{-1}\xi \rangle_{L^2(\mathbb{R}_+)}
	\end{equation*} 
	that Corollary  \ref{cor:admissible_condition} reduces to Proposition \ref{prop:admissible}.
		\end{example}
When both $S$ and $T$ are admissible trace-class operators, their convolution $T\star_\aff S$ behaves well with respect to both the left and right Haar measures.
\begin{corollary}
Let $S$ and $T$ be admissible trace-class operators on $L^{2}(\mathbb{R}_{+})$. Then the convolution $T\star_\aff S$ satisfies $T\star_\aff S \in L^1_r(\aff)\cap L^1_l(\aff)$ and 
\begin{align*}
    \int_\aff T\star_\aff S(x,a)\, \frac{dx\, da}{a}=\tr(T) \tr(\mathcal{D}^{-1}S\mathcal{D}^{-1}), \\
    \int_\aff T\star_\aff S(x,a)\, \frac{dx\, da}{a^2}=\tr(S) \tr(\mathcal{D}^{-1}T\mathcal{D}^{-1}).
\end{align*}
\end{corollary}
\begin{proof}
The first equation and the claim that $T\star_\aff S \in  L^1_r(\aff)$ is Corollary \ref{cor:admissible_condition}. The second equation and the claim that $T\star_\aff S \in  L^1_l(\aff)$ follows since
\begin{equation*}
    T\star_\aff S(x,a) = S\star_\aff T((x,a)^{-1}).\qedhere
\end{equation*} 
\end{proof}

We now turn to the case where $S$ is a positive compact operator. We first note that admissibility in this case becomes a statement about the eigenvectors and eigenvalues of $S$.

\begin{proposition} \label{prop:specdec}
	Let $S$ be a non-zero positive compact operator with spectral decomposition 
	\begin{equation*}
		S=\sum_{n=1}^N s_n \xi_n \otimes \xi_n
	\end{equation*}
	for $N\in \mathbb{N}\cup \{\infty\}$. Then $S$ is admissible if and only each $\xi_n$ is admissible and
	\[\sum_{n=1}^N s_n\|\mathcal{D}^{-1}\xi_n\|_{L^2(\mathbb{R}_+)}^2<\infty.\]
\end{proposition}

\begin{proof}
We first assume that $S$ is admissible. By linearity and Lemma \ref{lem: convolution_of_rank_one} we get for $\xi \in L^{2}(\mathbb{R}_{+})$ with $\|\xi\|_{L^{2}(\mathbb{R}_{+})} = 1$ that
\begin{equation}\label{eq:convolution_rank_one_in_proof}
    \xi \otimes \xi \star_\aff S(x,a) = \sum_{n=1}^N s_n  |\langle\xi,U(-x,a)^*\xi_n \rangle_{L^{2}(\mathbb{R}_{+})}|^2.
\end{equation}
Integrating \eqref{eq:convolution_rank_one_in_proof} using the monotone convergence theorem and Proposition \ref{prop:admissible}, we obtain
\begin{equation*}
    \int_\aff  \xi \otimes \xi \star_\aff S(x,a)\, \frac{dx\, da}{a} = \sum_{n=1}^N s_n \|\mathcal{D}^{-1}\xi_n\|_{L^2(\mathbb{R}_+)}^2.
\end{equation*}
The claim now follows from Corollary \ref{cor:admissible_condition}.
	
	For the converse, it is clear by the assumption that the operator 
	\begin{equation} \label{eq:singularvalueadmissibleproof}
	    \sum_{n=1}^N s_n (\mathcal{D}^{-1}\xi_n) \otimes (\mathcal{D}^{-1}\xi_n)
	\end{equation}
	is a trace-class operator. It only remains to show that $S$ maps $\mathrm{dom}(\mathcal{D})$ into $\mathrm{dom}(\mathcal{D}^{-1})$ and that $\mathcal{D}^{-1}S\mathcal{D}^{-1}$ is given by \eqref{eq:singularvalueadmissibleproof}. This is easily shown when $N$ is finite, so we do the proof for $N=\infty$. \par 
	The partial sums for $\psi \in L^2(\mathbb{R}_+)$ are denoted by
	\begin{equation*}
		(S\psi)_M\coloneqq \sum_{n=1}^M s_n \langle \psi,\xi_n \rangle_{L^2(\mathbb{R}_+)} \xi_n,
	\end{equation*}
	  and converge in the sense that $(S\psi)_M\to S\psi$ as $M\to \infty$. Furthermore, it is clear that $(S\psi)_M$ is in the domain of $\mathcal{D}^{-1}$ for each $M$ as each $\xi_n$ is admissible. We also have that 
	\begin{equation*}
		\mathcal{D}^{-1}(S\psi)_M=\sum_{n=1}^M s_n \langle \psi,\xi_n \rangle_{L^2(\mathbb{R}_+)} \mathcal{D}^{-1}\xi_n.
	\end{equation*}
	The sequence of partial sums $\mathcal{D}^{-1}(S\psi)_M$ also converges in $L^{2}(\mathbb{R}_{+})$, since by using H\"older's inequality and Bessel's inequality we obtain
	\begin{align*}
		\sum_{n=1}^\infty s_n |\langle \psi,\xi_n \rangle_{L^2(\mathbb{R}_+)}| \|\mathcal{D}^{-1}\xi_n\|_{L^2(\mathbb{R}_+)} &\leq \left(\sum_{n=1}^\infty |\langle \psi,\xi_n \rangle_{L^2(\mathbb{R}_+)}|^2\right)^{1/2} \left(\sum_{n=1}^\infty s_n^2 \|\mathcal{D}^{-1}\xi_n\|_{L^2(\mathbb{R}_+)}^2\right)^{1/2} \\
		&\lesssim \|\psi\|_{L^2(\mathbb{R}_+)} \left(\sum_{n=1}^\infty s_n \|\mathcal{D}^{-1}\xi_n\|_{L^2(\mathbb{R}_+)}^2\right)^{1/2}.
	\end{align*}
	Since $\mathcal{D}^{-1}$ is a closed operator, we get that $S\psi$ belongs to the domain of $\mathcal{D}^{-1}$ and \ \[\mathcal{D}^{-1}S\psi=\sum_{n=1}^\infty s_n \langle \psi,\xi_n \rangle_{L^2(\mathbb{R}_+)} \mathcal{D}^{-1}\xi_n.\] \par 
	For any $\phi \in \mathrm{dom}(\mathcal{D}^{-1})$, we have that
	\begin{equation*}
		\mathcal{D}^{-1}S\mathcal{D}^{-1}\phi = \sum_{n=1}^\infty s_n \langle \mathcal{D}^{-1} \phi,\xi_n \rangle_{L^2(\mathbb{R}_+)} \mathcal{D}^{-1}\xi_n 
		=\sum_{n=1}^\infty s_n \langle \phi, \mathcal{D}^{-1}\xi_n \rangle_{L^2(\mathbb{R}_+)} \mathcal{D}^{-1}\xi_n,
	\end{equation*}
	so $\mathcal{D}^{-1}S\mathcal{D}^{-1}$ agrees with \eqref{eq:singularvalueadmissibleproof} on this dense subspace. In fact, they agree on all of $L^{2}(\mathbb{R}_{+})$ since
	\begin{equation*}
		\|\mathcal{D}^{-1}S\mathcal{D}^{-1}\phi\|_{L^2(\mathbb{R}_+)}\leq \|\phi\|_{L^2(\mathbb{R}_+)} \sum_{n=1}^\infty s_n \|\mathcal{D}^{-1}\xi_n\|_{L^2(\mathbb{R}_+)}^2,
	\end{equation*}
	shows that $\mathcal{D}^{-1}S\mathcal{D}^{-1}$ extends to a bounded operator.
\end{proof}

As a consequence of Proposition \ref{prop:specdec}, we obtain a compact reformulation of admissibility for positive trace-class operators.
\begin{corollary}\label{corr:admissibility_check}
	Let $T$ be a non-zero positive trace-class operator on $L^{2}(\mathbb{R}_{+})$, and let $S$ be a non-zero positive compact operator. If 
	\begin{equation*}
		\int_\aff T\star_\aff S(x,a) \, \frac{dx \, da}{a}<\infty,
	\end{equation*}
	then $S$ is admissible with
	\begin{equation*}
	    \tr(\mathcal{D}^{-1}S\mathcal{D}^{-1})=\frac{1}{\tr(T)}	\int_\aff T\star_\aff S(x,a) \, \frac{dx \, da}{a}.
	\end{equation*}
	In particular, if $S$ is a non-zero, positive trace-class operator, then $S$ is admissible if and only if $S\star_\aff S\in L_{r}^{1}(\aff)$.
\end{corollary}
\begin{proof}
Let \[S=\sum_{n=1}^N s_n \xi_n \otimes \xi_n\] be the spectral decomposition of $S$. An argument similar to the one giving in the proof of Proposition \ref{prop:specdec} shows that \[\int_\aff T\star_\aff S(x,a) \, \frac{dx \, da}{a}= \tr(T) \sum_{n=1}^N s_n \|\mathcal{D}^{-1}\xi_n\|_{L^2(\mathbb{R}_+)}^2.\] The claims now follow immediately from Proposition \ref{prop:specdec}.
\end{proof}

\subsection{Admissible Operators from Laguerre Functions}
\label{sec: Laguerre Connection}

Although we derived several basic properties of admissible operators in Section \ref{sec: Admissibility for operators}, we have not given any way to construct such operators in practice. Our construction is based on the following observation: From Proposition \ref{prop:specdec} we know that if \[S=\sum_{n=1}^\infty s_n \varphi_n \otimes \varphi_n\] is a non-zero positive compact operator with \[\sum_{n=1}^\infty s_n \|\mathcal{D}^{-1}\varphi_n\|_{L^2(\mathbb{R}_+)}^2<\infty,\] then $S$ is admissible. So if we can find an orthonormal basis $\{\varphi_n\}_{n=1}
^\infty$ of admissible functions such that we can control the terms $\|\mathcal{D}^{-1}\varphi_n\|_{L^2(\mathbb{R}_+)}$, then we can construct admissible operators as infinite linear combinations of rank-one operators. It turns out that the Laguerre basis works extremely well in this regard. 

\begin{definition}
\label{def: laguerre}
For fixed $\alpha\in \mathbb{R}_+$ we define the \textit{Laguerre basis} $\left\{\mathcal{L}_{n}^{(\alpha)}\right\}_{n=0}^{\infty}$ for $L^2(\mathbb{R}_+)$ by
\begin{equation*}
    \mathcal{L}_{n}^{(\alpha)}(r)  \coloneqq\sqrt{\frac{n!}{\Gamma(n+\alpha+1)}}r^{\frac{\alpha + 1}{2}}e^{-\frac{r}{2}}L_{n}^{(\alpha)}(r), \qquad n \in \mathbb{N}_{0}, \, r \in \mathbb{R}_{+},
\end{equation*}
where $\Gamma$ denotes the gamma function and $L_{n}^{(\alpha)}$ denotes the \textit{generalized Laguerre polynomials} given by \[L_{n}^{(\alpha)}(r)  \coloneqq\frac{r^{-\alpha}e^{r}}{n!}\frac{d^n}{dr^n}\left(e^{-r}r^{n + \alpha}\right) = \sum_{k = 0}^{n}(-1)^{k}\binom{n + \alpha}{n-k}\frac{r^k}{k!}.\]
\end{definition}
The classical orthogonality relation

\begin{equation}
\label{eq:orthogonality_Laguerre}
    \int_{0}^{\infty}x^{\alpha}e^{-x}L_{n}^{(\alpha)}(x)L_{m}^{(\alpha)}(x) \, dx = \frac{\Gamma(n + \alpha + 1)}{n!}\delta_{n,m},
\end{equation}
for the generalized Laguerre polynomials ensures that the Laguerre bases are orthonormal bases for $L^2(\mathbb{R}_+)$ for any fixed $\alpha \in \mathbb{R}_+$. The following result shows that the Laguerre basis is especially compatible with the Duflo-Moore operator $\mathcal{D}^{-1}$.

\begin{proposition}
    For any $\alpha \in \mathbb{R}_+$ and $n\in \mathbb{N}_0$ we have
    \begin{equation}\label{eq:laguerre_norm}
        \left\|\mathcal{D}^{-1}\mathcal{L}_n^{(\alpha)}\right\|^2_{L^2(\mathbb{R}_+)} = \frac{n!}{\Gamma(n+\alpha+1)}\int_0^\infty e^{-r} r^{\alpha-1} \left(L_n^{(\alpha)}(r)\right)^2 \, dr = \frac{1}{\alpha}.
    \end{equation}
\end{proposition}

\begin{proof}
The first equality in \eqref{eq:laguerre_norm} follows from unwinding the definitions. For the second equality in \eqref{eq:laguerre_norm}, we will use the well-known identity 
\begin{equation*}
    L_{n}^{(\alpha)}(r)=\sum_{j=0}^n L^{(\alpha - 1)}_j(r)
\end{equation*}
together with the orthogonality relation \eqref{eq:orthogonality_Laguerre}. This gives 
\begin{align*}
    \int_0^\infty e^{-r} r^{\alpha-1} \left(L_n^{(\alpha)}(r)\right)^2 \, dr & = \sum_{i,j = 0}^{n}\int_0^\infty e^{-r} r^{\alpha-1} L_{i}^{(\alpha - 1)}(r)L_{j}^{(\alpha - 1)}(r) \, dr \\ & = \sum_{i = 0}^{n} \frac{\Gamma(i + \alpha)}{i!} \\ & = \frac{1}{\alpha}\frac{\Gamma(n + \alpha + 1)}{n!},
\end{align*}
where the last equality follows from a straightforward induction argument.
\end{proof}
The following consequence from Proposition \ref{prop:specdec} shows that we can explicitly construct admissible operators by using the Laguerre basis.
\begin{corollary}
Let $\{s_n\}_{n=0}^\infty \in \ell^{1}(\mathbb{N})$ be a sequence of non-negative numbers and let $\alpha\in \mathbb{R}_+$. Then
\begin{equation*}
    S\coloneqq\sum_{n=0}^\infty s_n \mathcal{L}_{n}^{(\alpha)} \otimes \mathcal{L}_{n}^{(\alpha)}
\end{equation*}
is an admissible operator with \[\tr(\mathcal{D}^{-1}S\mathcal{D}^{-1})=\frac{1}{\alpha}\sum_{n=0}^\infty s_n.\]
\end{corollary}

\begin{remark}
The corollary may be considered a reformulation with slightly different proof of the calculations in  \cite[Section 3.3]{gazeau2016covariant}, where a resolution of the identity operator is constructed from thermal states that are diagonal in the Laguerre basis. We will return to resolutions of the identity operator and the relation to admissibility in Section \ref{sec: Other Covariant Integral Quantizations}.
\end{remark}

\subsection{Connection with Convolutions and Quantizations}
\label{sec: Connection with the Affine Weyl Symbol}

We will now see how admissibility relates to the convolution of a function with an operator. The following result shows that we can use convolutions to generate new admissible operators from a given admissible operator.

\begin{proposition}\label{prop:convolution_with_admissible}
    Let $f\in L_{l}^1(\aff)\cap L_{r}^{1}(\aff)$ be a non-zero positive function. If $S$ is a positive, admissible trace-class operator on $L^{2}(\mathbb{R}_{+})$, then so is $f\star_\aff S$ with \begin{equation*}
        \tr\left(\mathcal{D}^{-1}(f\star_\aff S) \mathcal{D}^{-1}\right)= \int_\aff f(x,a) \, \frac{dx \, da}{a^2}  \tr(\mathcal{D}^{-1} S \mathcal{D}^{-1}).
    \end{equation*}
\end{proposition}

\begin{proof}
It is clear from \eqref{eq:conv_op_fun_bochner} that $f \star_{\aff} S$ is a trace-class operator, and positivity follows from the definition of the convolution $f \star_{\aff} S$. Let $T$ be a non-zero positive trace-class operator on $L^{2}(\mathbb{R}_{+})$. It suffices by Corollary \ref{corr:admissibility_check} to show that \begin{equation*}
    \int_\aff T\star_\aff (f\star_\aff S)(y, b) \, \frac{dy \, db}{b}=\tr(T)  \int_\aff f(x,a) \, \frac{dx \, da}{a^2}  \tr(\mathcal{D}^{-1}S\mathcal{D}^{-1}).
\end{equation*}
We have that 
\begin{align*}
    T\star_\aff (f\star_\aff S)(y,b)&= \tr\left( TU(-y,b)^* \int_\aff f(x,a)U(-x,a)^* S U(-x,a) \, \frac{dx \, da}{a} U(-y,b) \right) \\
    &= \int_\aff f(x,a)\tr(TU((-x,a)\cdot (-y,b))^* S U((-x,a)\cdot(-y,b))  \, \frac{dx \, da}{a} \\
    &= \int_\aff f(x,a) T\star_\aff S((x,a)\cdot (y,b)) \, \frac{dx \, da}{a}.
\end{align*}
We may then use Fubini's theorem, which applies by our assumptions on $f$ and $S$, to show that 
\begin{align*}
    \int_{\aff} T\star_\aff (f\star_\aff S)(y,b) \, \frac{dy \, db}{b}&=\int_\aff f(x,a) \int_{\aff}  T\star_\aff S((x,a)\cdot (y,b)) \, \frac{dy \, db}{b}  \, \frac{dx \, da}{a} \\
    &= \int_\aff f(x,a) \, \frac{dx \, da}{a} \Delta(x,a) \int_{\aff}  T\star_\aff S(y,b) \, \frac{dy \, db}{b}  \\
    &=  \int_\aff f(x,a) \, \frac{dx \, da}{a^2}  \tr(T) \tr\left(\mathcal{D}^{-1}S\mathcal{D}^{-1}\right),
\end{align*}
where we used the admissibility of $S$ and Theorem \ref{cor:admissible_condition} in the last line.
\end{proof}

\begin{remark}
We can give a simple heuristic argument for Proposition \ref{prop:convolution_with_admissible} by ignoring that $\mathcal{D}^{-1}$ is unbounded as follows: We have by using \eqref{eq:comrelD} that \begin{align*}
    \mathcal{D}^{-1} (f\star_\aff S)\mathcal{D}^{-1}
    & =
    \int_\aff f(x,a) \mathcal{D}^{-1} U(-x,a)^*SU(-x,a)\mathcal{D}^{-1} \, \frac{dx \, da}{a}
    \\ & = 
    \int_\aff f(x,a)  U(-x,a)^* \mathcal{D}^{-1} S \mathcal{D}^{-1} U(-x,a) \, \frac{dx \, da}{a^2}.
\end{align*}
Since $\mathcal{D}^{-1}S\mathcal{D}^{-1}$ is a trace-class operator, the integral above is a convergent Bochner integral and we obtain the desired equality.
\end{remark}

\subsection{Admissibility as a Measure of Non-Unimodularity}
\label{sec: Admissibility as a Measure of Non-Unimodularity}
In this section we will delve more into how the non-unimodularity of the affine group affects the affine Weyl quantization. As we will see, both the left and right Haar measures take on an active role in this picture.

\begin{proposition} \label{prop:admissibilityandweylsymbol}
	Let $S$ be an admissible Hilbert-Schmidt operator on $L^{2}(\mathbb{R}_{+})$ such that its affine Weyl symbol $f_S$ satisfies $f_S\in L_{l}^{1}(\aff)$. Then 
	\begin{equation*}
		\tr\left(\mathcal{D}^{-1}S\mathcal{D}^{-1}\right) = \int_\aff f_S(x,a) \, \frac{dx \, da}{a^2}.
	\end{equation*}
\end{proposition}

\begin{proof}
	Let $T=\varphi\otimes \varphi$ for some non-zero $\varphi \in \mathscr{S}(\mathbb{R}_+)$. Then the affine Weyl symbol of $T$ is $f_T=W^\varphi_\aff\in \mathscr{S}(\aff)$. We know by Corollary \ref{cor:admissible_condition} that 
	\begin{equation*}
		\int_{\aff}T\star_\aff S(x,a) \, \frac{dx \, da}{a}=\tr(T)\tr\left(\mathcal{D}^{-1}S\mathcal{D}^{-1}\right).
	\end{equation*}
	On the other hand, Fubini's theorem together with Proposition \ref{prop:convolutionasweyl} allows us to calculate that 
	\begin{align*}
		\int_{\aff}T\star_\aff S(x,a) \, \frac{dx \, da}{a}&=\int_{\aff} f_T\ast_\aff \check{f_S}(x,a) \, \frac{dx \, da}{a} \\
		&= \int_{\aff}f_T(y,b) \int_{\aff} f_S((y,b)(x,a)^{-1}) \, \frac{dx \, da}{a} \, \frac{dy \, db}{b} \\
		&=  \int_{\aff} f_T(y,b) \, \frac{dy \, db}{b} \int_{\aff}  f_S(x,a) \, \frac{dx \, da}{a^2}.
	\end{align*}
	The marginal properties of the affine Wigner distribution \eqref{eq:first_marginal_property} show that
\begin{align*}
	\intA{f_{T}(y,b)}{y}{b} = \|\varphi\|_{L^{2}(\mathbb{R}_{+})}^{2}=\tr(T).
\end{align*}
The claim now follows from combining the calculations we have done.
\end{proof}
\begin{remark}
 Assuming that $T$ is a trace-class operator we have that 
 \begin{equation*}
     \tr(T)=\int_\aff f_T(x,a)\,\frac{dx\, da}{a},
 \end{equation*}
 which follows from a similar proof to the one in Proposition \ref{prop:admissibilityandweylsymbol}. This gives the interesting heuristic interpretation that taking $\mathcal{D}^{-1}T\mathcal{D}^{-1}$ of an operator $T$ coincides with multiplying $f_T$ by $\frac{1}{a}$.
\end{remark}
The following result shows that the affine Wigner distribution satisfies both left and right integrability when more is assumed of the input. This should be compared with the Heisenberg case where the Heisenberg group $\mathbb{H}^{n}$ is unimodular.

\begin{theorem}\label{Wigner_right_left}
   Assume that $\phi,\psi,\mathcal{D}\phi, \mathcal{D}\psi\in L^2(\mathbb{R}_{+})$. Then the affine Wigner distribution satisfies \[W_\aff^{\phi,\psi}\in L^2_{r}(\aff)\cap L^2_{l}(\aff).\]
\end{theorem}
\begin{proof}
We already know that $W^{\phi,\psi}_\aff$ is in $L^{2}_{r}(\aff)$ by the orthogonality relations \eqref{affine_orthogonality_relation_new}. Using the definition of the affine Wigner distribution and Plancherel's theorem, we have that
\begin{align*}
    \|W_{\aff}^{\phi,\psi}\|_{L_{l}^{2}(\aff)} & = \int_\aff \left|\phi(a\lambda(x))|^{2}|\psi(a\lambda(-x))\right|^2 \, \frac{dx \, da}{a^2}
    \\ & = \int_0^\infty\int_0^\infty |\phi(v)|^2|\psi(w)|^2\frac{v-w}{\log(v/w)}\frac{dw\,dv}{vw},
\end{align*}
where we used the change of variables $v=a\lambda(x)$ and $w=a\lambda(-x)$ in the last line. By our assumptions on $\phi$ and $\psi$, it will suffice to show that for all $v,w \in \mathbb{R}_{+}$ we have the upper bound \[\frac{v-w}{vw\log(v/w)} \leq 2 \cdot \max\left\{1, \frac{1}{v}, \frac{1}{w}, \frac{1}{vw}\right\}.\]\par
It will be enough by symmetry to consider $\Lambda = \{(v,w) \in \mathbb{R}_{+} \times \mathbb{R}_{+} \, : \, v > w \}$. We have the decomposition $\Lambda = \mathcal{C}_{1} \cup \mathcal{C}_{2} \cup \mathcal{C}_{3}$, where
\begin{align*}
    \mathcal{C}_{1} & \coloneqq \bigg\{(v,w) \in \Lambda \, : \, w \leq -2\sigma(-v/2) \bigg\}, 
    \\
    \mathcal{C}_{2} & \coloneqq \bigg\{(v,w) \in \Lambda \, : \, w \geq \frac{-1}{\sigma(-1/v)} \bigg\},
    \\
    \mathcal{C}_{3} & \coloneqq \bigg\{(v,w) \in \Lambda \, : \, -2\sigma(-v/2) \leq w \leq \frac{-1}{\sigma(-1/v)}\bigg\},
\end{align*}
where $\sigma$ is the function appearing in Lemma \ref{lambda_inverse}. 
\begin{figure}[ht!]
    \centering
    \includegraphics[width=10cm]{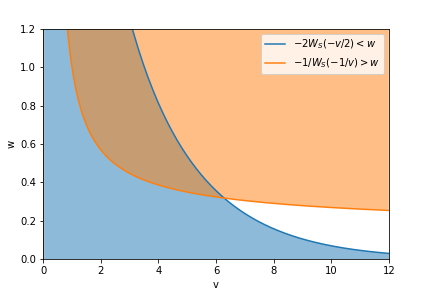}
    \caption{A drawing marking the beginning and end of the different domains.}
    \label{fig:my_label2}
\end{figure}
\begin{itemize}
\item The level surface $g(v,w) = (v-w)/\log(v/w) = C$ for $C > 0$ is given by the equation 
\begin{equation}\label{eq:level_surface}
    w = -C\sigma\left(-\frac{v}{C}\right).
\end{equation} On $\mathcal{C}_{1}$ we are below the level surface \eqref{eq:level_surface} with $C = 2$. Notice that $(1,0.5) \in \mathcal{C}_{1}$ with $g(1,0.5) = \log(\sqrt{2}) < 2$. The continuity of $g$ forces the inequality $g(v,w) \leq 2$ for all $(v,w) \in \mathcal{C}_{1}$. Hence \[\frac{v-w}{vw\log(v/w)} \leq \frac{2}{vw}.\]

\item Notice that \[\frac{v-w}{vw\log(v/w)}=\frac{\frac{1}{v}-\frac{1}{w}}{\log((1/v)/(1/w))}.\] Hence the case of $\mathcal{C}_{2}$ follows from the previous the argument for $\mathcal{C}_{1}$ by considering the level surface of \[g(1/v,1/w) = 1.\]

\item It is straightforward to verify that $v > 2$ and $w < 1$ when  $(v,w) \in \mathcal{C}_{3}$. Hence we obtain for any $(v,w) \in \mathcal{C}_{3}$ that \[\frac{v-w}{wv\log(v/w)}\le \frac{v}{wv\log(2)} \leq 2/w. \qedhere\]
\end{itemize}
\end{proof}
\begin{remark}
The connection from this result to admissibility is that the assumptions boil down to $S=\mathcal{D}\psi\otimes \mathcal{D}{\phi}$ being an admissible operator.
\end{remark}

\begin{remark}
Let $A$ be a Hilbert-Schmidt operator on $L^{2}(\mathbb{R}_{+})$ with integral kernel $A_{K}$. Then one can gauge from the proof of Theorem \ref{Wigner_right_left} that the affine Weyl symbol $f_{A}$ satisfies $f_{A} \in L^2_{r}(\aff)\cap L^2_{l}(\aff)$ if and only if the integral kernel $A_{K}$ satisfies \[A_{K} \in L^2\left(\mathbb{R}_+ \times \mathbb{R}_{+},\frac{s-t}{st\log(s/t)}dt \, ds\right)\cap L^2\left(\mathbb{R}_+ \times \mathbb{R}_{+},\frac{1}{st}dt \, ds\right).\]
\end{remark}

\subsection{Extending the Setting}
\label{sec: Extending The Setting}

Except for Section \ref{sec: Operator Convolution for Tempered Distributions}, we have so far considered convolutions between rather well-behaved functions and operators and obtained norm estimates for the norms of $L^1_r(\aff)$, $L^\infty(\aff)$, $\mathcal{S}_1$ and $\mathcal{L}(L^2(\mathbb{R}_+))$. We have seen that 
\begin{align*}
    \|f\star_\aff S\|_{\mathcal{S}_1}&\leq \|f\|_{L^1_r(\aff)} \|S\|_{\mathcal{S}_1}, \\
    \|T\star_\aff S\|_{L^\infty(\aff)} &\leq \|T\|_{\mathcal{L}(L^2(\mathbb{R}_+))} \|S\|_{\mathcal{S}_1}.
\end{align*}
This generalizes these inequalities to other Schatten classes and $L^p$ spaces.
\begin{proposition} \label{prop:interpolationbasic}
 Let $1\leq p \leq \infty$ and let $q$ be its conjugate exponent given by $p^{-1} + q^{-1} = 1$. If $S\in  \mathcal{S}_p,T\in \mathcal{S}_q$, and $f\in L^1_r(\aff)$, then the following hold:
    
    \begin{enumerate}
        \item $f\star_\aff S\in \mathcal{S}_p$ with $\|f\star_\aff S\|_{\mathcal{S}_p}\leq \|f\|_{L_{r}^{1}(\aff)}\|S\|_{\mathcal{S}_p}$.
        \item $T\star_\aff S\in L^\infty(\aff)$ with $\|T\star_\aff S\|_{L^{\infty}(\aff)}\leq \|S\|_{\mathcal{S}_p} \|T\|_{\mathcal{S}_q}$.
    \end{enumerate}
\end{proposition}

\begin{proof}
    For $p<\infty$, we can clearly interpret the definition of $f\star_\aff S$ as a convergent Bochner integral in $\mathcal{S}_p$. Hence the first inequality follows from \cite[Prop.~1.2.2]{hytonen2016book}. For $p=\infty$, we avoid the unpleasantness of Bochner integration in non-separable Banach spaces by interpreting $f\star_\aff S$ weakly by 
    \begin{equation*}
        \langle f\star_\aff S \psi,\phi \rangle_{L^2(\mathbb{R}_+)}=\int_\aff f(x,a) \langle S U(-x,a)\psi,U(-x,a)\phi \rangle_{L^2(\mathbb{R}_+)} \, \frac{dx\, da}{a},
    \end{equation*}
    for $\psi, \phi \in L^{2}(\mathbb{R}_{+})$.
    A standard argument shows that $f\star_\aff S$ is a bounded operator with \[\|f\star_\aff S\|_{\mathcal{L}(L^2(\mathbb{R}_+))}\leq \|f\|_{L^1_r(\aff)} \|S\|_{\mathcal{L}(L^2(\mathbb{R}_+))}.\]
    Inequality 2.\ follows from the H\"older type inequality \cite[Thm.~2.8]{simon2005traceideals}.
\end{proof}

We have already seen in Section \ref{sec: Admissibility for operators} that we can say more about operator convolutions when one of the operators is admissible. As the next lemma shows, admissibility is also the correct condition to ensure that $f\star_\aff S$ defines a bounded operator for all $f\in L^\infty(\aff)$.

\begin{lemma} \label{lem:boundedquantization}
    Let $S\in \mathcal{S}_1$ and $f\in L^\infty(\aff)$. Define the operator $f\star_\aff \mathcal{D}S\mathcal{D}$ weakly for $\psi,\phi\in \mathrm{Dom}(\mathcal{D})$ by
    \begin{equation} \label{eq:weakboundeddef}
    \langle f\star_\aff \mathcal{D}S\mathcal{D}\psi,\phi \rangle_{L^{2}(\mathbb{R}_{+})}= \int_\aff f(x,a) \langle S\mathcal{D} U(-x,a) \psi, \mathcal{D}U(-x,a) \phi \rangle_{L^{2}(\mathbb{R}_{+})} \, \frac{dx\, da}{a}.
\end{equation}
    Then $f\star_\aff \mathcal{D}S\mathcal{D}$ uniquely extends to a bounded linear operator on $L^2(\mathbb{R}_+)$ satisfying \[\|f\star_\aff \mathcal{D}S\mathcal{D}\|_{\mathcal{L}(L^2(\mathbb{R}_+))}\leq \|f\|_{L^{\infty}(\aff)} \|S\|_{\mathcal{S}_1}.\]
    In particular, if $R$ is an admissible operator, then $f\star_\aff R\in \mathcal{L}(L^2(\mathbb{R}_+))$ with 
    \begin{equation*}
        \|f\star_\aff R\|_{\mathcal{L}(L^2(\mathbb{R}_+))}\leq \|f\|_{L^{\infty}(\aff)} \|\mathcal{D}^{-1}R\mathcal{D}^{-1}\|_{\mathcal{S}_1}.
    \end{equation*}
\end{lemma}
\begin{proof}
By using \eqref{eq:comrelD} we get that 
\begin{align*} 
    \langle f\star_\aff \mathcal{D}S\mathcal{D}\psi,\phi \rangle_{L^{2}(\mathbb{R}_{+})}&= \int_\aff f(x,a) \langle  S U(-x,a)\mathcal{D} \psi, U(-x,a)\mathcal{D}\phi \rangle_{L^{2}(\mathbb{R}_{+})} \, \frac{dx\, da}{a^2} \\
    &= \int_\aff \check{f}(x,a) \langle  S U(-x,a)^*\mathcal{D} \psi, U(-x,a)^*\mathcal{D}\phi \rangle_{L^{2}(\mathbb{R}_{+})} \, \frac{dx\, da}{a} \\
    &= \int_\aff \check{f}(x,a) (S\star_\aff (\mathcal{D}\psi\otimes \mathcal{D}\phi))(x,a) \, \frac{dx \, da}{a}.
\end{align*}
Clearly $\mathcal{D}\psi \otimes \mathcal{D}\phi$ is an admissible operator with \[|\tr(\mathcal{D}^{-1}(\mathcal{D}\psi \otimes \mathcal{D}\phi)\mathcal{D}^{-1})|=|\langle \psi,\phi \rangle|_{L^2(\mathbb{R}_+)}\leq \|\psi\|_{L^2(\mathbb{R}_+)} \|\phi\|_{L^2(\mathbb{R}_+)}.\]
By Corollary \ref{cor:admissible_condition} we therefore get
\begin{equation*}
    \left|\langle f\star_\aff \mathcal{D}S\mathcal{D}\psi,\phi \rangle_{L^{2}(\mathbb{R}_{+})} \right|\leq \|f\|_{L^{\infty}(\aff)} \|S\|_{\mathcal{S}_1} \|\psi\|_{L^{2}(\mathbb{R}_{+})} \|\phi\|_{L^{2}(\mathbb{R}_{+})}.
\end{equation*}
The density of $\mathrm{dom}(\mathcal{D})$ implies that $f\star_\aff \mathcal{D}S\mathcal{D}$ extends to a bounded operator on $L^{2}(\mathbb{R}_{+})$. 
\end{proof}

Armed with Lemma \ref{lem:boundedquantization} and Corollary \ref{cor:admissible_condition}, we prove the following result describing $L^p$ and $\mathcal{S}_p$ properties of convolutions with admissible operators. The proof is essentially an application of complex interpolation: we refer to \cite[Thm.~2.10]{simon2005traceideals} and \cite[Thm.~5.1.1]{bergh1976interpolation} for the interpolation theory of $\mathcal{S}_p$ and $L^p_r(\aff)$.

\begin{proposition} \label{prop:interpolation}
    Let $1\leq p \leq \infty$ and let $q$ be its conjugate exponent given by $p^{-1} + q^{-1} = 1$. If $R\in \mathcal{S}_p$, $g\in L^p_r(\aff)$, and $S$ is an admissible trace-class operator, then: 
    \begin{enumerate}
        \item $g\star_\aff S\in \mathcal{S}_p$ with $\|g\star_\aff S\|_{\mathcal{S}_p}\leq \|S\|_{\mathcal{S}_1}^{1/p} \|\mathcal{D}^{-1}S\mathcal{D}^{-1}\|_{\mathcal{S}_1}^{1/q} \|g\|_{L_{r}^{p}(\aff)}$.
        \item $R\star_\aff S \in L^p_r(\aff)$ with $\|R\star_\aff S\|_{L_{r}^p(\aff)}\leq \|S\|_{\mathcal{S}_1}^{1/q} \|\mathcal{D}^{-1}S\mathcal{D}^{-1}\|_{\mathcal{S}_1}^{1/p} \|R\|_{\mathcal{S}_p}$.
    \end{enumerate}
\end{proposition}

\begin{proof}
For $g\in L^1_r(\aff)\cap L^\infty(\aff)$, we have for $p=\infty$ that Lemma \ref{lem:boundedquantization} gives \[\|g\star_\aff S\|_{\mathcal{L}(L^{2}(\mathbb{R}_{+}))}\leq \|\mathcal{D}^{-1}S\mathcal{D}^{-1}\|_{\mathcal{S}_1} \|g\|_{L^{\infty}(\aff)}.\] Since we also have $\|g\star_\aff S\|_{\mathcal{S}_1}\leq \|g\|_{L_{r}^{1}(\aff)}\|S\|_{\mathcal{S}_1}$, the first result follows by complex interpolation.
For the second claim, if $R\in \mathcal{S}_1$ we know from Corollary \ref{cor:admissible_condition} that
  \[\|R\star_\aff S\|_{L_{r}^{1}(\aff)}\leq \|\mathcal{D}^{-1}S\mathcal{D}^{-1}\|_{\mathcal{S}_1} \|R\|_{\mathcal{S}_1}.\] The result follows by complex interpolation since \[\|R\star_\aff S\|_{L^{\infty}(\aff)}\leq \|S\|_{\mathcal{S}_1} \|R\|_{\mathcal{L}(L
^2(\mathbb{R}_{+}))}. \qedhere\] 

\end{proof}

\section{From the Viewpoint of Representation Theory}
\label{sec: From the Viewpoint of Representation Theory}
 
We will for completeness investigate how various notions of affine Fourier transforms fit into our framework. As we will see, known results from abstract wavelet analysis give connections between affine Weyl quantization, affine Fourier transforms, and admissibility for operators.
 
\subsection{Affine Fourier Transforms}
\label{sec: Affine Fourier Transforms}

\begin{definition}
For $f \in L_{l}^{1}(\aff)$ we define the \textit{(left) integrated representation} $U(f)$ to be the operator on $L^{2}(\mathbb{R}_{+})$ given by \[U(f)\psi \coloneqq \int_{\aff}f(x,a) U(x,a)\psi \, \frac{dx \, da}{a^2}, \qquad \psi \in L^{2}(\mathbb{R}_{+}).\] The \textit{inverse affine Fourier-Wigner transform} $\mathcal{F}_W^{-1}(f)$ of $f \in L_{r}^{1}(\aff)$ is given by 
\[\mathcal{F}_W^{-1}(f) \coloneqq U(\check{f}) \circ \mathcal{D}, \qquad \check{f}(x,a) \coloneqq f((x,a)^{-1}).\]
\end{definition}

The inverse affine Fourier-Wigner transform $\mathcal{F}_W^{-1}(f)$ of $f \in L_{r}^{1}(\aff)$ is explicitly given by 
\begin{equation*}
\mathcal{F}_W^{-1}(f)\psi(s)
=\int_{0}^{\infty} \sqrt{r}\mathcal{F}_{1}(f)(r,s/r)\psi(r) \, \frac{dr}{r},
\end{equation*}
where $\mathcal{F}_{1}$ denotes the Fourier transform in the first coordinate and $\psi\in L^2(\mathbb{R}^+)$. Hence the integral kernel of $\mathcal{F}_W^{-1}(f)$ is given by
\begin{equation}
\label{eq:integral_kernel_fourier_wigner}
    K_{f}(s,r)=\sqrt{r}(\mathcal{F}_1f)(r,s/r), \qquad s,r \in \mathbb{R}_{+}.
\end{equation} 
It is straightforward to verify that we have the estimate \[\|\mathcal{F}_W^{-1}(f)\|_{\mathcal{S}_{2}} \leq \|f\|_{L_{r}^{2}(\aff)},\]
for every $f \in L_{r}^{1}(\aff) \cap L_{r}^{2}(\aff)$. Hence we can extend $\mathcal{F}_W^{-1}$ to be defined on $L_{r}^{2}(\aff)$ and we have that $\mathcal{F}_W^{-1}(f) \in \mathcal{S}_{2}$ for any $f \in L_{r}^{2}(\aff)$.

\begin{proposition}
\label{proposition_correspondence_fourier_wigner}
The inverse affine Fourier-Wigner transform is a unitary transformation $\mathcal{F}_W^{-1}:\mathcal{Q}_{1} \to \mathcal{S}_{2}$, where
\[\mathcal{Q}_{1} \coloneqq \{f\in L^2_r(\aff)\mid \mathrm{ess\, supp}(\mathcal{F}_1(f))\subset \mathbb{R}_{+}\times \mathbb{R}_{+}\}.\]
\end{proposition}

\begin{proof}
Any function $K \in L^{2}(\mathbb{R}_{+} \times \mathbb{R}_{+})$ can be written uniquely on the form $K_{f}$ in \eqref{eq:integral_kernel_fourier_wigner} for some $f \in \mathcal{Q}_{1}$. Moreover, we have \[\|K_{f}\|_{L^{2}(\mathbb{R}_{+} \times \mathbb{R}_{+})} = \sqrt{\int_{0}^{\infty}\int_{0}^{\infty}|\mathcal{F}_{1}f(r, s/r)|^2 \, dr \, \frac{ds}{s}} = \|f\|_{L_{r}^{2}(\aff)}.\]
Since there is a norm-preserving correspondence between integral kernels in $L^{2}(\mathbb{R}_{+} \times \mathbb{R}_{+})$ and Hilbert-Schmidt operators on $L^{2}(\mathbb{R}_{+}),$ the claim follows.
\end{proof}

It is straightforward to check that the inverse affine Fourier-Wigner transform $\mathcal{F}_W^{-1}$ satisfies for $f,g \in \mathcal{Q}_{1}$ the properties
    \begin{itemize}
        \item $\mathcal{F}_W^{-1}(f)^* = \mathcal{F}_W^{-1}(\Delta^{1/2}f^*), \qquad f^{*}(x,a) \coloneqq \overline{f((x,a)^{-1})}$;
        \item $\mathcal{F}_W^{-1}(f*_{\aff} g) = \mathcal{F}_W^{-1}(f) \circ \mathcal{D}^{-1} \circ \mathcal{F}_W^{-1}(g) = U(\check{f}) \circ \mathcal{F}_W^{-1}(g)$;
        \item$U(x,a) \circ \mathcal{F}_W^{-1}(f)  = \mathcal{F}_W^{-1}(R_{(x,a)}(f))$;
        \item $\mathcal{F}_W^{-1}(f)\circ U(x,a)   = \mathcal{F}_W^{-1}\left(\sqrt{a}L_{(x,a)^{-1}}(f)\right)$.
    \end{itemize} 
    
\begin{definition}
The \textit{affine Fourier-Wigner transform} $\mathcal{F}_{W}:\mathcal{S}_{2} \to \mathcal{Q}_{1}$ is defined to be the inverse of $\mathcal{F}_{W}^{-1}|_{\mathcal{Q}_1}$.
\end{definition}

\begin{remark}
\hfill
\begin{itemize}
    \item To avoid overly cluttered notation, we have used the symbol $\mathcal{F}_{W}$ for both the classical Fourier-Wigner transform in Section \ref{sec: Operator Convolution}, and the affine Fourier-Wigner transform. It should be clear from the context which operator we are referring to.
    \item Recall that the right multiplication $R$ acts on elements in $L_{r}^{2}(\aff)$ by \[R_{(y,b)}f(x,a) = f((x,a) (y, b))\] for $(x,a), (y,b) \in \aff$. For a closed subspace $\mathcal{H} \subset L_{r}^{2}(\aff)$ invariant under $R$, we write $R\vert_\mathcal{H}\cong U$ if there exists a unitary map $T:\mathcal{H} \to L^{2}(\mathbb{R}_{+})$ satisfying \[T \circ R(x,a)f = U(x,a) \circ Tf,\] 
for all $f \in \mathcal{H}$ and $(x,a) \in \aff$.
Define 
	\begin{equation*}
	L^2_U(\aff)\coloneqq\overline{\mathrm{span}} \{\mathcal{H}\subset L^2_r(\aff) : R\vert_\mathcal{H}\cong U\}. 
	\end{equation*} 
	From \cite[Lem.~3]{duflo1976} we deduce that 
	\begin{equation*}
		L^2_U(\aff) = \mathcal{Q}_{1},
	\end{equation*}
as both spaces are the image of the Hilbert-Schmidt operators under the Fourier-Wigner transform. Note that \cite{duflo1976} uses left Haar measure, but translating to right Haar measure is an easy exercise using that $f\mapsto \check{f}$ is a unitary equivalence from the left regular representation on $L^2_l(\aff)$ to the right regular representation on $L^2_r(\aff)$.
\end{itemize}
\end{remark}

\begin{example}\label{ex:Fourier_Wigner_of_rank_one}
Let $\phi, \psi \in L^2(\mathbb{R}_+)$ with $\psi \in \mathrm{dom}(\mathcal{D})$. If $f(x,a)=\langle \phi,U(x,a)^*\mathcal{D}\psi\rangle_{L^2(\mathbb{R}_+)}$, one finds using Proposition \ref{prop:admissible} that $f\in L^2_r(\aff)$ and
\[\langle \mathcal{F}_W^{-1}(f)\xi,\eta\rangle_{L^2(\mathbb{R}_+)}=\langle (\phi \otimes \psi)\xi,\eta \rangle_{L^2(\mathbb{R}_+)}\] for $\eta \in L^2(\mathbb{R}_+)$ and $\xi \in \mathrm{dom}(\mathcal{D}).$ This implies that $\mathcal{F}_W^{-1}(f)=\phi \otimes \psi$, in other words for $(x,a) \in \aff$ that
\[\mathcal{F}_W(\phi\otimes\psi)(x,a) 
=\langle \phi, U(x,a)^*\mathcal{D}\psi\rangle_{L^2(\mathbb{R}_+)}.\]
\end{example}

For the Heisenberg group, the Fourier-Wigner transform has a very convenient expression for trace-class operators, see \eqref{eq:fwheisenberg}. The corresponding expression on the affine group is $\mathcal{F}_W(A)(x,a) =\tr(A\mathcal{D}U(x,a))$, and the next result shows that it holds as long as the objects in the formula are well-defined.  The result is due to F\"uhr in this generality \cite[Thm.~4.15]{fuhr2004}, and builds on an earlier result due to Duflo and Moore \cite[Cor.~2]{duflo1976}.

\begin{proposition}[F\"uhr, Duflo, and Moore]
\label{prop:fourier_wigner_explicit_formula}
Let $A \in \mathcal{S}_1$ be such that $A\mathcal{D}^{-1}$ extends to a Hilbert-Schmidt operator. Then
\begin{equation}  \label{eq:Fourier_Wigner_via_trace}
    \mathcal{F}_W(A\mathcal{D}^{-1})(x,a)=\tr(A U(x,a)).
\end{equation}
\end{proposition}

\begin{proof}
To see how the result follows from \cite[Thm.~4.15]{fuhr2004}, we need some terminology regarding direct integrals, see \cite[Section 3.3]{fuhr2004}. Recall that the Plancherel theorem \cite[Thm.~3.48]{fuhr2004} supplies a measurable field of Hilbert spaces indexed by the dual group $\{\mathcal{H}_\pi\}_{[\pi]\in \hat{G}}$. For the affine group $G=\aff$, the Plancherel measure is counting measure supported on the two irreducible representations $\pi_1(x,a)=U(x,a)$ on $L^2(\mathbb{R}_+)$ and $\pi_2(x,a)=U(x,a)$ on $L^2(\mathbb{R}_-)\coloneqq L^2(\mathbb{R}_-, r^{-1} \, dr)$. So we can construct an element $\{A_{[\pi]}\}_{[\pi]\in \hat{G}}$ of the direct integral 
\begin{equation*}
 \int_{\hat{G}}
    ^\oplus HS(\mathcal{H}_\pi) d\hat{\mu}([\pi])
\end{equation*}
by choosing $A_{[\pi_{1}]} = A\mathcal{D}^{-1}$ and $A_{[\pi]}=0$ for $[\pi]\neq [\pi_{1}]$. Inserting this measurable field of trace-class operators into \cite[Thm.~4.15]{fuhr2004} then gives the conclusion. 
\end{proof}

For $f,g \in L^{2}(\mathbb{R})$ we denote by $\mathrm{SCAL}_{g}f$ the \textit{scalogram} of $f$ with respect to $g$ given by $\mathrm{SCAL}_{g}f(x,a) \coloneqq|\mathcal{W}_{g}f(x,a)|^2$ where $\mathcal{W}_{g}f$ is the continuous wavelet transform 
\[\mathcal{W}_{g}f(x,a) \coloneqq \frac{1}{\sqrt{a}}\int_{\mathbb{R}}f(t)\overline{g\left(\frac{t - x}{a}\right)} \, dt.\]
The following result, which follows from Lemma \ref{lem: convolution_of_rank_one} and Example \ref{ex:Fourier_Wigner_of_rank_one}, gives a connection between the affine Fourier-Wigner transform, affine convolutions, and the scalogram.

\begin{corollary}\label{cor:equallity_of_scalogram_and_other_stuff}
Let $f,g \in L^{2}(\mathbb{R})$ such that $\psi \coloneqq \hat{f}$ and $\phi\coloneqq \hat{g}$ are supported in $\mathbb{R}_{+}$ and are in $L^{2}(\mathbb{R}_{+})$. If $\psi$ is admissible then
\begin{equation}
\label{eq:equality_of_different_stuff}
    |\mathcal{F}_W(\phi\otimes \mathcal{D}^{-1}\psi)(x,a)|^2 = (\phi \otimes \phi)\star_\aff (\psi \otimes \psi)(-x,a) = \frac{1}{a}\mathrm{SCAL}_{g}f(x,a).
\end{equation}
\end{corollary}

\begin{remark}
The condition that $\psi$ is admissible in Corollary \ref{cor:equallity_of_scalogram_and_other_stuff} is only necessary for the first equality in \eqref{eq:equality_of_different_stuff}.
Recall that the affine Wigner distribution $W_{\aff}^{\psi}$ is the affine Weyl symbol of the rank-one operator $\psi \otimes \psi$. If we use Proposition \ref{prop:convolutionasweyl} together with Corollary \ref{cor:equallity_of_scalogram_and_other_stuff}, then we recover \cite[Thm.~5.1]{berge2019affine}.
\end{remark}

Corollary \ref{cor:equallity_of_scalogram_and_other_stuff} shows that we have the simple relation 
\begin{equation} \label{eq:simpleadmissiblerelation}
    |\mathcal{F}_W(A\mathcal{D}^{-1})(x,a)|^2=A\star_\aff A(-x,a)
\end{equation}
for positive rank-one operators $A$. By Corollary \ref{corr:admissibility_check}, admissibility therefore means that ${\mathcal{F}_W(A\mathcal{D}^{-1})\in L^2_r(\aff)}$ in this case. For more general operators, \eqref{eq:simpleadmissiblerelation} will no longer hold. However, we still obtain a result relating admissibility to the Fourier-Wigner transform.
Note that in the first statement in Proposition \ref{prop:equvalence_right} if $A\in \mathcal{S}_1$ we interpret $\mathcal{F}_W(A\mathcal{D}^{-1})\coloneqq\tr(AU(x,a))$ if we do not know that $A\mathcal{D}^{-1}$ extends to a Hilbert-Schmidt operator.
\begin{proposition}\label{prop:equvalence_right}
    Let $A$ be a trace-class operator on $L^{2}(\mathbb{R}_{+})$. Then the following are equivalent:
    \begin{enumerate}[1)]
        \item $\mathcal{F}_W(A\mathcal{D}^{-1})\in L^2_r(\aff)$.\label{1}
        \item $A\mathcal{D}^{-1}$ extends from $\mathrm{dom}(\mathcal{D}^{-1})$ to a Hilbert-Schmidt operator on $L^2(\mathbb{R}_+)$.\label{2}
        \item $A^*A$ is admissible.\label{3}
    \end{enumerate}
\end{proposition}
\begin{proof}
The equivalence of \ref{1} and \ref{2} follows from \cite[Thm.~4.15]{fuhr2004}, by applying that theorem to the element $\{A_{[\pi]}\}_{[\pi]\in \hat{G}}$ of the direct integral (see proof of Proposition \ref{prop:fourier_wigner_explicit_formula})
\begin{equation*}
 \int_{\hat{G}}
    ^\oplus HS(\mathcal{H}_\pi) d\hat{\mu}([\pi])
\end{equation*}
given by choosing $A_{[\pi_{1}]} = A$ and $A_{[\pi]}=0$ for $[\pi]\neq [\pi_{1}]$. 

The equivalence of \ref{2} and \ref{3}  is clear apart from technicalities resulting from the unboundedness of $\mathcal{D}^{-1}$. If we assume \ref{2}, then \cite[Thm.~13.2]{rudin1991functionalanalysis} gives that $(A\mathcal{D}^{-1})^*=\mathcal{D}^{-1}A^*$, where the equality includes equality of domains. As the domain of the left term is all of $L^2(\mathbb{R}_+)$ by assumption, this means that the range of $A^*$ is contained in $\mathrm{dom}(\mathcal{D}^{-1}).$ In particular, $A^*A$ maps $\mathrm{dom}(\mathcal{D})$ into $\mathrm{dom}(\mathcal{D}^{-1})$, and as we also have $\mathcal{D}^{-1}A^*A\mathcal{D}^{-1}=(A\mathcal{D}^{-1})^*A\mathcal{D}^{-1}$ where $A\mathcal{D}^{-1}$ is Hilbert-Schmidt, $A^*A$ satisfies all requirements for being admissible.

Conversely, if $A^*A$ is admissible, then we have for $\psi \in \mathrm{dom}(\mathcal{D}^{-1})$
\begin{equation*}
  \|A\mathcal{D}^{-1}\psi\|^2_{L^2(\mathbb{R}_+)}= \langle \mathcal{D}^{-1}A^*A\mathcal{D}^{-1}\psi,\psi \rangle_{L^2(\mathbb{R}_+)}\leq \|\mathcal{D}^{-1}A^*A\mathcal{D}^{-1}\|_{\mathcal{L}(L^2(\mathbb{R}_+))} \|\psi\|_{L^2(\mathbb{R}_+)}^2.
\end{equation*}
So $A\mathcal{D}^{-1}$ extends to a bounded operator, and as this operator satisfies that \[(A\mathcal{D}^{-1})^*A\mathcal{D}^{-1}=\mathcal{D}^{-1}A^*A\mathcal{D}^{-1}\] is trace-class, $A\mathcal{D}^{-1}$ is a Hilbert-Schmidt operator.
\end{proof}

\begin{remark}
Recall that we consider $\mathcal{F}_W$ a Fourier transform of operators. The inequality $\|\mathcal{F}_W(A\mathcal{D}^{-1})\|_{L^\infty(\aff)}\leq \|A\|_{\mathcal{S}_{1}}$ and the equality $\|A\|_{\mathcal{S}_{2}}=\|\mathcal{F}_W(A)\|_{L^2_r(\aff)}$ might therefore be interpreted as the endpoints $p=\infty$ and $p=2$ of a Hausdorff-Young inequality, where the appearance of $\mathcal{D}^{-1}$ suggests that the definition of the Fourier-Wigner transform must depend on $p$. In fact, a Hausdorff-Young inequality of this kind---formulated in the other direction, i.e.\ for maps from functions on $\aff$ to operators---was shown in \cite[Thm.~1.41]{Eymard:1979} for $1\leq p \leq 2$.
\end{remark}

There is a second Fourier transform related to the affine group that comes from representation theory. We define the \textit{affine Fourier-Kirillov transform} as the map $\mathcal{F}_{\mathrm{KO}}: \mathcal{Q}_{1} \to L_{r}^{2}(\aff)$ given by
\begin{align*}
(\mathcal{F}_{\mathrm{KO}}f)(x,a)&=\sqrt{a}\int_{\mathbb{R}^{2}}f\left(\frac{v}{\lambda(-u)},e^u\right)e^{-2\pi i (xu+av)} \, \frac{du \, dv}{\sqrt{\lambda(-u)}}, \qquad (x,a) \in \aff.
\end{align*}

More information about the Fourier-Kirillov transform can be found in \cite{kirillov2004lectures}. The following result, which is motivated by \eqref{eq: heisenbergweylFW} and is a slight generalization of \cite[Section VIII.6]{ali2000}, shows that the affine Weyl quantization is intrinsically linked with the Fourier transforms on the affine group. 

\begin{proposition}
\label{commutative_diagram_result}
    Let $A_f$ be a Hilbert-Schmidt operator on $L^{2}(\mathbb{R}_{+})$ with affine symbol $f \in L_{r}^{2}(\aff)$. Then the following diagram commutes:
    \begin{equation*}
\begin{tikzcd}[column sep=1.5em]
 & \mathcal{S}_{2} \arrow[dl,"\mathcal{F}_W"'] \\
\mathcal{Q}_{1}  \arrow[rr,"\mathcal{F}_{\mathrm{KO}}"'] && L_{r}^{2}(\aff)\arrow[ul,"f\longmapsto A_f"']
    \end{tikzcd}
    \end{equation*}
\end{proposition}
\begin{proof}
Recall from \eqref{eq:integral_kernel_fourier_wigner} that the integral kernel of $\mathcal{F}_W^{-1}(g)$ for $g \in \mathcal{Q}_{1}$ is given by \[K_{g}(s,r)=\sqrt{r}(\mathcal{F}_1g)(r,s/r), \qquad s,r \in \mathbb{R}_{+}.\] Hence by using \eqref{eq: dequantization_formula} and a change of variables, we see that the affine Weyl symbol of $\mathcal{F}_W^{-1}(g)$ is given at the point $(x,a) \in \aff$ by
\begin{align*}
    \int_{-\infty}^{\infty}\sqrt{a\lambda(-u)}\mathcal{F}_1(g)(a\lambda(-u),e^{u}) e^{-2\pi ixu} \, du
    &=\int_{\mathbb{R}^{2}}\sqrt{a\lambda(-u)}g(v,e^{u}) e^{-2\pi i(xu +av\lambda(-u))} \, du \, dv\\
    &=\sqrt{a}\int_{\mathbb{R}^{2}}g\left(\frac{v}{\lambda(-u)},e^{u}\right) e^{-2\pi i(xu +av)} \, \frac{du \, dv}{\sqrt{\lambda(-u)}} 
    \\ & = (\mathcal{F}_{\mathrm{KO}}g)(x,a). \qedhere
\end{align*}
\end{proof}

\begin{remark}
\hfill
\begin{itemize}
    \item In \cite{Mantoiu2019} the authors define an alternative quantization scheme on general type 1 groups. Their quantization scheme together with the affine Weyl quantization is used in \cite{Mantoiu2019} to define a quantization scheme on the cotangent bundle $T^*\aff$.
    \item Consider $A_f$ for some $f\in L^2_r(\aff).$ Inserting $f=\mathcal{F}_{KO}\mathcal{F}_W(A_f)$ into Proposition \ref{prop:admissibilityandweylsymbol} allows us to obtain a formal expression for $\tr(\mathcal{D}^{-1}A_f\mathcal{D}^{-1})$ in terms of $\mathcal{F}_W(A_f)$: a formal calculation gives that for sufficiently nice operators $A_f$ we have
\begin{equation} \label{eq:admissbilityfw}
	\tr(\mathcal{D}^{-1}A_f\mathcal{D}^{-1}) = \int_0^\infty [\mathcal{F}_1\mathcal{F}_W(A_f)](a,1)\, \frac{da}{a^{3/2}},
\end{equation}	
where $\mathcal{F}_1$ is the Fourier transform in the first coordinate. This is similar to a condition in \cite[Cor.~5.2]{gazeau2016covariant}, where finiteness of \eqref{eq:admissbilityfw} is used as a necessary condition for $1\star_\aff A_f=I_{L^2(\mathbb{R}_+)}$ to hold, where $1(x,a)=1$ for all $(x,a)\in \aff$. We will see in Section \ref{sec: Other Covariant Integral Quantizations} that this is closely related to admissibility of $A_f$. Unfortunately, the formal calculation leading to \eqref{eq:admissbilityfw} does not give clear conditions on $A_f$ for the equality to hold. 
\end{itemize}
\end{remark}

\subsection{Affine Quantum Bochner Theorem}
\label{sec: Plancherel Measure and Affine KLM conditions}
On the Heisenberg group, the Fourier-Wigner transform behaves in many ways like the Fourier transform on functions. In particular, for $f \in L^{1}(\mathbb{R}^{2n})$ and $S,T \in \mathcal{S}_{1}(\mathbb{R}^n)$ we get the \textit{decoupling equations}
\begin{equation}\label{eq: conv_fourier}
\mathcal{F}_W(f\star S)=\mathcal{F}_\sigma(f)\mathcal{F}_W(S), \qquad \mathcal{F}_\sigma(S\star T)=\mathcal{F}_W(S)\mathcal{F}_W(T),
\end{equation}
where $\mathcal{F}_{\sigma}$ denotes the symplectic Fourier transform and $\mathcal{F}_{W}$ denotes the classical Fourier-Wigner transform introduced in Section \ref{sec: Operator Convolution}. Although the affine version of \eqref{eq: conv_fourier} does not hold, one can develop as a special case of \cite[Thm.~4.12]{fuhr2004} a version of Bochner's theorem for the affine Fourier-Wigner transform. This is analogous to the \textit{quantum Bochner theorem} \cite[Prop.~3.2]{werner1984} for the Heisenberg group. \par 
Bochner's classical theorem \cite[Thm.~4.19]{folland2016course} characterizes functions that are Fourier transforms of positive measures. The Bochner theorem for the affine Fourier-Wigner transform answers the following question: Which functions on $\aff$ are of the form $\mathcal{F}_W(S)$, where $S$ is a positive trace-class operator? As in Bochner's classical theorem, it turns out that the correct notion to consider is functions of positive type.
Recall that a function $f:\aff \to \mathbb{C}$ is a \textit{function of positive type} if for any finite selection of points
$\Omega \coloneqq \{(x_{1}, a_{1}), \dots ,(x_{n}, a_{n})\} \subset \aff$ the 
matrix $A_\Omega$ with entries
\[(A_{\Omega})_{i,j} \coloneqq f((x_{i}, a_{i})^{-1}(x_{j}, a_{j}))\] is positive semi-definite. 
Before stating the general result we consider an illuminating special case.

\begin{example}
Assume that $A = \phi \otimes \psi$ is a rank-one operator where $\phi, \psi \in L^{2}(\mathbb{R}_{+})$. We will show that 
\begin{equation}
\label{eq:computation_motivation_positive}
    \mathcal{F}_{W}(A\mathcal{D}^{-1})(x,a) = \langle U(x,a)\phi,\psi\rangle_{L^{2}(\mathbb{R}_{+})}
\end{equation}
is a function of positive type on $\aff$ if and only if $A$ is a positive operator.
If $A$ is positive, then a standard fact \cite[Prop.~3.15]{folland2016course} shows that \eqref{eq:computation_motivation_positive} is a function of positive type.
Conversely, we have from \cite[Cor.~3.22]{folland2016course} that \[\mathcal{F}_{W}(\phi \otimes \psi\mathcal{D}^{-1})((x,a)^{-1}) = \overline{\mathcal{F}_{W}(\psi \otimes \phi\mathcal{D}^{-1})(x,a)} =  \overline{\mathcal{F}_{W}(\phi \otimes \psi\mathcal{D}^{-1})(x,a)}.\]  Hence $\langle U(x,a)\phi, \psi \rangle_{L^2(\mathbb{R}_+)} = \langle U(x,a)\psi, \phi \rangle_{L^2(\mathbb{R}_+)}$ and it follows from \cite[Thm.~4.2]{cont_wavelet_transform} that $\phi = c \cdot \psi$ for some $c \in \mathbb{C}$. We can conclude from \cite[Cor.~3.22]{folland2016course} that $c \geq 0$ since \[\mathcal{F}_{W}(c\psi \otimes \psi\mathcal{D}^{-1})(0,1) = c \cdot \|\psi\|_{L^{2}(\mathbb{R}_{+})} \geq 0.\]
\end{example}

We are now ready to state the main result regarding positivity. This result is actually, when interpreted correctly, a special case of the general result \cite[Thm.~4.12]{fuhr2004}.

\begin{theorem}
    \label{quantum_bochner_theorem}
    Let $A$ be a trace-class operator on $L^2(\mathbb{R}_+)$. Then $A$ is a positive operator if and only if the function
    \begin{equation*}
        \mathcal{F}_W(A\mathcal{D}^{-1})(x,a)=\tr(AU(x,a))
    \end{equation*}
    is of positive type on $\aff$.
\end{theorem}
\begin{proof}
We use the same notation as in the proof of Proposition \ref{prop:fourier_wigner_explicit_formula}. For $G=\aff$, the abstract result in \cite{fuhr2004} says that if \begin{equation*}
    \{A_{[\pi]}\}_{[\pi]\in \hat{G}}\in \int_{\hat{G}}
    ^\oplus HS(\mathcal{H}_\pi) d\hat{\mu}([\pi])
\end{equation*} consists of trace-class operators, then $A_{[\pi]}$ is positive a.e.\ with respect to $\hat{\mu}$ if and only if the function $\int_{\hat{G}} \tr(A_{[\pi]} \pi(g)^*) d\hat{\mu}([\pi])$ is of positive type.

As in the proof of Proposition \ref{prop:equvalence_right},  we pick $A_{[\pi_{1}]} = A$ and $A_{[\pi]}=0$ for $[\pi]\neq [\pi_{1}]$. The resulting section consists of positive operators for a.e.\ $[\pi]$ if and only if $A$ is positive.  By the abstract result in \cite{fuhr2004}, this happens if and only if \[\int_{\hat{G}} \tr(A_{[\pi]} \pi(g)^*) d\hat{\mu}([\pi])=\tr(AU(x,a)^*)\] is a function of positive type. The definition of functions of positive type gives that this is equivalent to $\tr(AU(x,a))$ being of positive type.
\end{proof}

\section{Examples}
\label{sec: Examples}

In this section, we show how the theory developed in this paper provides a common framework for various operators and functions studied by other authors. We also introduce an analogue of the Cohen class of time-frequency distributions for the affine group, and deduce its relation to the previously studied \textit{affine quadratic time-frequency representations}.

\subsection{Affine Localization Operators}
\label{sec: Representing Localization Operators Through Affine Convolution}

There is no general consensus of a localization operator in the affine setting. We will use the following definition based on the convolution framework.

\begin{definition}
Let $f \in L_{r}^{1}(\aff)$ and $\varphi \in L^{2}(\mathbb{R}_{+})$. We say that
\[A = f \star_{\aff} (\varphi \otimes \varphi)\]
is an \textit{affine localization operator} on $L^{2}(\mathbb{R}_{+})$.
\end{definition}
Inequality \eqref{eq:conv_op_fun_bochner} shows that an affine localization operator $A$ is a trace-class operator on $L^{2}(\mathbb{R}_{+})$ with \[\|A\|_{\mathcal{S}_{1}} \leq  \|f\|_{L_{r}^{1}(\aff)}\|\varphi\|_{L^{2}(\mathbb{R}_{+})}^{2}.\] Moreover, Proposition \ref{prop:convolution_with_admissible} implies that $A$ is admissible whenever $\varphi$ is admissible and $f\in L^1_l(\aff)\cap L^1_r(\aff).$

We will now see that the affine localization operators are naturally unitarily equivalent to the more commonly defined localization operators on the Hardy space $H^2_+(\mathbb{R})$.
Recall that the space $H^2_+(\mathbb{R})$ is the subspace of $L^2(\mathbb{R})$ consisting of elements $\psi$ whose Fourier transform $\mathcal{F}\psi$ is supported on $\mathbb{R}_+$. Note that the composition $\mathcal{D}\mathcal{F}$ is a unitary map from $H_+^2(\mathbb{R})$ to $L^2(\mathbb{R}_+)$. An \textit{admissible wavelet} $\xi \in H^2_+(\mathbb{R})$ satisfies by definition that \[c_{\xi} \coloneqq \int_{0}^{\infty}\frac{|\mathcal{F}(\xi)(\omega)|^{2}}{\omega} \, d\omega < \infty.\] In other words, $\mathcal{D}\mathcal{F}\xi\in L^2(\mathbb{R}_+)$ is an admissible function in the sense of Definition \ref{def:admissiblefunction}. In \cite[Thm.~18.13]{wong2002localization} the localization operator $A_{f}^{\xi}$ on $H^2_+(\mathbb{R})$, given an admissible wavelet $\xi \in H^2_+(\mathbb{R})$ and $f \in L_{l}^{1}(\aff)$, is defined by 
\begin{equation*}
	A_{f}^{\xi} \psi=c_\xi \int_\aff f(x,a) \langle \xi, \pi(x,a) \xi \rangle_{H_{+}^{2}(\mathbb{R})} \pi(x,a)\xi \, \frac{dx \, da}{a^2}, \quad \xi \in H^2_+(\mathbb{R}),
\end{equation*}
where $\pi$ acts on $H^2_+(\mathbb{R})$ by
\begin{equation}
\label{eq:wavelet_representation}
\pi(x,a)\xi(t)=\frac{1}{\sqrt{a}} \xi\left(\frac{t-x}{a}\right), \qquad \psi \in H^2_+(\mathbb{R}).
\end{equation} 
The next proposition is straightforward and relates operators on the form $A_{f}^{\xi}$ with affine localization operators. 
\begin{proposition} \label{prop:wong}
Consider $f \in L_{l}^{1}(\aff)$ and an admissible wavelet $\xi \in H^2_+(\mathbb{R})$. Then
\begin{equation*}
(\mathcal{D}\mathcal{F})A^\xi_{f} (\mathcal{D}\mathcal{F})^* = c_\xi \cdot \check{f}\star_\aff (\mathcal{D}\mathcal{F}\xi \otimes \mathcal{D}\mathcal{F}\xi).
\end{equation*}
\end{proposition}

\begin{remark}
\hfill
\begin{enumerate}
    \item From Proposition \ref{prop:wong} it follows that Proposition \ref{prop:interpolation} is a generalization of the result \cite[Thm.~18.13]{wong2002localization}.
    \item In \cite{daubechies1988localization}, Daubechies and Paul define localization operators in the same way as in \cite{wong2002localization}, except that they use $\pi(-x,a)$ instead of $\pi(x,a)$ in \eqref{eq:wavelet_representation} and consider symbols $f$ on the full affine group $\aff_{F} = \mathbb{R} \times \mathbb{R}^{*}$. The eigenfunctions and eigenvalues of the resulting localization operators acting on $L^2(\mathbb{R})$ are studied in detail in \cite{daubechies1988localization} when the window is related to the first Laguerre function, and $f=\chi_{\Omega_{C}}$ where
\[\Omega_{C} \coloneqq \{(x,a) \in \aff \, : \, |(x,a)-(0,C)|^2\leq (C^2-1) \}.\] The corresponding inverse problem, i.e.\ conditions on the eigenfunctions of the localization operator that imply that $\Omega=\Omega_C$, is studied in \cite{abreu2012dorfler}.
    \item Localization operators with windows related to Laguerre functions have also been extensively studied by Hutn\'{\i}k, see for instance \cite{Hutnik2009toeplitzrelatedwavelets,Hutnik2011calderonboundedness,Hutnik2011waveletslaguerretoeplitz}, with particular emphasis on symbols $f$ depending only on either $x$ or $a$. When $f(x,a)=f(a)$, it is shown that the resulting localization operator is unitarily equivalent to multiplication with some function $\gamma_f$. This correspondence allows properties of the localization operator to be deduced from properties of $\gamma_f$.
\end{enumerate}
\end{remark}

\subsection{Covariant Integral Quantizations}
\label{sec: Other Covariant Integral Quantizations}

Operators of the form $f\star_\aff S$ form the basis of the study of covariant integral quantizations by Gazeau and his collaborators in \cite{ali2014book,bergeron2014intquant,bergeron2018variations,gazeau2020signal,gazeau20192d,gazeau2016covariant}. Apart from differing conventions that we clarify at the end of this section, \textit{covariant integral quantizations} on $\aff$ are maps $\Gamma_S$ sending functions on $\aff$ to operators given by 
\begin{equation*}
	\Gamma_S(f) = f\star_\aff S,
\end{equation*}
for some fixed operator $S$. By varying $S$ we obtain several quantization maps $\Gamma$ with properties depending on the properties of $S$. Examples of such quantization procedures with a different parametrization of $\aff$ are studied in \cite{gazeau2016covariant,bergeron2018variations}. Their approach is to define $S$ either by $\mathcal{F}_W(S)$ or by its kernel as an integral operator, and deduce conditions on this function that ensures the condition  \[1\star_\aff S=I_{L^2(\mathbb{R}_+)}.\]

\begin{example}
The affine Weyl quantization is an example of a covariant integral quantization $\Gamma_{S}$, where $S$ is not a bounded operator. It corresponds to choosing $S=P_{\aff}$ by Theorem \ref{thm:quantization_through_convolution}.
\end{example}

\begin{remark}
The example above leads to a natural question: could there be other operators $P$ such that $f\star_\aff P$ behaves as an affine analogue of Weyl quantization? Since Weyl quantization on $\mathbb{R}^{2n}$ is given by convolving with the parity operator, a natural guess is \[P\psi(r)=\psi(1/r), \qquad \psi \in L^{2}(\mathbb{R}_{+}).\] The resulting quantization $\Gamma_P(f)=f\star_\aff P$ has been studied by Gazeau and Murenzi in \cite[Sec.~7]{gazeau2016covariant}. It has the advantage that $P$ is a bounded operator, but unfortunately by \cite[Prop.~7.5]{gazeau2016covariant} it does not satisfy the natural dequantization rule 
\begin{equation*}
    f=\Gamma_P(f)\star_\aff P.
\end{equation*}
We also mention that Gazeau and Bergeron have shown that this choice of $P$ is merely a special case corresponding to $\nu=-1/2$ of a class $P_\nu$ of operators defining possible affine versions of the Weyl quantization \cite[Sec.~4.5]{bergeron2018variations}.
\end{remark}

In quantization theory one typically wishes that the domain of $\Gamma_S$ contains $L^\infty(\aff)$. This, by Lemma \ref{lem:boundedquantization}, leads us to chose $S=\mathcal{D}T\mathcal{D}$ for some trace-class operator $T$. In particular, one requires that $\Gamma_S(1)=I_{L^2(\mathbb{R}_+)}$, which can be easily satisfied as the following proposition shows.

\begin{proposition}
    Let $T$ be a trace-class operator on $L^{2}(\mathbb{R}_{+})$. Then
    \begin{equation*}
        1\star_\aff \mathcal{D}T\mathcal{D}=\tr(T)I_{L^2(\mathbb{R}_+)}.
    \end{equation*}
\end{proposition}

\begin{proof}
Let $\psi,\phi \in \mathrm{dom}(D)$. We have by \eqref{eq:weakboundeddef} that
\begin{align*}
     \langle 1\star_\aff \mathcal{D}T\mathcal{D}\psi,\phi \rangle_{L^2(\mathbb{R}_+)} & = \int_\aff \langle U(-x,a)^* \mathcal{D}T\mathcal{D} U(-x,a) \psi, \phi \rangle_{L^2(\mathbb{R}_+)} \, \frac{dx\, da}{a}\\ 
     &= \int_\aff  T\star_\aff (\mathcal{D}\psi\otimes \mathcal{D}\phi) \, \frac{dx \, da}{a} \\
     &= \tr(T)  \langle \psi,\phi \rangle_{L^2(\mathbb{R}_+)},
\end{align*}
where the last equality uses Theorem \ref{thm:operator_orthogonality_relation}.
\end{proof}

Following the terminology used by Gazeau et al., we have a \textit{resolution of the identity operator} of the form
\begin{equation*}
    I_{L^2(\mathbb{R}_+)}=\Gamma_{\mathcal{D}T\mathcal{D}}(1)=\int_\aff U(-x,a)^*\mathcal{D}T\mathcal{D}U(-x,a) \, \frac{dx \, da}{a},
\end{equation*}
where $\tr(T)=1$ and the integral has the usual weak interpretation.

Given a positive trace-class operator $T$ with $\tr(T)=1$, we know that \[\Gamma_{\mathcal{D}T\mathcal{D}}(f)=f\star_\aff \mathcal{D}T\mathcal{D}\] defines a bounded map $\Gamma_{\mathcal{D}T\mathcal{D}}:L^\infty(\aff)\to \mathcal{L}(L^2(\mathbb{R}_+))$ with $\Gamma_{\mathcal{D}T\mathcal{D}}(1)=I_{L^2(\mathbb{R}_+)}$. Moreover, $\Gamma_{\mathcal{D}T\mathcal{D}}$ maps positive functions to positive operators and by a variation of Lemma \ref{L: Right mul of symbol} satisfies the covariance property 
\begin{equation*}
    U(-x,a)^*\Gamma_{\mathcal{D}T\mathcal{D}}(f)U(-x,a)=\Gamma(R_{(x,a)^{-1}}f).
\end{equation*}
The following result, which is a modification of the remark given at the end of \cite{kiukas2006}, shows a remarkable converse to these observations. 

\begin{theorem} \label{thm:kiukas}
    Let $\Gamma:L^\infty(\aff)\to \mathcal{L}(L^2(\mathbb{R}_+))$ be a linear map satisfying
    \begin{enumerate}
        \item $\Gamma$ sends positive functions to positive operators,
        \item $\Gamma(1)=I_{L^2(\mathbb{R}_+)}$,
        \item $\Gamma$ is continuous from the weak* topology on $L^\infty(\aff)$ (as the dual space of $L^1_r(\aff)$) to the weak* topology on $\mathcal{L}(L^2(\mathbb{R}_+))$,
        \item  $U(-x,a)^*\Gamma(f)U(-x,a)=\Gamma(R_{(x,a)^{-1}}f)$.
    \end{enumerate}
    Then there exists a unique positive trace-class operator $T$ with $\tr(T)=1$ such that \begin{equation*} 
        \Gamma(f)=f\star_\aff \mathcal{D}T\mathcal{D}.
    \end{equation*}
\end{theorem}

\begin{proof}
	The map $\Gamma\mapsto \Gamma_l$ where $\Gamma_l(f)=\Gamma(\check{f})$ is a bijection from maps $\Gamma$ satisfying the four assumptions to maps $\Gamma_l$ satisfying 
	\begin{enumerate}[label=\roman*)]
			\item $\Gamma_l$ sends positive functions to positive operators,
			\item $\Gamma_l(1)=I_{L^2(\mathbb{R}_+)}$,
			\item $\Gamma_l$ is continuous from the weak* topology on $L^\infty(\aff)$ (as the dual space of $L^1_l(\aff)$) to the weak* topology on $\mathcal{L}(L^2(\mathbb{R}_+))$,
			\item $U(-x,a)^*\Gamma_l(f)U(-x,a)=\Gamma_l(L_{(x,a)^{-1}}f)$.
		\end{enumerate}
	The remark in \cite{kiukas2006} applied to $G=\aff$ and $U(-x,a)$ says that if a map $\Gamma_l$ satisfies i)-iv)
	then it must be given for $\psi,\phi \in \mathrm{dom}(\mathcal{D})$ by 
	\begin{equation*}
	  \langle \Gamma_l(f)\psi,\phi \rangle_{L^2(\mathbb{R}_+)}  =\int_\aff f(x,a) \langle U(-x,a)TU(-x,a)^* \mathcal{D}\psi,\mathcal{D}\phi \rangle_{L^2(\mathbb{R}_+)} \, \frac{dx\, da}{a},
	\end{equation*}
	for some trace-class operator $T$ as in the theorem.
	The relation \eqref{eq:comrelD} gives that
	\begin{align*}
	  \langle \Gamma_l(f)\psi,\phi \rangle_{L^2(\mathbb{R}_+)}  &=\int_\aff f(x,a) \langle U(-x,a)\mathcal{D}T\mathcal{D}U(-x,a)^* \psi,\phi \rangle_{L^2(\mathbb{R}_+)} \, \frac{dx\, da}{a^2} \\
	  &= \int_\aff \check{f}(x,a) \langle U(-x,a)^*\mathcal{D}T\mathcal{D}U(-x,a) \psi,\phi \rangle_{L^2(\mathbb{R}_+)}\, \frac{dx\,da}{a}.
	\end{align*}
	Hence $\Gamma_l(f)=\check{f}\star_\aff \mathcal{D}T\mathcal{D}$ and the result follows.
	\end{proof}

\subsubsection*{Quantization using admissible trace-class operators}
As we have mentioned, the properties of the quantization map $\Gamma(f)=f\star_\aff S$ depend on the properties of $S$. From Lemma \ref{lem:boundedquantization} we know that if $S$ is admissible, i.e.\ we can write $S=\mathcal{D}T\mathcal{D}$ for some trace-class operator $T$, then $\Gamma_{S}:L^\infty(\aff)\to \mathcal{L}(L^2(\mathbb{R}_+))$ is bounded. If we further assume that $S$ is a trace-class operator, then Proposition \ref{prop:interpolation} shows that $\Gamma_{S}$ is bounded from $L^p_r(\aff)$ to $\mathcal{S}_p$ for all $1\leq p \leq \infty$. In this sense, the ideal class of covariant integral quantizations $\Gamma_S$ are  those given by \textit{admissible trace-class operators.}  

\begin{example}
If $\varphi\in L^2(\mathbb{R}_+)$ is an admissible function, then $\varphi \otimes \varphi$ is an admissible operator. 
 The resulting quantization $\Gamma_{\varphi \otimes \varphi}$ is then a special case of the quantization procedures introduced by Berezin \cite{Berezin:1972}; Berezin calls $f$ the \textit{contravariant symbol} of $\Gamma_{\varphi \otimes \varphi}(f)$. In this sense, the quantization procedures $\Gamma_S$ for admissible $S$ generalize Berezin's contravariant symbols.
\end{example}

\subsubsection*{Relation to the Conventions of Gazeau and Murenzi}

Gazeau and Murenzi \cite{gazeau2016covariant} work with another parametrization of the affine group, namely $\Pi_{+} \coloneqq \mathbb{R}_+\times \mathbb{R}$ where the group operation between $(q_{1}, p_{1}), (q_{2}, p_{2}) \in \Pi_{+}$ is given by \[(q_1,p_1)\cdot (q_2,p_2) \coloneqq (q_{1}q_{2}, p_{2}/q_{1} + p_{1}).\] There is a unitary representation $U_{G}:\Pi_{+} \to \mathcal{U}(L^{2}(\mathbb{R}_{+}, dr))$ given by 
\[
	U_G(q,p)\psi(r)=\sqrt{\frac{1}{q}}e^{ipr}\psi(r/q)=\sqrt{\frac{1}{q}}U(p/2\pi,1/q)\psi(r).
\]

Given a function $\tilde{f}$ on $\Pi_{+}$ and an operator $S$ on $L^2(\mathbb{R}_+, dr)$, Gazeau and Murenzi define (note that the adjoint is now with respect to $L^2(\mathbb{R}_+,dr)$, not $L^2(\mathbb{R}_+)$)
\begin{equation*}
A^S_{\tilde{f}}\coloneqq\frac{1}{C_S}\int_{-\infty}^\infty \int_0^\infty \tilde{f}(q,p)U_G(q,p)S U_G(q,p)^* \, dq \, dp,
\end{equation*}
where we assume that $S$ satisfies
\begin{equation*}
\int_{-\infty}^\infty \int_0^\infty U_G(q,p)S U_G(q,p)^* \, dq \, dp = C_S \cdot I_{L^2(\mathbb{R}_+,dr)}.
\end{equation*}
The next proposition is straightforward and shows that Gazeau and Murenzi's framework is easily related to our affine operator convolutions.

\begin{proposition}
Let $S$ be an operator on $L^2(\mathbb{R}_+,dr)$. Then $\mathcal{D}^{-1}S\mathcal{D}$ is an operator on $L^2(\mathbb{R}_+,r^{-1}dr)$ and 
\[
	\mathcal{D} A^S_{\tilde{f}} \mathcal{D}^{-1}=\frac{2\pi}{C_S} f\star_\aff (\mathcal{D}S\mathcal{D}^{-1}),
\]
where $f(x,a)=\tilde{f}(a,\frac{2\pi x}{a})$ for $(x,a) \in \mathrm{Aff}$.
\end{proposition}

\subsection{Affine Cohen Class Distributions}
\label{sec: Affine Cohen Class Distributions}

The cross-Wigner distribution $W(\psi,\phi)$ of $\psi,\phi \in L^2(\mathbb{R}^n)$ is known to have certain undesirable properties. A typical example is that one would like to interpret $W(\psi,\phi)$ as a probability distribution, but $W(\psi,\phi)$ is seldom a non-negative function as shown by Hudson in \cite{hudson1974wigner}. To remedy this, Cohen introduced in \cite{Cohen:1966} a new class of time-frequency distributions $Q_f$ given by
\begin{equation} \label{eq:cohenheis}
	Q_f(\psi,\phi) \coloneqq W(\psi,\phi) \ast f,
\end{equation}
where $f$ is a tempered distribution on $\mathbb{R}^{2n}$. In light of our setup, it is natural to investigate the affine analogue of the Cohen class.

\begin{definition} \label{def:cohen}
We say that a bilinear map $Q:L^2(\mathbb{R}_+)\times L^2(\mathbb{R}_+)\to L^\infty(\aff)$ belongs to the \textit{affine Cohen class} if $Q = Q_S$ for some $S\in \mathcal{L}(L^2(\mathbb{R}_+))$, where
\begin{equation*}
	Q_S(\psi,\phi)(x,a) \coloneqq (\psi \otimes \phi) \star_\aff S(x,a)
	= \langle SU(-x,a)\psi,U(-x,a)\phi \rangle_{L^2(\mathbb{R}_+)}.
\end{equation*}
We will write $Q_S(\psi)\coloneqq Q_S(\psi,\psi).$
\end{definition}
By Proposition \ref{prop:convolutionasweyl} we get for $S = A_{f}$ that
\begin{equation} \label{eq:cohenaff_opp}
	Q_{S}(\psi, \phi) = W_{\textrm{Aff}}^{\psi, \phi} \ast_\aff \check{f},
\end{equation}
which shows that our definition of the affine Cohen class is a natural analogue of \eqref{eq:cohenheis}.
It is straightforward to verify that $Q_{S}(\psi, \phi)$ is a continuous function on $\aff$ for all $\psi, \phi \in L^{2}(\mathbb{R}_{+})$ and $S\in \mathcal{L}(L^2(\mathbb{R}_+))$.
Since the affine Cohen class is defined in terms of the operator convolutions, we get some simple properties: The statements 1 and 2 in Proposition \ref{prop:basic_properties_affine_cohen} follow from Proposition \ref{prop:interpolation} and Corollary \ref{cor:admissible_condition}. Statement 3 is a simple calculation and the last statement follows from a short polarization argument.

\begin{proposition}
\label{prop:basic_properties_affine_cohen}
	Let $S\in \mathcal{L}(L^2(\mathbb{R}_+))$. Then for $\psi, \phi \in L^2(\mathbb{R}_+)$ we have the following properties:
	\begin{enumerate}
		\item The function $Q_S(\psi,\phi)$ satisfies \[\|Q_S(\psi,\phi)\|_{L^{\infty}(\aff)}\leq \|S\|_{\mathcal{L}(L^{2}(\mathbb{R}_{+}))} \|\psi\|_{L^{2}(\mathbb{R}_{+})} \|\phi\|_{L^{2}(\mathbb{R}_{+})}.\]
		\item If $S$ is admissible, then $Q_S(\psi,\phi)\in L^1_r(\aff)$ and \[\int_\aff Q_S(\psi,\phi)(x,a) \, \frac{dx\,da}{a}=\langle \psi,\phi \rangle_{L^{2}(\mathbb{R}_{+})} \tr(\mathcal{D}^{-1}S\mathcal{D}^{-1}).\]
		\item We have the covariance property 
		\begin{equation}
		\label{eq:covariance_affine_cohen}
		Q_S(U(-x,a)\psi,U(-x,a)\phi)(y,b)=Q_S(\psi,\phi)((y,b)\cdot (x,a))  
		\end{equation}
		for all $(x,a), (y,b)\in \aff$.
		\item The function $Q_S(\psi,\psi)$ is (real-valued) positive for all $\psi \in L^2(\mathbb{R}_+)$ if and only if $S$ is (self-adjoint) positive. 
	\end{enumerate}
\end{proposition}

\begin{example}
\hfill
\begin{enumerate}
    \item For $\psi, \phi \in L^2(\mathbb{R}_+)$ we have \[Q_{\phi\otimes \phi}(\psi)(x,a)=|\langle \psi, U(-x,a)^* \phi \rangle_{L^2(\mathbb{R}_+)}|^2,\] which by Corollary \ref{cor:equallity_of_scalogram_and_other_stuff} is simply a Fourier transform away from being a scalogram. 
    \item If we relax the requirement that $S$ is bounded in Definition \ref{def:cohen}, then it follows from Theorem \ref{thm:quantization_through_convolution} that \[Q_{P_\aff}(\psi)=W_\aff^{\psi}\] for $\psi \in \mathscr{S}(\mathbb{R}_+).$ Hence the affine Wigner distribution can be represented as a (generalized) affine Cohen class operator. If we define an alternative affine Weyl quantization using an operator $P$ as in Section \ref{sec: Other Covariant Integral Quantizations}, then it is clear that $Q_P$ gives an alternative Wigner function. See \cite[Sec.~7.2]{gazeau2016covariant} for the case of $P\psi(r)=\psi(1/r)$.
\end{enumerate}
\end{example}

 The covariance property \eqref{eq:covariance_affine_cohen} and some rather weak continuity conditions completely characterize the affine Cohen class, as is shown in the following result.

\begin{proposition}
	Let $Q:L^2(\mathbb{R}_+)\times L^2(\mathbb{R}_+)\to L^\infty(\aff)$ be a bilinear map. Assume that for all $\psi,\phi \in L^2(\mathbb{R}_+)$ we know that $Q(\psi,\phi)$ is a continuous function on $\aff$ that satisfies \eqref{eq:covariance_affine_cohen} and the estimate \[|Q(\psi,\phi)(0,1)|\lesssim \|\psi\|_{L^{2}(\mathbb{R}_{+})} \|\phi\|_{L^{2}(\mathbb{R}_{+})}.\]
	Then there exists a unique $S\in \mathcal{L}(L^2(\mathbb{R}_+))$ such that $Q=Q_S$.
\end{proposition}

\begin{proof}
	By assumption, the map $(\psi,\phi)\mapsto Q(\psi,\phi)(0,1)$ is a bounded bilinear form. Hence there exists a unique bounded operator $S$ such that 
	\begin{equation*}
		\langle S\psi,\phi \rangle_{L^2(\mathbb{R}_+)}=Q(\psi,\phi)(0,1).
	\end{equation*}
	To see that $Q=Q_S$, note that we have
	\begin{align*}
		Q(\psi,\phi)(x,a)&=Q(U(-x,a)\psi,U(-x,a)\phi)(0,1) \\
		&= \langle SU(-x,a)\psi,U(-x,a)\phi \rangle_{L^2(\mathbb{R}_+)} \\
		&= Q_S(\psi,\phi)(x,a). \qedhere
	\end{align*}
\end{proof}

At this point we have seen that operators $S$ define a quantization procedure $\Gamma_S(f)= f\star_\aff S$ as in Section \ref{sec: Other Covariant Integral Quantizations}, and an affine Cohen class distribution $Q_S$. The connection between these concepts is provided by the next proposition.

\begin{proposition}
\label{prop:eigenvalue_maximization_statement}
	Let $S$ be a positive, compact operator on $L^2(\mathbb{R}_+)$ and let $f\in L^1_r(\aff)$ be a positive function. Then $f\star_\aff S$ is a positive, compact operator. Denote by $\{\lambda_n\}_{n=1}^\infty$ its eigenvalues in non-increasing order with associated orthogonal eigenvectors $\{\phi_n\}_{n=1}^\infty$. Then 
	\begin{equation*}
		\lambda_n = \max_{ \|\psi\| = 1} \left\{\int_\aff f(x,a) Q_S(\psi, \psi)(x,a) \, \frac{dx\, da}{a} \, : \,  \psi\perp \phi_k \text{ for } k=1,\dots , n-1 \right\}.
	\end{equation*} 
\end{proposition}

\begin{proof}
	The integral defining $f\star_\aff S$ is a Bochner integral of compact operators converging in the operator norm, hence it defines a compact operator. It is straightforward to check that $f\star_\aff S$ is also a positive operator. Furthermore, for $\psi \in L^{2}(\mathbb{R}_{+})$ we have
	\begin{align*}
		\langle f\star_\aff S \psi,\psi \rangle_{L^2(\mathbb{R}_+)}&=\int_\aff f(x,a) \langle SU(-x,a)\psi,U(-x,a)\psi \rangle_{L^2(\mathbb{R}_+)} \, \frac{dx\,da}{a} \\
		&= \int_\aff f(x,a) Q_S(\psi, \psi)(x,a) \, \frac{dx\,da}{a}. 
	\end{align*}
	The result therefore follows from Courant's minimax theorem \cite[Thm.~28.4]{lax2002book}.
\end{proof}
\begin{example} \label{first_localization_example_mod}
Let us consider a localization operator $\chi_\Omega \star_\aff \varphi \otimes \varphi$  for $\varphi \in L^{2}(\mathbb{R}_{+})$ and a compact subset $\Omega\subset \aff$. The first eigenfunction $\phi_0$ of this operator maximizes the quantity
\begin{equation*}
     \langle \chi_\Omega \star_\aff (\varphi\otimes \varphi) \phi_0,\phi_0 \rangle_{L^{2}(\mathbb{R}_{+})}=
      \int_\Omega |\langle \varphi_0,U(-x,a)^* \varphi \rangle_{L^{2}(\mathbb{R}_{+})}|^2 \, \frac{dx \, da}{a}.
\end{equation*}
Hence in this sense, the eigenfunctions are the best localized functions in $\Omega$, which explains the terminology of localization operators.
\end{example}

\subsubsection{Relation to the Affine Quadratic Time-Frequency Representations}

The signal processing literature contains a wealth of two-dimensional representations of signals. Among them we find the \textit{affine class of quadratic time-frequency representations}, see \cite{papandreou1998quadratic}. A member of the affine class of quadratic time-frequency representations is a map sending functions $\psi$ on $\mathbb{R}$ to a function $Q^A_\Phi(\psi)$ on $\mathbb{R}^2$ given by 
\begin{equation*}
    Q^A_\Phi(\psi)(x,a) = \frac{1}{a}\int_{-\infty}^\infty \int_{-\infty}^\infty \Phi(t/a,s/a) e^{2\pi i x(t-s)} \psi(t)\overline{\psi(s)} \, dt\,ds
\end{equation*}
for some kernel function $\Phi$ on $\mathbb{R}^2$. There are clearly a few differences between our setup and the affine class of quadratic time-frequency representations. The domain of the affine class consists of  functions on $\mathbb{R}$, whereas the affine Cohen class acts on functions on $\mathbb{R}_+$. For a function $\psi$ on $\mathbb{R}_+$ we therefore define
\begin{equation*}
    \psi_0(t)=\begin{cases} \psi(t)\quad  t>0 \\ 0 \quad \text{ otherwise.} \end{cases}
\end{equation*}
Finally, we recall that a function $K_S$ defined on $\mathbb{R}_+\times \mathbb{R}_+$ defines an integral operator $S$ with respect to the measure $\frac{dt}{t}$ by 
\begin{equation*}
    S\psi(s)=\int_0^\infty K_S(s,t)\psi(t) \, \frac{dt}{t}.
\end{equation*}
The following formal result is straightforward to verify. 
\begin{proposition} \label{prop:affineclass}
    Let $S$ be an integral operator with kernel $K_S$ and define 
    \begin{equation*}
        \Phi_S(s,t)=\begin{cases} \frac{K_S(t,s)}{\sqrt{st}} \quad & \text{ if } s,t > 0, \\ 0 \quad & \text{ otherwise.} \end{cases}
    \end{equation*}
    For $x>0$ and $\psi$ defined on $\mathbb{R}_+$, we have
    \begin{equation*}
        Q_S(\mathcal{D}\psi,\mathcal{D}\psi)(x,a)=Q^A_{\Phi_S}(\psi_0)(-x/a,a).
    \end{equation*}
\end{proposition}

\bibliographystyle{abbrv}
\bibliography{main}

\Addresses

\end{document}